\documentclass[10pt]{amsart} 
\usepackage{amsmath} 
\usepackage{amssymb} 
\usepackage{amsthm} 
\usepackage{mathtools}
\usepackage{xcolor}
\usepackage{comment}
\usepackage{graphicx}
\usepackage{tikz}
\usepackage{bbold} 
\usepackage{upgreek} 
\usepackage{physics}
\usepackage[colorlinks]{hyperref}

\numberwithin{equation}{section}

\newcommand{\sub}{\subseteq}
\newcommand{\Z}{\mathbb{Z}}
\newcommand{\R}{\mathbb{R}}
\newcommand{\C}{\mathbb{C}}

\newcommand{\N}{\mathbb{N}}
\newcommand{\eps}{\varepsilon}

\newcommand{\supp}{\mathrm{supp}}

\numberwithin{chap}{section}
\newtheorem{thm}{Theorem}
\numberwithin{thm}{section}
\newtheorem{conj}[thm]{Conjecture}
\newtheorem{prop}[thm]{Proposition}
\newtheorem{defn}[thm]{Definition}
\newtheorem{lem}[thm]{Lemma}
\newtheorem{claim}[thm]{Claim}
\newtheorem{cor}[thm]{Corollary}
\newtheorem{example}[thm]{Example}
\DeclarePairedDelimiterX{\inp}[2]{\langle}{\rangle}{#1, #2}

\makeatletter
\let\oldnorm\norm
\def\norm{\@ifstar{\oldnorm}{\oldnorm*}}

\makeatother

\begin{document}

\pagestyle{myheadings} \thispagestyle{empty} \markright{}
\title{Study guide for ``On restricted projections to planes in $\R^3$''}

\author{Tainara Borges, Siddharth Mulherkar and Tongou Yang}
\address[Tainara Borges]{Department of Mathematics, Brown University\\
Providence, RI 02912 United States
}
\email{tainara\textunderscore gobetti\textunderscore borges@brown.edu}

\address[Siddharth Mulherkar]{Department of Mathematics, University of California, Los Angeles\\
CA 90095, United States 
}
\email{sidmulherkar@math.ucla.edu}

\address[Tongou Yang]{Department of Mathematics, University of California, Los Angeles\\
CA 90095, United States 
}
\email{tongouyang@math.ucla.edu}

\date{}

\begin{abstract}
    This article is a study guide for ``{\it On restricted projections to planes in $\R^3$}" by Gan, Guo, Guth, Harris, Maldague and Wang \cite{GGGHMW2022}. We first present the main problems and preliminaries related to restricted projections in $\R^3$. Then we introduce the high-low method and decoupling, which are the two central and novel ideas in their proofs. We hope to provide as many details as possible so that this study guide is self-contained, with the only exception of the Bourgain-Demeter decoupling inequality for curves in the appendix.
\end{abstract}

\maketitle


\section{Introduction}

\subsection{Marstrand's and Mattila's projection theorems}
\subsubsection{Dimension preserving property in \texorpdfstring{$\R^2$}{Lg}}
In $\R^2$, given a direction $u\in \mathbb S^1$ on the unit circle, consider the orthogonal projection $p_u(x):=x\cdot u$ for $x\in \R^2$. For each Borel set $A\sub \R^2$, we consider the size of the image $p_u(A)\sub \R$ relative to the size of $A$. In particular, for each $u$ we ask if it is {\it dimension preserving}, that is, whether we have\footnote{Here and throughout this article, $\dim $ will always denote the Hausdorff dimension.}
\begin{equation}\label{eqn_dimension_preserving_2D}
    \dim p_u(A)=\min\{1,\dim A\}.
\end{equation}
Note that $\leq$ in \eqref{eqn_dimension_preserving_2D} is trivial since the range is a subset of $\R$ and a projection cannot increase dimensions. The celebrated Marstrand's projection theorem \cite{Marstrand} states that $\sigma^1$-a.e. $u\in \mathbb S^1$ is dimension preserving for $A$, where $\sigma^1$ denotes the standard surface measure on $\mathbb S^1$. For this reason, we say a direction $u$ is {\it exceptional} if it is not dimension preserving, that is, we have $<$ holds in \eqref{eqn_dimension_preserving_2D}. Lastly, we remark that the set of exceptional directions naturally depend on the choice of $A$; for instance, if $A$ is a straight line with direction $v\in \mathbb S^1$, then $u$ is dimension preserving for $A$ if and only if $u\cdot v\ne 0$.

\subsubsection{Exceptional set estimate in \texorpdfstring{$\R^2$}{Lg}}\label{sec_exceptional}
In many cases, we can prove stronger results than dimension preserving properties, known as exceptional set estimates. Given a Borel $A\sub \R^2$. For each $0< s<\min\{1,\dim A\}$, we consider the set of exceptional directions
\begin{equation}\label{eqn_defn_E_s}
    E_s:=\{u\in \mathbb S^1:\dim p_u(A)<s\}.
\end{equation}
If we can show that $\dim E_s<1$, then Marstrand's theorem follows as a corollary. In fact, Kaufman \cite{Kaufman} and Falconer \cite{Falconer} established the bound $\dim E_s\le s$ and $\dim E_s\le \max\{0,1+s-\dim A\}$, respectively. For $\dim A<1$ Kaufman's estimate is better, while for $\dim A>1$ Falconer's estimate is better. Recently, Ren and Wang \cite{RenWang} showed that we always have $\dim E_s\le \max\{0,2s-\dim A\}$, which is better than both Kaufman's and Falconer's results. Moreover, a simple example given by Kaufman and Mattila \cite{KaufmanMattila} shows that this is sharp.

\subsubsection{Dimension preserving property in higher dimensions}
Now we generalize to $\R^d$, $d \ge 2$. For each $1\le k\le d-1$, consider the set $G(d,k)$ of all $k$-dimensional subspaces equipped with the standard Haar measure $\sigma(d,k)$. Mattila \cite{Mattila_Marstrand} generalized Marstrand's theorem \cite{Marstrand} to higher dimensions. More precisely, given $1\le k\le d-1$ and a subspace $S\in G(d,k)$, consider the projection $p_S$ from $\R^d$ onto $S$. Then for every Borel set $A\sub \R^d$, we have
\begin{equation}
    \dim p_S(A)=\min\{k,\dim A\},\quad \text{for $\sigma(d,k)$-a.e. $S\in G(d,k)$}.
\end{equation}
One may also consider exceptional estimates in higher dimensions; however, obtaining a sharp exceptional estimates will be much more difficult than in $\R^2$, and so we will not discuss this topic here.

\subsection{Restricted projection theorems in \texorpdfstring{$\R^3$}{Lg}}
\subsubsection{Background}
Now we consider $d=3$ and $k=1,2$. For $k=1$ we are considering projection to lines, namely, for each $e\in \mathbb S^2$ we consider the function $p_e(x):=x\cdot e$. For $k=2$ we are considering projection to planes, namely, for each $e\in \mathbb S^2$ we consider the function $\pi_e(x):=x-(x\cdot e)e$.

Denote by $\sigma^2$ the standard surface measure on $\mathbb S^2$. Mattila's theorem implies that given any Borel $A\sub \R^3$, for $\sigma^2$-a.e. $e\in \mathbb S^2$ we have both
\begin{equation}\label{eqn_dimension_preserving_3D}
    \dim p_e(A)=\min\{1,\dim A\},\quad \dim \pi_e(A)=\min\{2,\dim A\}.
\end{equation}
However, this provides no information if we would like to study the projections of $A$ along a subset $N\sub \mathbb S^2$ with zero $\sigma^2$ measure, most typically when $N$ is the image of a continuous curve on $\mathbb S^2$. Nevertheless, this problem is still interesting, and we may ask the following main question: if $\gamma(\theta):[0,1]\to \mathbb S^2$ is a continuous curve, and $A\sub \R^3$ is a Borel set, is it true that for $\mathcal L^1$-a.e. $\theta\in [0,1]$ we have either equality below:
\begin{equation}\label{eqn_dimension_preserving_3D_restricted}
    \dim p_{\gamma(\theta)}(A)=\min\{1,\dim A\},\quad \dim \pi_{\gamma(\theta)}(A)=\min\{2,\dim A\}?
\end{equation}
The simple answer is no. In fact, if we take the image of $\gamma$ to be the equator and $A$ to be the straight line passing through the north and south poles, then the former equality of \eqref{eqn_dimension_preserving_3D_restricted} fails to hold for any $\theta\in [0,1]$. Similarly, if we take the image of $\gamma$ to be the equator and $A$ a $2$-dimensional cube contained in the Equator plane, say $A=[-1,1]^2\times \{0\}$ then the latter equality of \eqref{eqn_dimension_preserving_3D_restricted} fails to hold for any $\theta\in [0,1]$.

However, this triviality can be easily avoided once we require that the image of $\gamma$ ``escapes from being a great circle". More precisely, we require that $\gamma$ be $C^2$, and that the
following non-degeneracy condition hold:
\begin{equation}\label{eqn_non_degenaracy}
        \det (\gamma,\gamma',\gamma'')(\theta)\ne 0,\quad \forall \,\theta\in [0,1].
\end{equation}
Geometrically, this means that $\gamma$ is the velocity of a curve with nonzero torsion. A model case of $\gamma$ is a non-great circle which generates the standard light cone, that is,
\begin{equation}\label{eqn_light_cone}
\gamma(\theta)=\frac 1 {\sqrt 2}\left(\cos \sqrt 2 \,\theta, \sin \sqrt 2 \,\theta,1\right).
\end{equation}
The factor $\sqrt 2$ here is to normalize so that $|\gamma'|=1$ as well. We remark here that an interesting property of the standard light cone is that $\gamma$ and $\gamma\times \gamma'$ are two different parametrizations of the same circle, since as one can easily check $(\gamma\times \gamma')(\theta)=\frac{1}{\sqrt{2}}(-\cos\sqrt{2}\theta,-\sin\sqrt{2}\theta,1))$. This fails for a general non-degenerate curve $\gamma$.

\subsubsection{Literature review}
With the non-degeneracy condition \eqref{eqn_non_degenaracy}, F\"assler and Orponen \cite{FasslerOrponen} put forward the following conjecture:
\begin{conj}\label{conj_FasslerOrponen}
    Let $\gamma:[0,1]\to \mathbb S^2$ be $C^2$ and satisfies the non-degeneracy condition \eqref{eqn_non_degenaracy}. Then for every Borel set $A\sub \R^3$, we have    \begin{equation}\label{eqn_dimension_preserving_3D_restricted_conj}
    \dim p_{\gamma(\theta)}(A)=\min\{1,\dim A\},\quad \dim \pi_{\gamma(\theta)}(A)=\min\{2,\dim A\}
\end{equation}
both hold for $\mathcal L^1$-a.e. $\theta \in [0,1]$.
\end{conj}
In \cite{FasslerOrponen}, the authors mentioned that the former equality in \eqref{eqn_dimension_preserving_3D_restricted_conj} is easy when $\dim A\le 1/2$, and the latter equality in \eqref{eqn_dimension_preserving_3D_restricted_conj} is easy when $\dim A\le 1$ (see \cite{JJLL}). In the same article, they were able to prove something a little stronger. More precisely, if $\dim A>1/2$, they showed there is some $\sigma>1/2$ depending on $\dim A$ such that $\dim p_{\gamma(\theta)}>\sigma$ for a.e. $\theta$. If $\dim A>1$, then they showed there is some $\sigma>1$ depending on $\dim A$ such that $\dim \pi_{\gamma(\theta)}>\sigma$ for a.e. $\theta$.

Since \cite{FasslerOrponen}, a dozens of partial results have been established towards the main conjecture, see for instance \cite{JarvenpaaJarvenpaa,OberlinOberlin,OrponenVenieri,KOV,Harris2022}. The most significant result among them is \cite{KOV}, which is the first to establish the former equality in \eqref{eqn_dimension_preserving_3D_restricted_conj} when $\gamma$ is given by a non-great circle such as \eqref{eqn_light_cone}. Their main technique is to use Wolff's \cite{Wolff2000} incidence estimates on tangency of circles, which relies in turn on the algebraic property of the circle. Therefore, one major downside of this method is that it fails to work in our case when $\gamma$ is merely $C^2$. We also remark here that \cite{KOV} established the following exceptional set estimate: for every $0<s<\min\{1,\dim A\}$, we have
\begin{equation}\label{eqn_KOV_exceptional_set}
    \dim E_s\le \frac {s+1}2,
\end{equation}
where the exceptional set $E_s$ is given by
\begin{equation}\label{eqn_exceptional set}
    E_s:=\{\theta\in [0,1]:\dim p_{\gamma(\theta)}(A)<s\}.
\end{equation}
The full proof of Conjecture \ref{conj_FasslerOrponen} is given by Pramanik-Yang-Zahl \cite{PYZ}, Gan-Guth-Maldague \cite{GanGuthMaldague} and Gan-Guo-Guth-Harris-Maldague-Wang \cite{GGGHMW2022}. Among them, \cite{PYZ} and \cite{GanGuthMaldague} independently proved the former equality of \eqref{conj_FasslerOrponen}, using very different methods, and obtaining different exceptional set estimates. More precisely, \cite{PYZ} proved a Kaufman type exceptional set estimate, namely,
\begin{equation}
    \dim E_s\le s,\quad \forall\,\, 0<s<\min\{1,\dim A\}.
\end{equation}
In contrast, \cite{GanGuthMaldague} proved a Falconer type exceptional set estimate, namely,
\begin{equation}
    \dim E_s\le \max\{0,1+s-\dim A\},\quad \forall \,\,0<s<\min\{1,\dim A\}.
\end{equation}
As mentioned right after \eqref{eqn_defn_E_s}, neither estimate is strictly better than the other. 

We remark that \cite{harris2023length} also gives an analogue of Theorem \ref{thm_thm8} below in the case of projection to lines in $\R^3$.

\cite{GGGHMW2022} proved the latter equality of \eqref{conj_FasslerOrponen} via a Falconer type exceptional set estimate (see Theorem \ref{firstmainthm} below).
This article is devoted to the understanding of \cite{GGGHMW2022}.

\subsection{Main results}
We now focus solely on the paper \cite{GGGHMW2022}.

Let $\gamma:[0,1]\to \mathbb S^2$ be a $C^2$ curve. We assume it satisfies the non-degeneracy condition in (\ref{eqn_non_degenaracy}). We may also assume without loss of generality that $|\gamma'|=1$, whence
\begin{equation}
    \mathbf e_1:=\gamma,\quad \mathbf e_2:=\gamma',\quad \mathbf e_3:=\gamma\times \gamma'
\end{equation}
form an orthonormal basis for $\R^3$ known as a Frenet frame.

\begin{thm}[Theorem 1 in \cite{GGGHMW2022}]\label{firstmainthm}
    Suppose $A\sub \R^3$ is a Borel set of Hausdorff dimension $\alpha\in (0,3]$. For $0< s<\min\{2,\alpha\}$, define the exceptional set 
    $$E_s=\{\theta \in [0,1]\colon \textnormal{dim}(\pi_{\theta}(A))<s\}$$
    Then we have
    $$\textnormal{dim}(E_s)\leq \max\{1+s-\alpha,0\}$$
\end{thm}
The proof of Theorem \ref{firstmainthm} will be given in Section \ref{discretethm1}. 

Theorem \ref{firstmainthm} gives the following corollary, which solves Conjecture \ref{conj_FasslerOrponen}.
\begin{cor}[Corollary 1 in \cite{GGGHMW2022}]
    Suppose $A\sub \R^3$ is a Borel set of Hausdorff dimension $\alpha$. Then we have 
    $$\textnormal{dim}(\pi_{\theta}(A))=\min\{2,\alpha\},\,\text{for a.e. }\theta\in[0,1]$$
\end{cor}
\begin{proof}
    There are two possible cases: $\textnormal{dim}(A)=\alpha\leq 2$ or $\textnormal{dim}(A)=\alpha>2$. In the first case, take $s_{n}=\alpha-1/n$. Theorem \ref{firstmainthm} gives 
    $\textnormal{dim}(E_{s_{n}})\leq 1-1/n$. In particular, 
    $$\mathcal{L}^1(E_{s_n})=0$$

    Since 
    $$E_{\alpha}=\cup_{n=1}^{\infty}E_{s_n}$$
    it also follows that $ \mathcal{L}^1(E_{\alpha})=0$.
    The case $\textnormal{dim}(A)>2$ is similar, by taking $s_n=2-1/n$ and $E_2=\cup_{n=1}^{\infty} E_{s_n}$.
\end{proof}

\begin{thm}[Theorem 2 in \cite{GGGHMW2022}]\label{thm_thm2}
    Suppose $A \sub \R^3$ is a Borel set of Hausdorff dimension greater than $2$. Then $\mathcal H^2(\pi_\theta(A))>0$ for a.e. $\theta\in [0,1]$. 
\end{thm}

\subsection{Reduction of Theorem 2 to Theorem 8}

We first state the following theorem, from which Theorem \ref{thm_thm2} follows as a corollary.
\begin{thm}[Theorem 8 in \cite{GGGHMW2022}]\label{thm_thm8}
    If $a\in (2,3]$ and $\mu$ is a compactly supported Borel measure on $\R^3$ such that
     \begin{equation}\label{eqn_Frostman_mu_intro}
        \rho_a(\mu):=\sup_{x\in \R^3,r\in (0,1)}r^{-a}\mu(B^3(x,r))<\infty,
    \end{equation}
    then $\pi_{\theta\#}\mu$ is absolutely continuous with respect to $\mathcal H^2$, for a.e. $\theta\in [0,1]$. Here and throughout this article, $\pi_{\theta\#} \mu$ denotes the standard push-forward measure, namely, for every Borel set $A\sub \pi_\theta(\R^3)$,
    \begin{equation}\label{eqn_push_forward_defn}
        \pi_{\theta\#} \mu(A):=\mu(\pi_\theta^{-1}(A)).
    \end{equation}
\end{thm}
The proof of Theorem \ref{thm_thm8} will be given in Section \ref{sec_proof_thm8}. We also remark that since $\mu$ is compactly supported, we may actually regard $\mathcal H^2$ as supported on the unit ball of $\pi_\theta(\R^3)$, for all $\theta\in[0,1]$. Also, Theorem \ref{thm_thm8} actually gives that $\pi_{\theta \#}\mu$ can be identified with an integrable function on $\pi_\theta(\R^3)$ for a.e. $\theta$.

We now prove Theorem \ref{thm_thm2} assuming Theorem \ref{thm_thm8}. For this purpose, we recall the preliminary concept of Frostman measure.
\begin{defn}[Frostman measure]\label{deffrostmeasure} Let $A\subset \R^d$ be a Borel set. A finite nonzero measure $\mu$ supported in $A$ is said to be an $a$-Frostman measure in $A$ if $\rho_a(\mu)<\infty$ (see \eqref{eqn_Frostman_mu_intro} above).
\end{defn}
We state Frostman's Lemma below, where $\mathcal{H}^{a}$ stands for the $a$-dimensional Hausdorff measure in $\R^d$.

\begin{thm}[Frostman's lemma, Theorem 2.7 in \cite{Mattilabook}]\label{thm_Frostman_lemma}
    Let $A\subset \R^d$ be a Borel set. Then $A$ admits an $a$-Frostman measure if and only if $\mathcal{H}^{a}(A)\in (0,\infty]$.
\end{thm}
\begin{proof}[Proof of Theorem \ref{thm_thm2} assuming Theorem \ref{thm_thm8}]
    We may of course assume $A\sub [-1,1]^3$. Since $\dim A\in (2,3]$, we have $\mathcal H^a(A)>0$ where $a:=\frac {\dim A+2}2\in (2,\dim A)$ (indeed, since $\mathcal{H}^{a}(A)=\infty$ for such $a$). By Theorem \ref{thm_Frostman_lemma}, we can find an $a$-Frostman measure $\mu$ supported in $A$. By Theorem \ref{thm_thm8}, we have $\pi_{\theta\#}\mu$ is absolutely continuous with respect to $\mathcal H^2$ for a.e. $\theta\in [0,1]$. 
    
    For every such $\theta$, we claim that $\mathcal H^2(\pi_\theta(A))>0$. Otherwise, by definition of absolute continuity, we have $\pi_{\theta\#}\mu(\pi_\theta(A))=0$, that is, $\mu(\pi_\theta^{-1}(\pi_\theta(A)))=0$. Since $A\sub \pi_\theta^{-1}(\pi_\theta(A))$, this implies $\mu(A)=0$, a contradiction.
\end{proof}

\subsection{Preliminaries about \texorpdfstring{$C^2$}{Lg} curves}

Define 
\begin{equation}
\mathbf{e}_1(\theta)=\gamma(\theta),\,\mathbf{e}_{2}(\theta)=\gamma'(\theta),\,\mathbf{e}_3(\theta)=\mathbf{e}_1(\theta)\times \mathbf{e}_2(\theta).
\end{equation}
Then for any $\theta\in [0,1]$, $\{\mathbf{e}_i(\theta)\colon i=1,2,3\}$ is an orthonormal basis of $\R^3$.

\begin{lem}\label{lem_frame_derivatives}
    We have the following identities:    \begin{equation}
    \mathbf e_1'=\mathbf e_2,\quad    \mathbf e_2'=-\mathbf e_1+\uptau \mathbf e_3,\quad
    \mathbf e_3'=-\uptau \mathbf e_2,
\end{equation}
where we denote
$\uptau=\det (\gamma,\gamma',\gamma'')$.
\end{lem}

\begin{proof}
This is just a standard consequence of a Frenet frame with curvature $1$ and torsion $\uptau$, but for completeness we also give a simple proof here. The first relation is just the definition. To prove the second relation, we will need the relation 
\begin{equation}\label{eqn_gamma_times_gamma''}
    \gamma\cdot \gamma''=-1,
\end{equation}
which follows from differentiating $\gamma\cdot \gamma'=0$ and using $|\gamma'|=1$. Now write
\begin{equation}
    \mathbf e_2'=\gamma'':=a_1 \gamma+a_2\gamma'+a_3\gamma'\times \gamma',
\end{equation}
and using \eqref{eqn_gamma_times_gamma''} and the fact that $\{\gamma,\gamma',\gamma\times \gamma'\}$ forms an orthonormal basis, we get $a_1=-1$, $a_2=0$, and 
\begin{equation}
    a_3=\gamma''\cdot (\gamma\times \gamma')=\det (\gamma,\gamma',\gamma'')=\uptau.
\end{equation}
Similarly, we note that
\begin{equation}
    \mathbf e_3'=(\gamma\times \gamma')'=\gamma\times \gamma''.
\end{equation}
Write
\begin{equation}\label{eqn_Aug_27}
    \gamma\times \gamma''=b_1 \gamma+b_2\gamma'+b_3\gamma\times \gamma',
\end{equation}
and using the fact that $\{\gamma,\gamma',\gamma\times \gamma'\}$ forms an orthonormal basis, we get $b_1=0$ and $b_2=-\uptau$. To find $b_3$, taking dot products with $\gamma''$ on both sides of \eqref{eqn_Aug_27} and using $\gamma'\cdot \gamma''=0$ which comes from differentiating $|\gamma'|^2=1$, we get $b_3\uptau=0$, and so $b_3=0$ since $\uptau\ne 0$.    

\end{proof}

The following slightly more generalized Taylor expansion theorem will be used throughout the article to deal with low regularity issues. In words, with the assumption below, we can pretend that $f$ was a $C^2$ function when we do Taylor expansions.

\begin{prop}\label{prop_Taylor}
    Let $f:[a,b]\to \R^3$ be a $C^1$ function such that $f'=gh$ where $g:[a,b]\to \R^3$ is $C^1$, $g(a)=0$ and $h:[a,b]\to \R$ is bounded. Then we have the estimate
    \begin{equation}
        |f(b)-f(a)|\leq \frac 1 2(b-a)^2\norm{h}_\infty \norm{g'}_\infty.
    \end{equation}
\end{prop}
\begin{proof}
    By the fundamental theorem of calculus, we have
    \begin{align*}
        |f(b)-f(a)|
        &=\left|\int_0^1 (b-a) f'((1-t)a+tb)dt\right|\\
        &\le (b-a) \int_0^1 \left|g((1-t)a+tb)h((1-t)a+tb)\right|dt\\
        &\le (b-a) \norm{h}_\infty \int_0^1 \left|g((1-t)a+tb)\right|dt.
    \end{align*}
    For each $t\in [0,1]$, we can use the fundamental theorem of calculus again to get
    \begin{align*}
        |g((1-t)a+tb)|
        &=|g((1-t)a+tb)-g(a)|\\
        &=\left|\int_0^1 t(b-a) g'((1-s)a+s((1-t)a+tb))ds\right|\\
        &\le \int_0^1 \norm{g'}_\infty |t(b-a)|ds\\
        &= t (b-a)\norm {g'}_\infty.
    \end{align*}
    The result then follows.
\end{proof}

We also need the following elementary proposition.

\begin{prop}\label{prop_error_e_i}
    If $|\theta-\theta'|\leq\delta^{1/2}$, then we have
    \begin{align}
        |\mathbf{e}_i(\theta)\cdot \mathbf{e}_i(\theta')-1|&\lesssim \delta,\quad \forall i=1,2,3,\label{item_01_error_curve}\\
        |\mathbf{e}_2(\theta)\cdot \mathbf{e}_j(\theta')|&\lesssim \delta^{1/2},\quad \forall j=1,3,\label{item_02_error_curve}\\
        |\mathbf{e}_1(\theta)\cdot \mathbf{e}_3(\theta')|&\lesssim \delta.\label{item_03_error_curve}
    \end{align}
\end{prop}

\begin{proof}
The estimate \eqref{item_02_error_curve} is easy, since each $\mathbf e_i(\theta)$ is a Lipschitz function, and $\{\mathbf e_i(\theta):1\le i\le 3\}$ is an orthonormal basis for each $\theta$. The inequality \eqref{item_03_error_curve} follows from Taylor expanding $\mathbf e_1(\theta)$ with respect to the centre $\theta'$ and using that  $\{\mathbf e_i(\theta'):1\le i\le 3\}$ is an orthonormal basis.

To prove \eqref{item_01_error_curve} for $i=1$, we note that it follows directly from Lemma \ref{lem_frame_derivatives} and Taylor expansion. For $i=3$ we need to be slightly more careful since $\mathbf e_3$ is merely $C^1$.
Take $f(t)=\mathbf{e}_3(\theta)\cdot \mathbf{e}_3(t)$ and observe that by Lemma \ref{lem_frame_derivatives} one has $f'(t)=-\tau(t)\mathbf{e}_3(\theta)\cdot \mathbf{e}_2(t)=-\tau(t)g(t)$, for $g(t)=\mathbf{e}_3(\theta)\cdot \mathbf{e}_2(t)$. Since $g(\theta)=0$ and $f(\theta)=1$, Proposition \ref{prop_Taylor} gives $|\mathbf{e}_3(\theta')\cdot\mathbf{e}_3(\theta)-1 |\lesssim |\theta-\theta'|^{2}\lesssim \delta$.

It remains to prove the inequality for $i=2$. The idea is similar to the proof of Proposition \ref{prop_Taylor}: we compute, using Lemma \ref{lem_frame_derivatives},
\begin{align*}
|\mathbf e_2(\theta')\cdot \mathbf e_2(\theta)-1|
&=|\mathbf e_2(\theta')\cdot \mathbf e_2(\theta)-\mathbf e_2(\theta)\cdot \mathbf e_2(\theta)|\\
&=\left|\mathbf e_2(\theta)\cdot \int_0^1 \mathbf e'_2((1-t)\theta+t\theta')(\theta'-\theta)dt\right|\\
&\le |\theta'-\theta| \int_0^1 \left|\mathbf e_2(\theta)\cdot\mathbf e_1((1-t)\theta+t\theta')\right|dt\\
&+|\theta'-\theta| \int_0^1 \left|\uptau((1-t)\theta+t\theta')\mathbf e_2(\theta)\cdot\mathbf e_3((1-t)\theta+t\theta')\right|dt.
\end{align*}
Using $|\uptau|\sim 1$ and applying the already proved inequality (\ref{item_02_error_curve}) to control the inner products left inside the integrals by $\delta^{1/2}$, we are done.
\end{proof}

\subsection{Notation}
\begin{enumerate}
    \item $\lessapprox$ and $\gtrapprox$ will mean up to logarithmic losses. More precisely, for two nonnegative functions $A(\delta),B(\delta)$, we use $A(\delta)\lessapprox B(\delta)$ to mean the following: there exists some constant $\alpha>0$ and $C>1$ such that
    \begin{equation}
        A(\delta)\le C |\log \delta|^{\alpha}B(\delta),\quad \forall \delta\in (0,1/2].
    \end{equation}
    We can define $\gtrapprox$ in a similar way.
    
    \item We say a collection $\mathcal C$ of sets is $A$-overlapping if $\sum_{C\in \mathcal C}1_C\le A$.
    \item Unless otherwise specified, all implicit constants in this article are allowed to depend on the curve $\gamma$.
    \item For a family of subsets $\mathbb{T}$ of $\R^n$, we denote $\cup \mathbb{T}:=\cup_{T\in \mathbb{T}}T$.
    \item $\text{dim}(A)$ stands for the Hausdorff dimension of $A$.
    
\end{enumerate}

\subsection{Acknowledgements} {We are grateful to the organizers of the ``Study Guide Writing Workshop 2023" at the University of Pennsylvania for the opportunity of participating in this very interesting project and to all participants of this workshop for many fruitful discussions. We would like to give special thanks to our mentor, Professor Yumeng Ou, for all her guidance and support.}

\section{Some important definitions}
We first discuss the proof of Theorem \ref{firstmainthm}. It is proved using a standard technique commonly referred to as discretization. Namely, we first reduce Theorem \ref{firstmainthm} to a discretized version which can be proved using Fourier analysis. (see Theorem \ref{thm4paper} below). We start with some preliminaries in geometric measure theory.

We recall the definition of Hausdorff dimension in a metric space $X$. For $A\subset X$,
$$\textnormal{dim}(A)=\inf\{s\geq 0\colon \mathcal{H}^{s}(A)=0\}=\sup\{s\geq 0\colon \mathcal{H}^{s}(A)=\infty\}$$
where $\mathcal{H}^{s}(A)=\lim_{\delta\rightarrow 0}\mathcal{H}^{s}_{\delta}(A)$ and for $0<\delta\leq \infty$,
$$\mathcal{H}_{\delta}^{s}(A)=\inf\left\{\sum_{i=0}^{\infty} (\textnormal{diam}\,U_i)^s\colon A\subset \cup_{i=0}^{\infty}\, U_i, \,U_i\subset X\,\textnormal{ with } \textnormal{diam}\,(U_i)\leq \delta,\,\forall i\right\}.$$ 


\begin{defn}[Definition 1 in \cite{GGGHMW2022}]
    For a number $\delta$ and any set $X$, we use $|X|_{\delta}$ to denote the maximal number of $\delta$-separated points in $X$.
\end{defn}

\begin{example}
Here are two typical cases:
    \begin{itemize}
    \item If $X=B^d(0,r)$ and $\delta<r$, then $|X|_\delta\sim (r/\delta)^d$.

    \item If $X=C_n\subseteq \R$, the $n$-th stage of Cantor middle third set, then 
    \begin{itemize}
        \item if $\delta<3^{-n}$, we have $|X|_\delta\sim 2^{n}3^{-n}\delta^{-1}$.
        \item if $\delta>3^{-n}$, say $\delta=3^{-m}$ where $m<n$, then $|X|_\delta\sim 2^m$. In particular, if $C$ is the Cantor middle third set, then $|C|_\delta\sim 2^{m}=\delta^{-s}$ where $s=\text{dim}(C)=\log_32$ (see Figure \ref{fig:Cantor example}).
    \end{itemize}
\end{itemize}
\end{example}

\begin{figure}[h]
    \centering
    \scalebox{4}{
\begin{tikzpicture}
\draw[gray, line width=0.3pt] (0,0)--(1,0);

\draw[red, line width=0.7pt] (0,0)--(1/27,0);
\draw[red, line width=0.7pt] (2/27,0)--(3/27,0);

\draw[red, line width=0.7pt] (2/9,0)--(2/9+1/27,0);
\draw[red, line width=0.7pt] (2/9+2/27,0)--(2/9+3/27,0);

\draw[red, line width=0.7pt] (2/3,0)--(2/3+1/27,0);
\draw[red, line width=0.7pt] (2/3+2/27,0)--(2/3+3/27,0);

\draw[red, line width=0.7pt] (8/9,0)--(8/9+1/27,0);
\draw[red, line width=0.7pt] (8/9+2/27,0)--(8/9+3/27,0);

\draw[red] (1.2,0) node {\scalebox{.3}{$C_n$}};
\draw[blue] (0,-0.1)--(1/9,-0.1);
\draw[blue] (2/9,-0.1)--(2/9+1/9,-0.1);
\draw[blue] (2/3,-0.1)--(2/3+1/9,-0.1);
\draw[blue] (8/9,-0.1)--(8/9+1/9,-0.1);
\draw[blue] (1/18,-0.2) node {\scalebox{.3}{$\delta$} };
\end{tikzpicture}}
    \caption{Middle third Cantor set example.}
    \label{fig:Cantor example}
\end{figure}
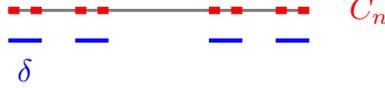

    The following proposition exhibits some uncertainty principle.

    \begin{prop}\label{prop_blurred}
    Given any finite set $X$, $|X|_\delta\sim |N_{\delta/4}(X)|_\delta$ where $N_\delta(X)$ denotes the $\delta$-neighbourhood of $X$. In words, when considering $|X|_\delta$, we can only see a $\delta/4$-blurred picture of $X$.
    \end{prop}
    \begin{proof}
        It suffices to show $|X|_\delta\gtrsim |N_{\delta/4}(X)|_\delta$. We need to show that given any $\delta$-separated subset $A\sub N_{\delta/4}(X)$, we have $\#A\lesssim |X|_\delta$.
For each point $a\in A$ choose $x_a\in X$ with $|x_a-a|<\delta/4$. 
        For any $u,v\in A$ with $u\ne v$ one has  $|u-v|>\delta$, and $|x_u-u|<\delta/4$, $|x_v-v|<\delta/4$. Thus $|x_u-x_v|>\delta/2$. Thus we have obtained a $\delta/2$-separated subset $B=\{x_a\colon a\in A\}\sub X$ with $\#A=\#B$. By dividing $\R^d$ into a $\delta/2$-grid and a simple pigeonholing argument, we can find a subset $B'\sub B$ such that $B'$ is $\delta$-separated and $\#B\sim \# B'$. Thus we have $\#A\sim \# B'\leq |X|_\delta$.

    \end{proof}

To make the Definition 2 in \cite{GGGHMW2022} slightly more rigorous, we define
\begin{defn}[Definition 2 in \cite{GGGHMW2022}]
    Let $X\sub \R^d$ be a bounded set. Let $\delta>0$ and let $0\le s\le d$. We say that $X$ is a $(\delta,s,C)$-set if for every ball $B_r\sub \R^d$,
    \begin{equation}
        |X\cap B_r|_\delta\le C(r/\delta)^s,\quad \forall\, \delta\le r\le 1.
    \end{equation}
    For simplicity, we say that $X$ is a $(\delta,s)$-set if $X$ is a $(\delta,s,C)$-set for some $C$ depending on $s,d$ only.
\end{defn}
By Proposition \ref{prop_blurred}, we see that if $X$ is a $(\delta,s)$-set, then $N_{\delta/4}(X)$ is also a $(\delta,s)$-set.

    \begin{example}
         A typical case of a $(\delta,s)$-set is given by the union of $2^n$ disjoint intervals of length $\delta$, which appears at the $n$-th stage of the standard middle $(1-2\alpha)$-Cantor set $C$, where $\delta=\alpha^{n}$ and $\dim C=\log_{\alpha^{-1}}(2)$.

    To check such set $C_n$ is a $(\delta,s)$-set for $\delta=\alpha^{n}$ and $s=\text{dim}(C)$ we observe that for any $0\leq m<n$ and $r=\alpha^m$ we have
    $|C_n\cap B_{\alpha^m}|_{\alpha^{n}}\sim 2^{n-m}$
    and 
    \begin{equation}
        (r/\delta)^{s}=\alpha^{-(n-m)s}=2^{n-m} \text{(see Figure \ref{fig:second example})}.
    \end{equation}

    \end{example}

\begin{figure}[h]
    \centering
    \scalebox{4}{
\begin{tikzpicture}
\draw[gray, line width=0.3pt] (0,0)--(1,0);

\draw[red, line width=0.7pt] (0,0)--(1/27,0);
\draw[red, line width=0.7pt] (2/27,0)--(3/27,0);

\draw[red, line width=0.7pt] (2/9,0)--(2/9+1/27,0);
\draw[red, line width=0.7pt] (2/9+2/27,0)--(2/9+3/27,0);

\draw[red, line width=0.7pt] (2/3,0)--(2/3+1/27,0);
\draw[red, line width=0.7pt] (2/3+2/27,0)--(2/3+3/27,0);

\draw[red, line width=0.7pt] (8/9,0)--(8/9+1/27,0);
\draw[red, line width=0.7pt] (8/9+2/27,0)--(8/9+3/27,0);

\draw[red] (1.2,0) node {\scalebox{.3}{$C_n$}};
\draw[blue] (1.2,-0.2) node {\scalebox{.3}{$C_m$}};

\draw[blue](0,-0.1)--(1/3,-0.1);
\draw[blue](2/3,-0.1)--(1,-0.1);
\draw[blue] (2/3,-0.1)--(2/3+1/9,-0.1);
\draw[blue] (8/9,-0.1)--(8/9+1/9,-0.1);

\draw[blue] (1/9+1/18,-0.2) node {\scalebox{.2}{$3^{-m}$} 
};
\draw[red] (1/36,0.1) node {\scalebox{.15}{$3^{-n}$} 
};
\draw[orange] (2/9,0.3) node {\scalebox{.2}{$2^{n-m}$ red intervals within each blue interval}};
\draw[orange,->] (1/9+1/18,0.25)--(1/9+1/18,0.1);
\end{tikzpicture}}
    \caption{$(3^{-n},\log_{3}(2))$-set example, with $n=3$ and $m=1$.}
    \label{fig:second example}
\end{figure}
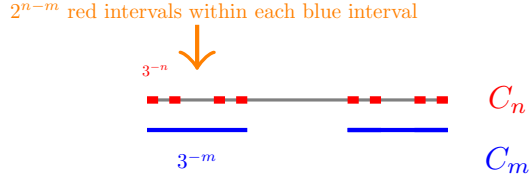

\section{Discretized Marstrand's Theorem via high-low method}

To illustrate the high-low method in Fourier analysis, we pause the proof of Theorem \ref{firstmainthm} and give an alternative proof of Marstrand's projection theorem in $\R^2$, resulting in a Falconer type exceptional set estimate (see Section \ref{sec_exceptional}).

\subsection{Statement of discretized Marstrand's Theorem}
We first slightly rephrase Proposition 1 of \cite{GGGHMW2022} as follows.
\begin{prop}[Proposition 1]\label{Mdiscrete}
    Fix $0<s<1$. Fix a small scale $\delta>0$ and a $\delta$-separated set $\Theta\sub [0,\pi)$. For each $\theta\in \Theta$, define $L_\theta:=\{x\in \R^2:\arg (x)=\theta\}\cup\{(0,0)\}$ and let $p_\theta:\R^2\to L_\theta$ be the projection. Suppose 
    \begin{itemize}
        \item $A\sub B^2(0,1)$ is a union of disjoint $\delta$-balls with measure $|A|=\delta^{2-a}$, or equivalently $|A|_\delta\sim \delta^{-a}$.
        \item Assume for every $\theta\in \Theta$, $p_\theta(A)$ (which is a union of line segments of length $\delta$ in $L_\theta$) satisfies the $s$-dimensional condition: for each $r\in [\delta,1]$ and line segment $I_r\sub L_\theta$ of length $r$, we have 
    \begin{equation}\label{eqn_s_dim}
        |p_\theta(A)\cap I_r|_\delta\lesssim (r/\delta)^s.
    \end{equation}   
    \end{itemize}
     
    Then
    \begin{equation}
       \#\Theta\lesssim_s \delta^{-1+a-s}.
    \end{equation}

    In particular, if $\#\Theta\sim \delta^{-1}$, then the above becomes
    \begin{equation}
        \delta^{-a}\lesssim_s \delta^{-s}.
    \end{equation}
\end{prop}
Under the uncertainty scale $\delta$, this form of $A$ is general enough.
The set of directions $\Theta'$ can be thought of as containing almost all directions since it has cardinality $\sim \delta^{-1}$. For simplicity, we may just take $\Theta'=\Theta$.

Intuitively, this proposition roughly implies the following. Assume $a=\dim A$. Then for each uncertainty scale $\delta$, $A$ can be thought of as a disjoint union of  $\delta^{-a}$ disjoint $\delta$-balls. If the projections of $A$ along almost all directions is a $(\delta,s)$-set, then $a\le s$. This is roughly the statement of Marstrand's Projection Theorem. For more precise details about the discretization step for Marstrand's theorem we refer to Subsection \ref{subsecdiscmarstrand}.

Now we start with the proof of Proposition \ref{Mdiscrete}, divided into the following subsections.
\subsection{Construction of wave packets}

    For each $\theta\in \Theta$, let $\mathbb T_\theta$ be a set of $\delta\times 1$ tubes that cover $p_\theta^{-1}(p_\theta(A))\cap B^2(0,1)$ and hence cover $A$. We may assume all tubes have the same orientation, that is, perpendicular to $L_\theta$, and are all centred on $L_\theta$. Note that $\#\mathbb T_\theta\lesssim \delta^{-s}$ by the $s$-dimensional condition.

    Now we come to the Fourier side. Let $S_\theta$ be the dual rectangle to all $T_\theta$'s, that is, $S_\theta$ is a $\delta^{-1}\times 1$ rectangle centred at $0$ with orientation perpendicular to $T_\theta$. 
    
    For each $T_\theta\in \mathbb T_\theta$, define a Schwartz function $\psi_{T_\theta}:\R^2\to \R$ Fourier supported on $S_\theta$ such that $\psi_{T_\theta}\gtrsim 1$ on $T_\theta$ and decays rapidly outside $T_\theta$. 
    
    For a concrete example, assume $L_\theta$ is the $x$-axis and $T_\theta=[c-\delta,c+\delta]\times [-1,1]$. Then $S_\theta=[-\delta^{-1},\delta^{-1}]\times [-1,1]$, and we may take 
    $$
    \psi_{T_\theta}(x,y)=\eta(\delta^{-1} (x-c))\eta(y)
    $$
    where $\widehat{\eta}$ is a nonnegative even function supported on $(-1,1)$ and such that $\widehat \eta(0)=1$. 

    Lastly, we define
    \begin{equation}
        f_\theta=\sum_{T_\theta\in \mathbb T_\theta}\psi_{T_\theta},\quad f=\sum_{\theta\in \Theta}f_\theta.
    \end{equation}
    Thus $f_\theta$ is also Fourier supported on $S_\theta$, but physically essentially supported on $\cup\mathbb T_\theta$. The function $f$ is Fourier supported on $\cup_{\theta\in \Theta}S_\theta$, which is a union of rectangles of the same size centered at $0$, but with $\delta$-separated orientations.

    \subsection{High-low decomposition}
    The high-low (frequency) decomposition is based on the following simple fact in plane geometry. This gives some sort of orthogonality when we deal with the high part.
    \begin{lem}\label{lem_high_orthogonality}
        Let $\mathcal S$ be a collection of $\delta^{-1}\times 1$ rectangles centred at the origin, with $\delta$-separated orientations. If $1\le K\le \delta^{-1}$, then $\cup \mathcal S\backslash B(0,K^{-1}\delta^{-1})$ is at most $O(K)$-overlapping.
    \end{lem}
    \begin{proof}[Proof of lemma]
        Let $T\in \mathcal S$ and let $l_T$ be the central axis along the longer side of $T$. If $x\in T\backslash B(0,K^{-1}\delta^{-1})$, then the angle between $Ox$ and $l_T$ is less than $O(K)\delta$ (see Figure \ref{pic:Koverlap}). Thus if $x\in (T\cap T')\backslash B(0,K^{-1}\delta^{-1})$, by the triangle inequality, the angle between $l_{T'}$ and $l_{T}$ is less than $O(K)\delta$. Since $\mathcal S$ is $\delta$-separated, this gives the claim.
    \end{proof}

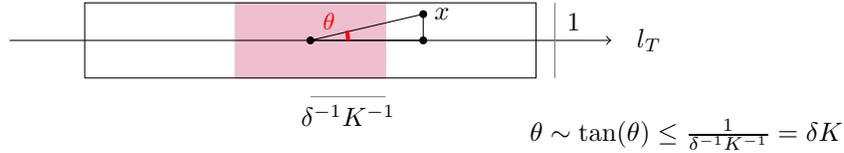
\begin{figure}[h]\label{pic:Koverlap}
\begin{center}
         \scalebox{1}{
\begin{tikzpicture}
\filldraw[purple!25!white]  (-1,-0.5)--(1,-0.5)--(1,0.5)--(-1,0.5)--(-1,-0.5);
\draw (-3,-0.5)--(3,-0.5)--(3,0.5)--(-3,0.5)--(-3,-0.5);
\draw [->] (-4,0)--(4,0);
\draw (0,0)--(1.5,0)--(1.5,0.35)--(0,0);
\fill (0,0) circle (1.5pt);
\fill (1.5,0.35) circle (1.5pt);
\fill (1.5,0) circle (1.5pt);
\draw[gray] (0,-0.75)--(1,-0.75);
\draw[gray] (3.25,-0.5)--(3.25,0.5);
\draw[line width=1.5pt,red] (0.5,0) arc (0:15:0.5);
\draw[red] (0.25,0.25) node {\small{$\theta$}};
\draw (4.5,0) node {$l_T$};
\draw (1.75,0.35) node {$x$};
\draw (0.5,-1) node {{$\delta^{-1}K^{-1}$}};
\draw (3.5,0.25) node {$1$};
\draw (5,-1.25) node {$\theta\sim\tan(\theta)\leq \frac{1}{\delta^{-1}K^{-1}}=\delta K$};
\end{tikzpicture}}
\end{center}
\caption{Illustration of Lemma \ref{lem_high_orthogonality}.}
\end{figure}

Now given $K\in [1,\delta^{-1}]$ to be determined, and let $\eta_l$ be a smooth bump function adapted to $B^2(0,K^{-1}\delta^{-1})$ in the frequency space. Let $\eta_h=1-\eta_l$. We decompose 
\begin{equation}
    f=f_l+f_h
\end{equation}
where $\widehat {f_l}=\widehat f \eta_l$, $\widehat {f_h}=\widehat f \eta_h$. 

\subsection{Analysis of high part}

In this part, we use Fourier analysis. By Plancherel's identity and Lemma \ref{lem_high_orthogonality}, we have
\begin{equation}
\begin{split}
     \int |f_h|^2=\int |\widehat {f_h}|^2&\leq\int |\sum_{\theta\in \Theta}\hat{f_{\theta}}\eta_h\cdot\mathbb{1}\text{supp}(\hat{f}_{\theta,h})|^2\\
     &\lesssim \int K\sum_{\theta\in\Theta}|\widehat{f_\theta}|^2\\
     &\lesssim K\sum_{\theta\in \Theta}\int |f_\theta|^2.
\end{split}
\end{equation}
By the rapid decay of $\psi_{T_\theta}$ (see \eqref{eqn_weight_single_T} and Corollary \ref{cor_sum_weights} for details), we have 
\begin{equation}
    \int |f_\theta|^2\sim \delta\#\mathbb T_\theta\lesssim \delta^{1-s}.
\end{equation}
Thus
\begin{equation}\label{eqn_high_upper_bound}
    \int |f_h|^2\lesssim K(\#\Theta)\delta^{1-s}.
\end{equation}

\subsection{Analysis of low part}\label{sec_step_4}

For the low part, we do not have orthogonality to aid us. Instead, we use the $s$-dimensional condition \eqref{eqn_s_dim} in the assumption of the proposition. We first claim that for each $x\in A$ and $\theta\in \Theta$, we have
\begin{equation}\label{eqn_blur_scale_K}
|f_\theta*{\eta}_l^\vee(x)|\lesssim K^{-1}\#\{T_\theta\in\mathbb T_\theta: T_\theta\cap B^2(x,CK\delta)\neq \varnothing\}.
\end{equation}
Intuitively, this holds by the uncertainty principle, since we now view the rectangles $T_\theta$ blurred by scale $K\delta$.

We give a heuristic proof of \eqref{eqn_blur_scale_K} first. Assume for simplicity that $f_\theta\approx \sum_{T_\theta\in \mathbb T_\theta}1_{T_\theta}$. Assume also $\eta_l^\vee\approx (K^{-1}\delta^{-1})^21_{B^2(0,K\delta)}$. Then for each $x\in \R^2$,
\begin{equation}    1_{T_\theta}*1_{B^2(0,K\delta)}(x)
\begin{cases}
    =0, &\text{ if } T_\theta\cap B^2(x,K\delta)=\varnothing,\\
    \lesssim K\delta^2, &\text{ if } T_\theta\cap B^2(x,K\delta)\ne\varnothing.
\end{cases}
\end{equation}
The rigorous proof will be given later in Section \ref{sec_weight} to avoid distraction from the main idea of the proof.

Using \eqref{eqn_blur_scale_K} and the $s$-dimensional condition \eqref{eqn_s_dim}, for every $x\in A$ we have
\begin{equation}\label{eqn_low_part_upper_bound}
    |f_l(x)|\leq \sum_{\theta\in \Theta}|f_\theta*\eta_l^\vee(x)|\lesssim \# \Theta K^{-1}(K\delta/\delta)^s\sim \# \Theta K^{s-1}.
\end{equation}

\subsection{Choosing \texorpdfstring{$K$}{Lg} such that the high part dominates}

For each $x\in A$, we have $f_\theta(x)\gtrsim 1$ by construction of $\psi_{T_\theta}$. Thus $f(x)\gtrsim \# \Theta$ for $x\in A$.

Since $s<1$, if $K=K(s)$ is chosen to be large enough, then $|f_l(x)|<|f(x)|/2$ for each $x\in A$ by \eqref{eqn_low_part_upper_bound}. Thus, by the triangle inequality, the high part dominates for $x\in A$.

\subsection{Conclusion}

Since the high part dominates, by \eqref{eqn_high_upper_bound}, we have
\begin{equation}
    \int_A f^2\lesssim K\#\Theta\delta^{1-s}.
\end{equation}
On the other hand, since $f(x)\gtrsim \# \Theta$ on $A$, we trivially have the lower bound
\begin{equation}
    \int_A f^2\gtrsim (\# \Theta)^2 |A|\sim (\# \Theta)^2\delta^{2-a}.
\end{equation}
Since $K$ depends on $s$ only, we have $(\# \Theta)^2\delta^{2-a}\lesssim_s \#\Theta\delta^{1-s}$, as required.

\subsection{Weight inequalities}\label{sec_weight}
We now give a rigorous proof of the heuristic argument used to prove \eqref{eqn_blur_scale_K}, by introducing weight functions. See also \cite{BD2017_study_guide} for inequalities of weights. 

\subsubsection{Notation}
Let $\psi:\R\to \C$ be a Schwartz function such that $\psi\gtrsim 1$ in $(-1,1)$. Given an axis parallel rectangle $T=[c_1-r_1,c_1+r_1]\times [c_2-r_2,c_2+r_2]$, we define the (unnormalised) Schwartz function $\eta_T$ adapted to $T$ to be
\begin{equation}\label{eqn_Schwartz_definition}
\eta_T(x_1,x_2)=\psi\left(\frac {x_1-c_1}{r_1}\right)\psi\left(\frac {x_2-c_2}{r_2}\right).
\end{equation}
For a general rectangle $T$, define $\eta_T$ in the most natural way with rotations.

Let
\begin{equation}\label{eqn_definition_E}
    E=\max\left\{100,1+\frac 2 {1-s}\right\},
\end{equation}
where we recall $s\in (0,1)$ is given at the beginning of Proposition \ref{Mdiscrete}.

Similarly to \eqref{eqn_Schwartz_definition}, we define weight functions $w=w^{(E)}$
\begin{equation}
    w(x_1,x_2)=(1+|x_1|)^{-E}(1+|x_2|)^{-E}.
\end{equation}
In general, for a rectangle $T\sub \R^d$ with centre $c$, sides $e_i$ and dimensions $r_i$, we define
\begin{equation}\label{eqn_weight_T_2D}
    w_T(x)=w_{T,E}(x)=\prod_{i=1}^d\left(1+\frac{|(x-c)\cdot e_i|}{r_i}\right)^{-E}.
\end{equation}
Similarly, for a ball $B=B^2(c,r)$, we define $\eta_B=\eta_T$ and $w_B=w_T$ where $T$ is the square centred at $c$ of side length $r$.

It is easy to see that 
\begin{equation}\label{eqn_weight_single_T}
    1_T\lesssim \eta_T\lesssim w_T.
\end{equation}
Also, if $T$ and $T'$ are essentially the same, meaning $C^{-1}T\sub T'\sub CT$ for some absolute constant $C>1$, then $w_T\sim w_{T'}$. All implicit constants below are allowed to depend on $E$. (Note that this property fails if $w$ is replaced by $\eta$. Similar for Proposition \ref{prop_weights_convolution} below.)

We also need the following proposition.
\begin{prop}\label{prop_tiling_weight}
    Suppose $\mathbb I$ is a tiling of $\R$ by intervals $I$ of the same size. Let $C>1$, $x\in \R$ and $\mathbb I'\sub \mathbb I$ be such that $\mathrm{dist}(x,I)\ge C|I|$ for every $I\in \mathbb I'$. Then
    \begin{equation}
        \sum_{I\in \mathbb I'}\left(1+\frac{|x-c_I|}{|I|}\right)^{-E}\lesssim C^{-E+1},
    \end{equation}
    where $c_I$ denotes the centre of $I$.
\end{prop}
\begin{proof}
    By scaling we may assume $|I|=1$, for all $I\in \mathbb{I}$. By translation we may assume $x=0$. Then the left hand side is essentially bounded above by
    \begin{equation}
        \int_{|y|\ge C}(1+|y|)^{-E}dy\sim C^{-E+1}.
    \end{equation}
\end{proof}
\begin{cor}\label{cor_sum_weights}
    If $\mathbb T$ is a tiling of $\R^2$ by rectangles $T$, then $\sum_{T\in \mathbb T}w_T\sim 1$.
\end{cor}

\subsubsection{Elementary inequalities}

\begin{prop}
    Let $A,B\sub \R^d$ be Borel sets. Then
    \begin{equation}
        1_A*1_B\leq \min\{|A|,|B|\} 1_{A+B}.
    \end{equation}
\end{prop}
The proof is trivial. Using the case $d=1$ and $A,B$ being intervals entry-wise, this can be upgraded to the following proposition:
\begin{prop}\label{prop_weights_convolution}
    Let $T,T'$ be parallel (perpendicular) rectangles of dimensions $a\times b$, $a'\times b'$, respectively. Then    
\begin{equation}
    w_{T}*w_{T'}\lesssim \min\{a,a'\}\min\{b,b'\}w_{T+T'}.
\end{equation}
\end{prop}
We leave the elementary proof to the reader.

\subsubsection{Proof of \texorpdfstring{\eqref{eqn_blur_scale_K}}{Lg}}

For each $x$, denote $\mathbb T_\theta(x):=\{T_\theta\in \mathbb T_\theta:T_\theta\cap B^2(x,CK\delta)\ne \varnothing\}$. 

Let $x\in A$. Then $\mathbb T_\theta(x)\neq \varnothing$. Write
\begin{equation}
    |f_\theta*\eta_l^\vee(x)|\leq \sum_{T_\theta\in \mathbb T_\theta(x)}|\psi_{T_\theta}*\eta_l^\vee(x)|+\sum_{T_\theta\notin \mathbb T_\theta(x)}|\psi_{T_\theta}*\eta_l^\vee(x)|:=I_1(x)+I_2(x).
\end{equation}
Using \eqref{eqn_weight_single_T} and Proposition \ref{prop_weights_convolution}, we bound 
\begin{equation}
    |\psi_{T_\theta}*\eta_l^\vee(x)|\lesssim K^{-2}\delta^{-2}\cdot CK\delta\cdot \delta w_{T'_\theta}(x)=CK^{-1}w_{T'_\theta}(x),
\end{equation}
where $T'_\theta:=T_\theta+B^2(0,K\delta)$ is a rectangle with the same orientation and centre as $T_\theta$ and with dimensions $1\times CK\delta$.

Using this and $w_{T'_\theta}\lesssim 1$, we trivially bound
\begin{equation}
    I_1(x)\lesssim CK^{-1}\#\mathbb T_\theta(x).
\end{equation}
To bound $I_2(x)$, we note that $\{T'_\theta\}$ are at most $O(K)$ overlapping. Since all centres of rectangles $T_\theta$ are on the same straight line, using Proposition \ref{prop_tiling_weight}, we have 
\begin{equation}
    I_2(x)\lesssim CK^{-1} \cdot K C^{-E+1}.
\end{equation}
Since $\#\mathbb T_\theta(x)\ge 1$ for each $x\in A$, choosing $C=K^{\frac {1}{E-1}}$ gives that
\begin{equation}    |f_\theta*\eta^\vee_l(x)|\lesssim K^{-1+\frac 1 {E-1}},
\end{equation}
which is actually slightly weaker than \eqref{eqn_blur_scale_K}, but it it turns out to be sufficient for our purpose. Indeed, note that \eqref{eqn_low_part_upper_bound} becomes
\begin{equation}
    |f_l(x)|\lesssim C\#\Theta K^{s-1}\lesssim \#\Theta K^{\frac {s-1}2},
\end{equation}
in view of our choice of $E$ in \eqref{eqn_definition_E}. Thus the low part is still dominated by the high part.

\section{Discretization}

In order to make some steps of the discretization of Theorem \ref{firstmainthm} more precise, we will make use of Frostman measures in place of the Hausdorff content, since $\mathcal{H}^{a}_{\infty}$ have the disadvantage of only being an outer measure. See Definition \ref{deffrostmeasure} for  the definition of Frostman measure.

By Frostman's Lemma (see Theorem \ref{thm_Frostman_lemma}) we know that $A$ admits an $s$-Frostman measure if and only if $\mathcal{H}^{s}(A)\in (0,\infty]$. Hence, $A$ admits an $s$-Frostman measure $\nu_s$ for any $0<s<\text{dim}(A)$.

\subsection{Lemma 1 and Lemma 2}

\begin{lem}[Lemma 1 in \cite{GGGHMW2022}]\label{Lemma1paper}
    Let $\delta\in (0,1)$, $s\in (0,d]$ and $B\sub [0,1]^d$. Assume there is an $s$-Frostman measure $\nu$ such that $\nu(B)=\kappa>0$. Then, there exists a $\delta$-separated finite $(\delta,s)$-set $P\sub B$ with cardinality $\#P\gtrsim \rho_s(\nu)^{-1} \kappa \delta^{-s}$.
\end{lem}

\begin{proof}[Proof of Lemma \ref{Lemma1paper}]
By losing at most a dimensional constant in $\gtrsim$, we may assume $\delta=2^{-n}$ for some $n\in \N$. For each $0\le m\le n$, denote by $\mathcal D_m$ the tiling of $[0,1]^d$ with axis parallel cubes of side length $2^{-m}$ that intersect $B$. We may also assume the cubes are closed on the left and open on the right for each coordinate; thus the tiling becomes strictly disjoint. 

We say a finite set $P\sub [0,1]^d$ satisfies the $s$-dimensional condition at scale $r=2^{-m}$ if
\begin{equation}
    |P\cap Q|_\delta\leq 2(r/\delta)^s,\quad \forall Q\in \mathcal D_m.
\end{equation}
The factor $2$ here is introduced just for technical treatment of rounding errors. Note that if $P$ is assumed to be $\delta$-separated, then the above is equivalent to
\begin{equation}\label{eqn_s_dimensional_condition_numbers}
    \# (P\cap Q)\le 2(r/\delta)^s,\quad \forall Q\in \mathcal D_m.
\end{equation}
Thus a $\delta$-separated set $P\sub [0,1]^d$ is a $(\delta,s)$-set if and only if it satisfies the $s$-dimensional condition \eqref{eqn_s_dimensional_condition_numbers} above at every dyadic scale $r\in [\delta,1]$.

We start at scale $2^{-n}$. Define $P_n$ by choosing one point from each member of $\mathcal D_n$; 

The set $P_n$ clearly satisfies the $s$-dimensional condition at scale $2^{-n}$. Note $B\sub \cup \mathcal D_n$. Denote $\mathcal C_n:=\mathcal D_n$.

    We now consider the scale $2^{1-n}$. For each $Q_{n-1}\in \mathcal D_{n-1}$, we say it is bad (at this level) if $Q_{n-1}\cap P_n$ does not satisfy the $s$-dimensional condition at scale $2^{1-n}$. In this case, we remove points from $Q_{n-1}\cap P_n$ until it just becomes good. Meanwhile, we remove from $\mathcal C_n$ all cubes $Q_n\sub Q_{n-1}$, and add $Q_{n-1}$ to it. Perform this for each bad cube $Q_{n-1}$, and call the resulting point set $P_{n-1}\sub P_n$ and the resulting collection of cubes $\mathcal C_{n-1}$. Note that $P_{n-1}$ satisfies the $s$-dimensional condition at scales $2^{-n}$ and $2^{1-n}$, and also $B\sub \cup \mathcal C_{n-1}$.
    
    We now consider the scale $2^{2-n}$. For each $Q_{n-2}\in \mathcal D_{n-2}$, if $Q_{n-2}\cap P_{n-1}$ does not satisfy the $s$-dimensional condition at scale $2^{2-n}$, we remove points from $Q_{n-2}\cap P_{n-1}$ until it does so. Also, replace all cubes $Q_{n-1}\sub Q_{n-2}$ from $\mathcal C_{n-1}$ and replace them with $Q_{n-2}$. Perform this for each bad cube $Q_{n-2}$, and call the resulting point set $P_{n-2}\sub P_{n-1}$ and collection of cubes $\mathcal C_{n-2}$. Note that $P_{n-2}$ satisfies the $s$-dimensional condition at scales $2^{-n}$, $2^{1-n}$ and $2^{2-n}$. Also, $B\sub \cup \mathcal C_{n-2}$.

    We run the above algorithm exactly $n$ times until we arrive at a point set $P=P_0$ and a collection of cubes $\mathcal C_0$. By construction, $P$ is a $(\delta,s)$-set. Also, $B\sub \cup \mathcal C_0$.

    It remains to prove that $\#P\gtrsim \rho_s(\nu)^{-1}\kappa \delta^{-s}$. For each $0\le m\le n$, we denote by $\mathcal C^m$ the collection of cubes in $\mathcal C_0$ that has length $2^{-m}$. Using the fact that $\nu$ is an $s$-Frostman measure, we have 
\begin{equation}\label{eqn_lower_bound_Hausdorff_content}
        \sum_{m=0}^n 2^{-ms}\rho_s(\nu)\# (\mathcal C^m)\gtrsim \sum_{m=0}^{n}\nu(\cup \mathcal{C}^m)\ge \nu(\cup \mathcal{C}_0)\ge\nu(B)= \kappa.
    \end{equation}
    On the other hand, recall that for each bad cube $Q_m$ at each scale $2^{-m}$, $0\le m\le n-1$, we removed points from $P_{m+1}\cap Q_m$ until $Q_m$ just becomes good. Also, if $Q_{m}\in \mathcal C^m$, this means that $Q_m$ is left untouched by the algorithm at subsequent larger scales. Thus 
    \begin{equation}
        \#(P\cap Q_m)=\#(P_m\cap Q_m)> 2(2^{-m}/2^{-n})^s-1\ge 2^{(n-m)s}.
    \end{equation}
    For $Q_n\in \mathcal C^n$, by construction we have $\# (P\cap Q_n)=1=2^{(n-n)s}$. Thus we have
    \begin{equation}
    \begin{split}
        \#P=\sum_{m=0}^n \#(P\cap Q_m)\#(\mathcal C^m)\ge&\sum_{m=0}^n 2^{(n-m)s} \#(  \mathcal C^m)\\
        \gtrsim &2^{ns}\rho_s(\nu)^{-1}\kappa=\rho_s(\nu)^{-1}\kappa \delta^{-s},
    \end{split}
    \end{equation}
    where we have used \eqref{eqn_lower_bound_Hausdorff_content}.
\end{proof}

We will now move to the proof of Lemma 2 in \cite{GGGHMW2022}. We introduce the following notation: let $\mathcal{D}$ be the collection of dyadic squares in $[0,1]^2$ and let $\mathcal{D}_{2^{-j}}$ denote the collection of dyadic squares with side-lengths $2^{-j}$.

\begin{lem}[Lemma 2 in \cite{GGGHMW2022}]\label{Lemma2paper}
Fix $\epsilon >0$ and suppose $\dim(X)<s$. Then there exists a family of collections of dyadic squares $\{\mathcal{C}_{2^{-k}}\}_{k=1}^{\infty}$ such that $\mathcal{C}_{2^{-k}} \subset \mathcal{D}_{2^{-k}}$ such that the following properties hold. 
    \begin{enumerate}
        \item $\sum_{k \geq 0} \sum_{D \in \mathcal{C}_{2^{-k}}} r(D)^s < \epsilon$
       \item  $X \subseteq \bigcup_{k \geq 0} \bigcup_{D \in \mathcal{C}_{2^{-k}}} D$
       \item For any $l < k$ and any $D \in \mathcal{D}_{2^{-l}}$ we have that 
       $$\# \{D' \in \mathcal{C}_{2^{-k}}: D' \subset D \} \leq 2^{(k-l)s}.$$ 
       This is called the $s$-dimensional condition.
    \end{enumerate}
\end{lem}

\textbf{Remark}: The first two conditions can be deduced immediately from the condition that $\dim(X)<s$. The third condition corresponds to the construction of a $(\delta, s)$ set (see Lemma 1) where in this particular case $\delta = 2^{-k}$.

\begin{proof}
We consider the collection $\Pi$ of all coverings $\mathcal{C}$ of $X$ by dyadic squares which also satisfy conditions $(1)$ and $(2)$ in Lemma $\ref{Lemma2paper}$. For each such covering, $\mathcal{C}$, we may assume (by the structure of dyadic squares) that the dyadic squares in $\mathcal{C}$ are disjoint. We define a partial order $\prec$ over the collection of all such coverings as follows: If $\mathcal{C} = \mathcal{C}'$ we declare that $\mathcal{C} \prec \mathcal{C}'$ and also that $\mathcal{C}' \prec \mathcal{C}$. If $\mathcal{C} \neq \mathcal{C}'$, we say $\mathcal{C} \prec \mathcal{C}'$  when $\mathcal{C} \neq \mathcal{C}'$ if the following two conditions hold:
\begin{enumerate}
    \item There exists (a unique) $k_0>0$ such that $\mathcal{C} \cap \mathcal{D}_{2^{-k}} = \mathcal{C}' \cap \mathcal{D}_{2^{-k}}$ for $k<k_0$ and $\mathcal{C} \cap \mathcal{D}_{2^{-k_0}} \subsetneq \mathcal{C}' \cap \mathcal{D}_{2^{-k_0}}$. 
    \item For each $x \in X$, the dyadic square contained in $\mathcal{C}$ that contains $x$ must be a subset of of the dyadic square in $\mathcal{C}'$ that contains $x$. 
\end{enumerate}

We check that $\prec$ determines a partial order on $\Pi$. It is easy to see that $\prec$ is reflexive and antisymmetric. We now that that it is transitive. Suppose, $\mathcal{C} \prec \mathcal{C}'$ and $\mathcal{C}' \prec \mathcal{C}''$. We must show that $\mathcal{C} \prec \mathcal{C}''$. Note that if either $\mathcal{C} = \mathcal{C}'$ or $\mathcal{C}' = \mathcal{C}''$ then this is immediate. Thus, let use assume that $\mathcal{C} \neq \mathcal{C}' \neq \mathcal{C}''$. By condition (a) we see that there exists $k_1>1$ such that $\mathcal{C} \cap  \mathcal{D}_{2^{-k}} = \mathcal{C} \cap \mathcal{D}_{2^{-k}}$ for $k< k_1$ and $\mathcal{C} \cap \mathcal{D}_{2^{-k_1}} \subsetneq \mathcal{C} \cap \mathcal{D}_{2^{-k_1}}$ and similarly there exists $k_2 > 0$ such that $\mathcal{C} \cap \mathcal{D}_{2^{-k}} = \mathcal{C}'' \cap \mathcal{D}_{2^{-k}}$ for $k< k_2$ and $\mathcal{C} \cap \mathcal{D}_{2^{-k_2}} \subsetneq \mathcal{C} \cap \mathcal{D}_{2^{-k_2}}$. Let $k_0 = \min\{k_1, k_2\}$. Note that for $k < k_0$ we have that $\mathcal{C} \cap \mathcal{D}_{2^{-k}} = \mathcal{C}'' \cap \mathcal{D}_{2^{-k}}$. Moreover we have that either $\mathcal{C} \cap \mathcal{D}_{2^{-k_0}} \subsetneq \mathcal{C}' \cap \mathcal{D}_{2^{-k_0}}$ or $\mathcal{C}' \cap \mathcal{D}_{2^{-k_0}} \subsetneq \mathcal{C}'' \cap \mathcal{D}_{2^{-k_0}}$. AS a consequence we have that $\mathcal{C} \cap \mathcal{D}_2^{-k_0} \subsetneq \mathcal{C}'' \cap \mathcal{D}_{2^{-k_0}}$. This proves condition (a). If $x \in X$, then since the $\mathcal{C} \prec \mathcal{C}' \prec \mathcal{C}''$, the dyadic cube that contains $x$ in $\mathcal{C}''$ contains the dyadic cube that contains $x$ in $\mathcal{C}'$ which in turn contains the dyadic cube that contains $x$ in $\mathcal{C}$. Thus condition (b) is also satisfied and thus $\mathcal{C} \prec \mathcal{C}''$.

Suppose $\mathcal{C}$ is a covering that is maximal with respect to the ordering $\prec$. We claim that such a covering must satisfy condition (3) in Lemma \ref{Lemma2paper}. Indeed, suppose in order to obtain a contradiction that (3) does not hold: then there exist $l$ and $k$ with $l<k$ such that for some square $D \in \mathcal{D}_{2^{-l}}$ we have the following: 
\begin{equation}\label{eqn_contraction_sdim}
    \# \{D' \in \mathcal{C}_{2^{-k}}: D' \subset D \} > 2^{(k-l)s}
\end{equation}
We construct a collection $\mathcal{C}' \in \Pi$ such that $\mathcal{C} \prec \mathcal{C}'$. We construct $\mathcal{C'}$ by removing all elements in $\{D' \in \mathcal{C}_{2^{-k}}: D' \subset D \}$ from $\mathcal{C}$ and replace this with $D$ in $\mathcal{C}'$. It is easy to see that $\mathcal{C} \prec \mathcal{C'}$. Moreover if $\mathcal{C'}$ satisfies condition (2) then $\mathcal{C} \prec \mathcal{C}'$.
Thus we show that 
\begin{equation*}
    \sum_{D \in \mathcal{C}'} r(D)^s \leq \sum_{D \in \mathcal{C}} r(D)^s \leq \epsilon
\end{equation*}
After cancelling like terms in the previous equation and using the condition in (\ref{eqn_contraction_sdim}), this is reduced to showing that 
\begin{equation*}
    r(D)^s \leq \sum_{D' \in \mathcal{C} : D' \subset D} r(D')^s \iff \sum 2^{-2ls} \leq 2^{(k-l)s}2^{-2ls} 
\end{equation*}
which is true since $l<k$.

Thus in order to prove this lemma it suffices to find such a maximal element. By Zorn's lemma, this is reduced to the task of showing that any totally ordered ascending chain has a maximal element. To that end let $(\mathcal{C}_j)_{j \in J}$ be such a chain. We claim that 
\begin{equation*}
    \mathcal{C} = \bigcap_{j \in J} \bigcup_{\mathcal{C}_{i} \succeq \mathcal{C}_{j}}  \mathcal{C}_{i}
\end{equation*}
is our required maximal element. First we note that $\mathcal{C}$ covers $X$. Indeed, if $x \in X$, let $D$ be the largest dyadic cube that contains $x$ in $\bigcup_{j \in J} \mathcal{C}_j$. If $D \in \mathcal{C}_k$ and if $\mathcal{C}_{k} \prec \mathcal{C}_{l}$, then since the dyadic square in $\mathcal{C}_{l}$ contains that dyadic square in $\mathcal{C}_k$ which contains $x$ (which is $D$). Thus $D \in \mathcal{C}$ since $D \in \bigcup_{\mathcal{C}_{i} \succeq \mathcal{C}_{j}}\mathcal{C}_i$ for all $j \in J$.

We note the following: For each $k \in \N$ we can choose $\mathcal{C}_{j}$ for some $j \in J$ such that $\mathcal{C}_{i} \cap \mathcal{D}_{2^{-k}} = \mathcal{C}_j \cap \mathcal{D}_{2^{-k}}$ whenever $\mathcal{C}_i \succeq \mathcal{C}_j$. Indeed suppose this is not true, then we can find a sequence $\{j_1, j_2,...\} \in J$ such that $\mathcal{C}_{j_{a}} \prec \mathcal{C}_{j_{b}}$ and $\mathcal{C}_{j_{a}} \cap \mathcal{D}_{2^{-k}} \subset \mathcal{C}_{j_{b}} \cap \mathcal{D}_{2^{-k}}$ whenever $a < b$ and where the containment $\subset$ is a \textit{strict} containment. However since $|\mathcal{D}_{2^{-k}}| < \infty$, it is impossible for these containments to always be strict and thus we can find $j \in J$ such that $\mathcal{C}_{i} \cap \mathcal{D}_{2^{-k}} = \mathcal{C}_j \cap \mathcal{D}_{2^{-k}}$ whenever $\mathcal{C}_i \succeq \mathcal{C}_j$. We can similarly show (as a consequence of the fact that $\{\mathcal{C}_{i}\}_{i \in J}$ is totally ordered that given any $K \in \N$ we can find $j \in J$ such that $\mathcal{C}_{i} \cap \mathcal{D}_{2^{-k}} = \mathcal{C}_j \cap \mathcal{D}_{2^{-k}}$ for $0 \leq k \leq K$ whenever $\mathcal{C}_i \succeq \mathcal{C}_j$. Fix this $j$. Then we have that (since $\mathcal{C} \subseteq \bigcup_{\mathcal{C}_{i} \succeq \mathcal{C}_{j}} \mathcal{C}_{i}$)
\begin{equation*}
    \sum_{k=0}^{K}\sum_{D \in \mathcal{C} \cap \mathcal{D}_{2^{-k}}} r(D)^s \leq \sum_{k=0}^{K}\sum_{D \in \mathcal{C}_j \cap \mathcal{D}_{2^{-k}}} r(D)^s \leq \epsilon
\end{equation*}

Sending $K \rightarrow \infty$ shows that $\mathcal{C}$ satisfies condition (3). We finally check that $\mathcal{C}_i \prec \mathcal{C}$ for $\mathcal{C}_i$. For each $x \in X$, notice that if the dyadic square in $\mathcal{C}$ that contains $x$ is the largest square containing $x$ in $\bigcup_j \mathcal{C}_j$, thus the dyadic square in $\mathcal{C}$ that contains $x$ must contain the dyadic square in $\mathcal{C}_i$ that contains $x$ for any $i$. This verifies condition (b). For condition (a) we proceed to prove this by contradiction. Suppose $\mathcal{C}$ and $\mathcal{C}_i$ are distinct collections. If we have that $\mathcal{C}_i \not\prec \mathcal{C}$, then we must have that $\mathcal{D}_{{2}^{-k}} \cap \mathcal{C}_i= \mathcal{D}_{2^{-k}} \cap \mathcal{C}$ for $k<k_0$ and $\mathcal{C}_i \cap \mathcal{D}_{2}^{-k_0} \not\subseteq \mathcal{C} \cap \mathcal{D}_{2}^{-k_0}$ for some $k_0>0$. Choose $D$ such that $D \in \mathcal{C}_i \cap \mathcal{D}_{2^{-k_0}}$ but $D \not\in \mathcal{C} \cap \mathcal{D}_{2^{-k_0}}$. Let $x \in X$ be such that $x \in D$. Then $x \in \Tilde{D}$ where $ \Tilde{D} \in \mathcal{C} \cap D_{2^{-k}}$ for some $k<k_0$ (by property (b) just established). However, this means that $\Tilde{D} \in \mathcal{C}$ (since the intersections of $\mathcal{C}$ and $\mathcal{C}'$ with dyadic squares upto scale $2^{-k_0}$ are the same.), which is impossible since $D \cap \Tilde{D} \neq \emptyset$.
\end{proof}
\textbf{Remark:} Lemma 3 in the paper requires a small extra condition to this lemma as presently stated, where we require our collection to be of the form $\{\mathcal{C}_{2^{-k}}\}_{k=k_0}^{\infty}$ for some fixed $k_0$ instead of $\{\mathcal{C}_{2^{-k}}\}_{k=1}^{\infty}$. This can clearly be done, by repeating the above procedure and noting that since $\dim(X) < s$ we have that $\mathcal{H}_{\delta}^{s}(X) = 0$ for arbitrarily small $\delta$ (where we choose $\delta=2^{-k_0}$).

\subsection{Discretization of Marstrand's Theorem}\label{subsecdiscmarstrand}
We first restate Marstrand's projection theorem.
\begin{thm}[Marstrand's projection theorem]\label{Mthm} For each $\theta\in [0,\pi]$, define $L_{\theta}=\{x\in \R^2\colon \textnormal{arg}(x)=\theta\}$ and let $p_{\theta}:\R^2\rightarrow L_{\theta}$ be the projection. Suppose $A\subset \R^2$ is a Borel set, then we have $\text{dim}(p_{\theta}(A))=\min\{1,\textnormal{dim}(A)\}$ for a.e. $\theta\in [0,\pi]$.
    
\end{thm}

The following easy lemma will be useful for this and the next subsection.

\begin{lem} Let $A\subset \R^{n}$ and $p:A\rightarrow \R^m$ be a Lipschitz function. Then $\textnormal{dim}(p(A))\leq \textnormal{dim}(A)$.
\end{lem}
\begin{proof}
It enough to show that for any $s>\textnormal{dim}(A)$, one has $\mathcal{H}^{s}(p(A))=0$. Fix such an $s$. Since $\mathcal{H}^{s}(A)=0$ and $\mathcal{H}^{s}_{\delta}(A)$ is monotone decreasing in $\delta$ , one has $\mathcal{H}^{s}_{\delta}(A)=0$ for all $\delta$. Then for every $\eps>0$ and every $\delta>0$, there exists a cover $\{U_i\}_{i\in \N}$ for $A$ with $\sum_{i}(\textnormal{diam}(U_i))^s<\eps$, and $\textnormal{diam}(U_i)\leq \delta$. That gives a cover $p(A)\subset \cup_{i} p(U_i)$ satisfying that $\textnormal{diam}(p(U_i))\leq M\textnormal{diam}(U_i)$ and
$ \sum_{i}(\textnormal{diam}\,p(U_i))^{s}\leq M^s\eps$
where $M$ is the Lipchitz constant of $p$. That implies $\mathcal{H}_{M\delta}^{s}(p(A))\leq M^s\eps$. Since $\eps>0$ is arbitrary, $\mathcal{H}_{M\delta}^{s}(p(A))=0$. We are done after sending $\delta$ to $0$.

\end{proof}

From the lemma above and since $p_{\theta}(A)\subset L_{\theta}$, $\textnormal{dim}(L_{\theta})=1$, we know that for every $\theta\in[0,\pi]$, $\textnormal{dim}(p_{\theta}(A))\leq \min\{1,\textnormal{dim}(A)\}$.

In the sequel we explain how one can go about reducing Marstrand Theorem to the discretized version stated in Proposition \ref{Mdiscrete}.

\begin{proof}[Proof that Proposition \ref{Mdiscrete} implies 
 Theorem \ref{Mthm}]

Let $\alpha=\textnormal{dim}(A)$ and $s_{A}=\min\{1,\alpha\}$. The case $s_A=0$ is trivial, so we assume $s_A>0$.

Because for any $\theta\in [0,1]$ it holds that $\textnormal{dim}(p_{\theta}(A))\leq s_A$, if we show that the exceptional set 
\begin{equation}
    E_{s_A}=\{\theta\in [0,1]\colon \textnormal{dim}(p_{\theta}(A))<s_A\}
\end{equation}
    has zero Lebesgue measure, then $\textnormal{dim}(p_{\theta}(A))=s_A$ for a.e. $\theta\in [0,1]$, and we are done.

Moreover, since for $s_n:=s_A-s_A/n$,  $E_{s_A}=\cup_{n=1}^{\infty}E_{s_n}$, it is enough to show that 
$$\mathcal{L}^1(E_{s_n})=0,\text{ for all }n\in \N.$$
Assume by contradiction that $\mathcal{L}^{1}(E_{s_n})>0$, then $\mathcal L^1$ is an $1$-Frostman measure in $E_{s_n}$ and from Lemma \ref{Lemma1paper}, for any $\delta\in (0,1)$ there exists a $\delta$-separated $(\delta,1)$-set $P\subset E_{s_n}$ of cardinality $\#P\gtrsim \delta^{-1}$. For $a<s_A$, one can reduce things to the discretization as stated in Proposition \ref{Mdiscrete} with $\#\Theta\gtrsim_{\eps} \delta^{-1+\eps}$ by following the same argument in the later discretization (see Section \ref{discretethm1}). After the discretization, the application of Proposition \ref{Mdiscrete} will give for every $\eps_0>0$ that
$$\delta^{-1+\eps}\lesssim_{\eps,s} \delta^{-1+a-s_n}$$
for some $j$ with $\delta=2^{-j}\leq \eps_0$.
Send $\eps_0$ to zero to conclude that $a-s_n\leq \eps$, for any $\eps$. So, $a\leq s_n<s_A$ which gives a contradiction when we take $a$ sufficiently close to $s_A$ ($a>s_n$ for example).
\end{proof}

\subsection{Reduction of Theorem 1 to the discretized version in Theorem 4}\label{discretethm1}

Now we return to the discretization of the main Theorem \ref{firstmainthm}. We first recall the statement of Theorem 4 in \cite{GGGHMW2022}.

\begin{thm}[Theorem 4 in \cite{GGGHMW2022}]\label{thm4paper}
    Fix $0<s<2$. For each $\eps>0$, there exists $C_{s,\eps}>0$ so that the following holds. Let $\delta>0$. Let $H\sub B^3(0,1)$ be a union of $h$ many disjoint $\delta$-balls. Let $\Theta$ be a $\delta$-separated $(\delta,t)$-subset of $[0,1]$ of cardinality $\#\Theta\gtrapprox \delta^{-t}$, such that for each $\theta\in \Theta$, we have a collection of $\delta\times \delta\times 1$-tubes $\mathbb T_\theta$ pointing at the direction $\gamma(\theta)$ and satisfying the following conditions:
    \begin{enumerate}        
        \item $\mathbb T_\theta$ satisfies the $s$-dimensional condition: for any $r\in [\delta,1]$ and any ball $B_r\sub \R^3$ of radius $r$, we have
        \begin{equation}
            \#\{T\in \mathbb T_\theta:T\cap B_r\ne \varnothing\}\lesssim (r/\delta)^{s}.
        \end{equation}
        In particular, $\#\mathbb T_\theta\lesssim \delta^{-s}$.
        \item Each $\delta$-ball contained in $H$ intersects $\gtrapprox\#\Theta $ many tubes from $\cup_{\theta\in \Theta} \mathbb T_\theta$.
    \end{enumerate}
    Then
        \begin{equation}
            \delta^{-t}h\le C_{s,\eps}\delta^{-1-s-\eps}.
        \end{equation}
\end{thm}

Let us now discuss the details of how Theorem \ref{thm4paper} implies Theorem \ref{firstmainthm}.
Assume $A\subset B^3(0,1)$ and let $ 0<s<\min\{2,\alpha\}$ where $\alpha=\textnormal{dim}(A)$. 
Recall the definition of the exceptional set 
$$E_s=\{\theta\in[0,1]\colon \textnormal{dim}(\pi_{\theta}(A))<s\}.$$

Observe that can reduce ourselves to checking that Theorem \ref{thm4paper} implies Theorem \ref{firstmainthm} for  $s>\max\{0,\alpha-1\}$ because for the remaining case $\alpha>1$ and $s\leq \alpha-1$ what Theorem \ref{firstmainthm} claims is that $\text{dim}(E_{s})=0$, and that will follow from 
$\text{dim}(E_s)\leq \text{dim}(E_{\alpha-1+\eps})\leq \eps$, for all $\eps>0$.

Fix $a<\text{dim}(A)$, and $t<\text{dim}(E_s)$.
We need to check that
\begin{equation}
 a+t\leq 1+s   
\end{equation}
then sending $a\rightarrow \text{dim}(A)$ and $t\rightarrow \text{dim}(E_s)$ we get Theorem \ref{firstmainthm}.

Let $\nu_A$ be an $a$-Frostman measure in $A$ and $\nu_{E_s}$ a $t$-Frostman measure in $E_s$.

For any $\theta$ in the exceptional set $E_s$, Lemma \ref{Lemma2paper} guarantees one can find a nice cover $\mathbb{D}_{\theta}=\{D\}$ for $\pi_{\theta}(A)\subset \gamma(\theta)^{\perp}\cong \R^2$ comprised of dyadic cubes with small associated $s$-dimensional content, in the sense that $\sum_{D\in \mathbb{D}_{\theta}}r(D)^s<1$, and the extra $s$-dimensional condition for each generation. That is, for any $j\in\N$ and any ball $B_r^2\subset \gamma(\theta)^{\perp}$, with $2^{-j}\leq r\leq 1$,
\begin{equation}
    \#\{D\in \mathbb{D}_{\theta,j}\colon D\subset B_r^2\}\lesssim \left(\frac{r}{2^{-j}}\right)^{s}
\end{equation}
where $\mathbb{D}_{\theta,j}=\{D\in \mathbb{D}_{\theta}\colon r(D)=2^{-j}\}$.

Let $\mathbb{T}_{\theta}:=\{T=\pi_{\theta}^{-1}(D)\cap B^{3}(0,1)\colon D\in \mathbb{D}_{\theta}\}$. Split $\mathbb{T}_{\theta}$ in generations $\mathbb{T}_{\theta}=\cup_{j} \mathbb{T}_{\theta,j}$, where for each $j$, $\mathbb{T}_{\theta,j}$ corresponds to the pre-images by ${\pi}_{\theta}$ of dyadic cubes $D\in \mathbb{D}_{\theta,j}$.
Similar condition is then inherited by the tubes in the collection $\mathbb{T}_{\theta,j}$. Namely, 
\begin{equation}\label{sdimforthm4}
   \#\{T\in \mathbb{T}_{\theta,j}\colon T\cap B^{3}(x,r)\neq \emptyset\}\lesssim \left(\frac{r}{2^{-j}}\right)^{s}\text{ for }2^{-j}\leq r\leq 1. 
\end{equation}

Given $\eps_0>0$, as remarked right after Lemma \ref{Lemma2paper} we can assume that we only use sufficiently small cubes, that is, $2^{-j}< \eps_{0}$. Observe that for each $j>|\log_2(\eps_0)|$ the collection $\mathbb{D}_{\theta,j}=\{D\in \mathbb{D}_{\theta}\colon r(D)=2^{-j}\}$ is pairwise disjoint and the same is then true for the $2^{-j}\times 2^{-j}\times 1$ tubes in $\mathbb{T}_{\theta,j}$. Also, by construction the tubes in $\mathbb{T}_{\theta}$ point in the direction $\gamma(\theta)$ and $A\subset \cup_{j> |\log_2(\eps_0)|}\cup_{T\in \mathbb{T}_{\theta,j}}T$.

Here is an overview of the discretization argument:
Recall we fixed $\nu_A$ $a$-Frostman measure in $A$ and $\nu_{E_s}$ a $t$-Frostman measure in $E_s$.

The key will be to find a suitable generation $j$ (among the $j$ such that $2^{-j}< \eps_0$) for which there is a positive $a$-dimensional content worth of points in $A$, say $A_g\subset A$ with $\nu_{A}(A_g)\gtrsim \frac{1}{j^2}>0$, which are covered by $\cup_{T\in \mathbb{T}_{\theta,j}}T$ for a big fraction of the angles $\theta$ in a suitable $\delta$-separated set $\Theta$ of angles in $E_s$. $\Theta$ will show up by application of Lemma \ref{Lemma1paper} for a subset $E_{s,j}\subset E_s$ to be defined later satisfying $\nu_{E_s}(E_{s,j})\gtrsim \frac{1}{j^2}>0$, and we will know that $\#\Theta\gtrsim \frac{1}{j^2}2^{jt}$.
By applying Lemma \ref{Lemma1paper} to the set $A_{g}$ we get a $2^{-j}$-separated set of cardinality $h$ with $h\gtrsim \nu_{A}(A_g)(2^{-j})^{-a}\gtrsim \frac{1}{j^2}2^{aj}$, and we use these collection of $2^{-j}$-separated points as centers of the $2^{-j}$ balls that will comprise the $H$ that we need for application of Theorem \ref{thm4paper}. That theorem will then give
\begin{equation}\label{eqn_after_applying_Thm_4}
    (2^{-j})^{-t}h\leq C_{s,\eps} (2^{-j})^{-1-s-\eps}
\end{equation}
which will turn into 
$\frac{1}{j^2}2^{(a+t)j}\leq C_{s,\eps} 2^{(s+1+\eps)j}$.

Sending $\eps_0 $ to zero (so $j$ to infinity) we get $a+t\leq 1+s+\eps$, for all $\eps>0$. Therefore, $a+t\leq 1+s$.

We now give the details for the ingredients used in the overview above. For each $\theta$, since $A\subset \cup_{j>|\log_{2}\eps_0|} (\cup_{T\in\mathbb{T}_{\theta,j}}T)$, by pigeonholing, and by sub-additivity of the measure $\nu_A$, there exists a $j(\theta)>|\log_2(\eps_0)|$, such that 

\begin{equation}\label{goodj}
    \nu_A(A\cap (\cup_{T\in \mathbb{T}_{\theta,j(\theta)}}T))\geq \frac{1}{10j(\theta)^2}\nu_A(A)=\frac{1}{10 j(\theta)^2}.
\end{equation}

Define for each $j>|\log_2(\eps_0)|$
$$E_{s,j}=\{\theta\in E_s\colon j(\theta)=j\}.$$
This gives a partition $E_s=\sqcup_{j>|\log_2\eps_0|} E_{s,j}$. Again by pigeonholing there exists $j>|\log_2\eps_0|$ with 
$$\nu_{E_s}(E_{s,j})\gtrsim \frac{1}{j^2}.$$

Set $\delta:=2^{-j}$ for this suitable fixed $j$.
An application of Lemma \ref{Lemma1paper} guarantees the existence of a $\delta$-separated set of angles $\Theta\subset E_{s,j}$ such that $\Theta$ is a $(\delta,t)$-set and $\#\Theta\gtrsim (\log\delta^{-1})^{-2}\delta^{-t}$.

Consider
\begin{equation}
    S:=\{(x,\theta)\in A\times \Theta\colon x\in \cup_{T\in \mathbb{T}_{\theta,j}}T\}
\end{equation}
and the sections 
$$S_x=\{\theta\in \Theta\colon (x,\theta)\in S\},\,S_{\theta}=\{x\in A\colon (x,\theta)\in S\}.$$
In words, $S_x$ is comprised of all the angles $\theta\in \Theta$ such that $x$ is covered by $\cup_{T\in \mathbb{T}_{\theta,j}}T$, while $S_{\theta}$ is comprised of the points in $A$ covered by union of tubes in $\mathbb{T}_{\theta,j}$.

Let $\mu$ be the counting measure in $\Theta$. 

By inequality (\ref{goodj}) and since $\Theta\subset E_{s,j}$, we know that 
\begin{equation}
    \begin{split}
        (\nu_A\times \mu)(S)=&\int_{\Theta}\left(\int_{A}\chi_{S}(x,\theta)d\nu_A(x)\right) d\mu(\theta)\\
=&\int_{\Theta}\nu_A(A\cap \cup_{T\in \mathbb{T}_{\theta,j}} T) d\mu(\theta)\\
\geq & \int_{\Theta}\frac{1}{10j^2}d\mu(\theta)= \frac{1}{10j^2}\mu(\Theta).
    \end{split}
\end{equation}

Observe that by Fubini,
\begin{equation}
    \begin{split}
(\nu_A\times \mu)&(\{(x,\theta)\in S\colon \mu(S_x)\leq  \frac{1}{20j^2}\mu(\Theta)\})\\
        = &\int_{A}\chi_{\mu(S_x)\leq \frac{1}{20j^2}\mu(\Theta)}(x)\left(\int_{\Theta} \chi_{S}(x,\theta)d\mu(\theta)\right)d\nu_A(x)\\
        = &\int_{A}\chi_{\mu(S_x)\leq \frac{1}{20j^2}\mu(\Theta)}(x)\mu(S_x)d\nu_A(x)\\
        &\leq \frac{1}{20j^2}\nu_A(A)\mu(\Theta)=\frac{1}{20j^2}\mu(\Theta).
    \end{split}
\end{equation}
Therefore, 
\begin{equation}\label{goodproductmeasure}
    \begin{split}
(\nu_A\times \mu)&(\{(x,\theta)\in S\colon \mu(S_x)>  \frac{1}{20j^2}\mu(\Theta)\})\\
\geq&(\nu_A\times \mu)(S)-(\nu_A\times \mu)(\{(x,\theta)\in S\colon \mu(S_x)\leq  \frac{1}{20j^2}\mu(\Theta)\})\\
\geq & \left(\frac{1}{10j^2}-\frac{1}{20j^2}\right)\mu(\Theta)=\frac{1}{20j^2}\mu(\Theta).
    \end{split}
\end{equation}

Take $A_g=\{x\in A\colon \mu(S_x)> \frac{1}{20j^2}\mu(\Theta)\}$, and notice that 
\begin{align*}
     (\nu_A\times \mu)(\{(x,\theta)\in S\colon \mu(S_x)>  \frac{1}{20j^2}\mu(\Theta)\})=&(\nu_A\times \mu)(\{(x,\theta)\in S\colon x\in A_g\})\\
    \leq& \mu(\Theta)\nu_A(A_g).
\end{align*}
   Hence, inequality (\ref{goodproductmeasure}) implies 
$$\nu_A(A_g)\geq \frac{1}{20j^2}.$$
Then, since
 $$\nu_{A}(A_g)\gtrsim \frac{1}{j^2}=(\log\delta^{-1})^{-2},$$
an application of Lemma \ref{Lemma1paper} gives a $\delta$-separated set $C_g\subset A_g$ which is a $(\delta,a)$-set and cardinality $\gtrsim (\log\delta^{-1})^{-2}\delta^{-a}$.

Define $H$ as the disjoint union of the $\delta$ balls centered at the points in $C_g$.

Observe that $H$ and $\Theta$ will satisfy the hypothesis of Theorem \ref{thm4paper}. Indeed, the $s$-dimensional condition was obtained by construction in (\ref{sdimforthm4}) and the center of which $2^{-j}$-ball contained in $H$ lives in $A_g$, so it intersects at least $ \frac{1}{20j^2}\#\Theta$ many tubes in $\cup_{\theta\in \Theta}\mathbb{T}_{\theta,j}$. Applying Theorem \ref{thm4paper}, we thus arrived at the desired inequality \eqref{eqn_after_applying_Thm_4} stated before.

\section{Proof of Theorem 5}

\subsection{The decoupling machinery}

Let $\gamma$ be a non-degenerate curve as discussed in the introduction. We first state in full detail the main decoupling theorem we will be using here. For completeness, the proof will be given in Section \ref{sec_decoupling_C2_cone}.

\begin{thm}[Bourgain-Demeter decoupling for the cone]\label{thm_Bourgain-Demeter_decoupling}
    Let $\delta\in 4^{-\N}$ and $\Omega$ be the $\delta^{1/2}$-lattice in $[0,1]$. Given $C,K\ge 1$, consider the collection of planks $\{A_{\omega,\delta^{1/2}}:\omega\in \Omega\}$, where
    \begin{equation}\label{eqn_A_plank}
        A_{\omega,\delta^{1/2}}:=\left\{\sum_{i=1}^3 \xi_i \mathbf e_i(\omega):|\xi_1|\le C\delta,|\xi_2|\le C\delta^{1/2},C^{-1}K^{-1}\le \xi_3\le C\right\}.
    \end{equation}
    Then we have
    \begin{enumerate}
        \item \label{item_BD_decoupling_01} If $C$ is large enough (depending on $\gamma$ only), the collection $\{A_{\omega,\delta^{1/2}}:\theta\in \Omega\}$ covers the $\delta$-neighbourhood of the truncated cone
        \begin{equation}
            \Gamma_{K^{-1}}:=\{r\mathbf e_3(\omega):\omega\in [0,1],r\in [K^{-1},1]\}.
        \end{equation}
        Moreover, the collection $\{A_{\omega,\delta^{1/2}}:\omega\in \Omega\}$ is at most $C'K^{2}$-overlapping, where $C'$ depends on $\gamma,C$ but not $K,\delta$.
        \item \label{item_BD_decoupling_02} The following Fourier decoupling inequality holds: for each $2\le p\le 6$, every family of Schwartz functions $\{f_\omega:\omega\in \Omega\}$ where each $f_\omega$ has Fourier support in $A_{\omega,\delta^{1/2}}$, we have
        \begin{equation}
            \left\|\sum_{\omega\in \Omega}f_\omega\right\|_{L^p(\R^3)}\leq C_{\eps}K^{10}\delta^{-\eps}\norm{\norm{f_\omega}_{L^p(\R^3)}}_{\ell^2(\omega\in \Omega)},\quad \forall \eps>0,
        \end{equation}
        where the constant $C_\eps$ depends on $\eps,C,p$ but not $K,\delta$ or the test functions $f_\omega$.
    \end{enumerate}     
\end{thm}

We will also use the following flat decoupling inequality. To avoid technical issues, we just state a version that is sufficient for our use here. We include in the appendix a proof of this inequality.

\begin{thm}[Flat decoupling]\label{thm_flat_decoupling}
    In $\R^3$, let $\mathbb T$ be a finite collection of tubes with possibly different dimensions and orientations. Let $C\ge 1$ and assume the collection of $C$-concentrically enlarged tubes $\{CT:T\in \mathbb T\}$ is at most $K$-overlapping. Then we have the following flat decoupling inequality: for every $2\le p\le \infty$, every family of Schwartz functions $\{f_T:T\in \mathbb T\}$ where each $f_T$ has Fourier support in $T$, we have
    \begin{equation}
        \left\|\sum_{T\in \mathbb T}f_T\right\|_{L^p(\R^3)}\leq C'K(\#\mathbb T)^{\frac 1 2-\frac 1 p}\norm{\norm{f_T}_{L^p(\R^3)}}_{\ell^2(T\in \mathbb T)},
    \end{equation}
    where $C'$ depends on $C$ but not $K$, $\#\mathbb T$, the dimensions of the tubes, or the test functions $f_T$.
\end{thm}
{\it Remark.} Flat decoupling is sharp when $\mathbb T$ is a tiling of a larger tube by congruent parallel smaller tubes. See, for instance, Section \cite{Yang_thesis} for a proof.

Now we restate Theorem 5 in \cite{GGGHMW2022}, which is a rescaled version of Theorem \ref{thm4paper}, more convenient for the application of the Decoupling theorems later.
\begin{thm}[Theorem 5]\label{thm_thm5}
    Fix $0<s<2$. For each $\eps>0$, there exists $C_{s,\eps}>0$ so that the following holds. Let $\delta\in (0,1]$.
    
    Let $H\sub B^3(0,\delta^{-1})$ be a union of $\delta^{-a}$ many disjoint unit balls so that $H$ has measure $|H|\sim \delta^{-a}$. Let $\Theta$ be a $\delta$-separated $(\delta,t)$-subset of $[0,1]$ of cardinality $\#\Theta\gtrapprox \delta^{-t}$, such that for each $\theta\in \Theta$, we have a collection of $1\times 1\times \delta^{-1}$-tubes $\mathbb T_\theta$ pointing at the direction $\gamma(\theta)$ and satisfying the following conditions:
    \begin{enumerate}        
        \item $\mathbb T_\theta$ satisfies the $s$-dimensional condition: for any $r\in [1,\delta^{-1}]$ and any ball $B_r\sub \R^3$ of radius $r$, we have
        \begin{equation}\label{eqn_s_dim_3D}
            \#\{T\in \mathbb T_\theta:T\cap B_r\ne \varnothing\}\lesssim r^s.
        \end{equation}
        In particular, $\#\mathbb T_\theta\lesssim \delta^{-s}$.
        \item \label{item_2_thm5} Each unit ball contained in $H$ intersects $\gtrapprox\#\Theta $ many tubes from $\cup_\theta \mathbb T_\theta$.
    \end{enumerate}
    Then
        \begin{equation}
            \delta^{-t-a}\le C_{s,\eps}\delta^{-1-s-\eps}.
        \end{equation}
\end{thm}

The proof structure will be very similar to that of Proposition \ref{Mdiscrete}.

We remark that the decoupling theorem follows for every scale $\delta\in (0,1]$. For $\delta\notin 4^{-\N}$, we should replace the tiling of $[0,1]$ by $\delta^{1/2}$ intervals by a slightly more general subset, namely, a maximally $\delta^{1/2}$-separated subset of $[0,1]$. Once we can show decoupling for all $\delta\in 4^{-\N}$, the slightly more general result follows from a trivial tiling argument, triangle and Cauchy-Schwarz, where we lose at most an absolute constant.

    \subsection{Construction of wave packets}\label{sec_step_1_3D}    

    Fix $\theta\in \Theta$. We define a frequency slab $P_\theta$ centred at the origin that has dimensions $1\times 1\times \delta$, and its shortest direction is parallel to $\mathbf e_1(\theta)$. That is, $P_\theta$ is dual to every $T_\theta\in \mathbb T_\theta$. We may also assume without loss of generality that $\mathbb T_\theta$ is a subcollection of the lattice that tiles $[-10\delta^{-1},10\delta^{-1}]^3$. 
    
    For each $T_\theta\in \mathbb T_\theta$, define a Schwartz function $\psi_{T_\theta}:\R^3\to \R$ Fourier supported on $P_\theta$ such that $\psi_{T_\theta}\gtrsim 1$ on $T_\theta$ and decays rapidly outside $T_\theta$. With this, we define
    \begin{equation}
        f_\theta=\sum_{T_\theta\in \mathbb T_\theta}\psi_{T_\theta},\quad f=\sum_{\theta\in \Theta}f_\theta.
    \end{equation}
    Thus $f_\theta$ is also Fourier supported on $P_\theta$, but physically essentially supported on $\cup \mathbb T_\theta$. The function $f$ is Fourier supported on $\cup_{\theta\in \Theta}P_\theta$, which is a union of slabs of the same size centered at $0$, but with normal orientations prescribed by $\mathbf e_1(\theta)$.

    Thus by definition, for every $x\in H$ we have
    \begin{equation}\label{eqn_f_bound_3D}
       \#\Theta \lessapprox f(x)\lesssim \#\Theta.
    \end{equation}
    The first inequality is true by Hypothesis \ref{item_2_thm5} for $\mathbb T_\theta$ in Theorem \ref{thm_thm5} and the second inequality follows from $f_{\theta}\lesssim 1$ (essentially by the disjointedness of the tubes in $\mathbb{T}_{\theta}$). In particular, for every $p<\infty$,
    
    \begin{equation}\label{eqn_lower_bound_f_3D}
        \int_H |f(x)|^p\gtrapprox |H|(\#\Theta)^p\sim \delta^{-a}(\#\Theta)^p.
    \end{equation}
    \subsection{High-low decomposition}

    Similar to the proof of Proposition \ref{Mdiscrete}, the high-low decomposition is based on an elementary fact in Euclidean geometry, which requires the nondenegenacy condition on $\mathbf e_1(\theta)=\gamma(\theta)$. Before we state the geometric fact, let us first make the following decomposition. Let $K\in [1,\delta^{-1}]$ to be determined (see \eqref{eqn_defn_K_3D} below). Decompose $P_\theta$ into:
    \begin{enumerate}
    \item The high part:
        \begin{equation}            P_{\theta,\mathrm{high}}:=\left\{\sum_{i=1}^3 \xi_i \mathbf e_i(\theta):|\xi_1|\le \delta,K^{-1}\le |\xi_2|\le 1,|\xi_3|\le 1\right\}.
        \end{equation}        
        \item The low part:
        \begin{equation}            P_{\theta,\mathrm{low}}:=\left\{\sum_{i=1}^3 \xi_i \mathbf e_i(\theta):|\xi_1|\le \delta,|\xi_2|\le K^{-1},|\xi_3|\le K^{-1}\right\}.
        \end{equation}        
        \item The $\delta^{1/2}$-mixed part:
        \begin{equation}            P_{\theta,\delta^{1/2}}:=\left\{\sum_{i=1}^3 \xi_i \mathbf e_i(\theta):|\xi_1|\le \delta,|\xi_2|\le \delta^{1/2},K^{-1}\le|\xi_3|\le 1\right\}.
        \end{equation}
        \item For dyadic numbers $\lambda\in (\delta^{1/2},K^{-1}]$, the $\lambda$-mixed part:
        \begin{equation}            P_{\theta,\lambda}:=\left\{\sum_{i=1}^3 \xi_i \mathbf e_i(\theta):|\xi_1|\le \delta,\lambda/2\le |\xi_2|\le \lambda,K^{-1}\le|\xi_3|\le 1\right\}.
        \end{equation}       
    \end{enumerate}
    Thus for each $\theta$,
    \begin{equation}\label{eqn_P_theta_decomposition}
        P_\theta=P_{\theta,\mathrm{high}}\sqcup P_{\theta,\mathrm{low}}\sqcup (\sqcup_{\text{dyadic } \lambda\in [\delta^{1/2},K^{-1}]} P_{\theta,\lambda}).
    \end{equation}
    
    We choose a smooth partition of unity adapted to this partition which we denote by $\eta_{\theta,\mathrm{high}}$, $\eta_{\theta,\mathrm{low}}$, $\eta_{\theta,\mathrm{\lambda}}$, respectively, so that
    \begin{equation}
        \eta_{\theta,\mathrm{high}}+\eta_{\theta,\mathrm{low}}+\sum_{\text{dyadic } \lambda\in [\delta^{1/2},K^{-1}]}\eta_{\theta,\lambda}=1
    \end{equation}
    on $P_\theta$. Since $\mathrm{supp}\widehat {f_\theta}\sub P_\theta$, we also obtain a decomposition of $f_\theta$
    \begin{equation}
        f_\theta=f_{\theta,\mathrm{high}}+f_{\theta,\mathrm{low}}+\sum_{\text{dyadic } \lambda\in [\delta^{1/2},K^{-1}]}f_{\theta,\lambda},
    \end{equation}
    where $\widehat f_{\theta,\mathrm{high}}=\eta_{\theta,\mathrm{high}} \widehat {f_\theta}$, $\widehat f_{\theta,\mathrm{low}}=\eta_{\theta,\mathrm{low}} \widehat {f_\theta}$, and $\widehat f_{\theta,\lambda}=\eta_{\theta,\lambda} \widehat {f_\theta}$. Similarly, we have a decomposition of $f$
    \begin{equation}\label{eqn_f_decomposition}
        f=f_{\mathrm{high}}+f_{\mathrm{low}}+\sum_{\text{dyadic } \lambda\in [\delta^{1/2},K^{-1}]}f_{\lambda},
    \end{equation}
    where $f_{\mathrm{high}}=\sum_{\theta\in \Theta}f_{\theta,\mathrm{high}}$, $f_{\mathrm{low}}=\sum_{\theta\in \Theta}f_{\theta,\mathrm{low}}$ and $f_\lambda=\sum_{\theta\in \Theta} f_{\theta,\lambda}$. 

    Lastly, in view of \eqref{eqn_lower_bound_f_3D} and using triangle inequality, we have
    \begin{equation}
        \int_H |f_{\mathrm{high}}|^p+\int_H |f_{\mathrm{low}}|^p+\sum_{\text{dyadic } \lambda\in [\delta^{1/2},K^{-1}]}\int_H |f_{\lambda}|^p\gtrapprox \delta^{-a}(\#\Theta)^p.
    \end{equation}
    
    \subsection{Analysis of high part}

    The geometric fact that gives orthogonality for the high part is as follows.    \begin{lem}\label{lem_high_orthogonality_3D}
        The collection of all high parts $\{P_{\theta,high}\}_{\theta\in \Theta}$  is at most $O(K^{2})$-overlapping.
    \end{lem}

    \begin{proof}[Proof of lemma]
    We first perform a preliminary reduction. Partition $\Theta$ into intervals of length $\ll K^{-1}$. By losing a factor of order $O(K)$, it suffices to assume $\Theta$ is contained in such an interval of length $\ll K^{-1}$.

        We may ignore the thin direction $\mathbf e_1$. More precisely, consider the plane
        \begin{equation}            \Pi_{\theta}:=\left\{\xi_2 \mathbf e_2(\theta)+\xi_3 \mathbf e_3(\theta):K^{-1}\le |\xi_2|\le 1,|\xi_3|\le 1\right\},
        \end{equation}
        whose $\delta$ neighbourhood contains $P_{\theta,\mathrm{high}}$. We claim the following: for any $\theta\in [0,1]$ and any $\theta'=\theta+\Delta\in [0,1]$ with $CK\delta\le \Delta\le C^{-1}K^{-1}$ 
        ($C$ to be determined at the end of the proof of this lemma), if $\xi\in \Pi_{\theta}$, then
        \begin{equation}\label{eqn_claim_separation}
            \mathrm{dist}(\xi,\Pi_{\theta'})>10\delta.
        \end{equation}

        Observe that if this holds then $P_{\theta,\mathrm{high}}\cap P_{\theta',\mathrm{high}}=\varnothing$ for $\Delta\in[CK\delta, C^{-1}K^{-1}]$. Indeed, if $\xi\in P_{\theta,\mathrm{high}}\cap P_{\theta',\mathrm{high}}$, then there exists $\xi_1\in \Pi_{\theta}$ and $\xi_2\in \Pi_{\theta'}$ such that $|\xi-\xi_1|\le \delta$, $|\xi-\xi_2|\le \delta$. Since $\xi_1\in \Pi_{\theta}$, we have $\mathrm{dist}(\xi_1,\Pi_{\theta'})>10\delta$. Thus $\mathrm{dist}(\xi,\Pi_{\theta'})>9\delta$. But $|\xi-\xi_2|\le \delta$, so $\mathrm{dist}(\xi_2,\Pi_{\theta'})>8\delta$, which contradicts the fact that $\xi_2\in \Pi_{\theta'}$. 
        
        Thus if $P_{\theta,\mathrm{high}}\cap P_{\theta',\mathrm{high}}\neq \varnothing$ then $\Delta\notin [CK\delta,C^{-1}K^{-1}]$. By the preliminary reduction at the beginning we see that if $P_{\theta,\mathrm{high}}$ and $P_{\theta+\Delta,\mathrm{high}}$ overlap then we must have $\Delta\le CK\delta$. Since $\Theta$ is $\delta$-separated, we see that $\{P_{\theta,\mathrm{high}}\}_{\theta\in \Theta}$ is at most $O(K)$ overlapping. Thus it remains to prove \eqref{eqn_claim_separation}.

        Let $\xi\in \Pi_{\theta}$. Write 
        \begin{equation}
            \xi=a\mathbf e_2(\theta)+b\mathbf e_3(\theta)=a\gamma'(\theta)+b\gamma(\theta)\times \gamma'(\theta),
        \end{equation}
        where $|a|\in [K^{-1},1]$ and $|b|\le 1$. Since the normal direction of $\Pi_{\theta+\Delta}$ is $\gamma(\theta+\Delta)$, we have $\mathrm{dist}(\xi,\Pi_{\theta'})=\xi\cdot \gamma(\theta')$, and so we need to show \begin{equation}\label{eqn_37_paper}
            \left|\gamma(\theta+\Delta)\cdot \left(a\gamma'(\theta)+b\gamma(\theta)\times \gamma'(\theta)\right)\right|> 10\delta.
        \end{equation}
        By Taylor expansion, we have $\gamma(\theta+\Delta)=\gamma(\theta)+\Delta \gamma'(\theta)+O(\Delta^2)$. Using $|a|\ge K^{-1}$, we see the left hand side of \eqref{eqn_37_paper} is 
        \begin{equation}
            |a|\Delta +O(\Delta^2)\ge (|a|-O(\Delta))\Delta\ge (|a|-O(C^{-1}K^{-1}))\Delta> 10\delta,
        \end{equation}
        if $C$ is large enough.
    \end{proof}
    Now we are ready to estimate $\int_H |f_{\mathrm{high}}|^p$. By Plancherel's identity and Lemma \ref{lem_high_orthogonality_3D}, we have
\begin{equation}
    \int |f_{\mathrm{high}}|^2=\int |\widehat f_{\mathrm{high}}|^2\lesssim K^2\sum_{\theta\in \Theta}\int |\widehat{f_\theta}|^2= K^2\sum_{\theta\in \Theta}\int |f_\theta|^2.
\end{equation}
By the rapid decay of $\psi_{T_\theta}$, we have 
\begin{equation}
    \int |f_\theta|^2\sim \delta^{-1}\#\mathbb T_\theta.
\end{equation}
Thus
\begin{equation}\label{eqn_high_upper_bound_3D}
    \int |f_{\mathrm{high}}|^2\lesssim K^2\delta^{-1}\sum_{\theta\in \Theta}\#\mathbb T_\theta.
\end{equation}
Thus, for $p\in [2,\infty)$, we have
\begin{equation}\label{eqn_high_contribution_3D_thm7}
    \int_H |f_{\mathrm{high}}|^p\le \int_H |f_{\mathrm{high}}|^2 (\#\Theta)^{p-2}\lesssim (\#\Theta)^{p-2}K^2 \delta^{-1}\sum_{\theta\in \Theta}\#\mathbb T_\theta.
\end{equation}
Using $\#\mathbb T_\theta\lesssim \delta^{-s}$, we have
\begin{equation}\label{eqn_high_contribution_3D}
    \int_H |f_{\mathrm{high}}|^p\lesssim (\#\Theta)^{p-1}K^2 \delta^{-1-s}.
\end{equation}

\subsection{Analysis of low part}\label{sec_step_4_3D}

For the low part, we do not have orthogonality to aid us. Instead, we use the $s$-dimensional condition \eqref{eqn_s_dim_3D}. Heuristically, for each $x\in H$, $\theta\in \Theta$ and $T_\theta\in \mathbb T_\theta$, if we ignore the tails of all Schwartz functions, then
\begin{equation}
    |\psi_{T_\theta}*\eta_{\theta,\mathrm{low}}^\vee(x)|\lesssim 
    \begin{cases}
        \delta K^{-2}\cdot\delta^{-1}, &\text{if }T_\theta\cap B^3(x,CK)\neq \varnothing,\\
        0,&\text{if }T_\theta\cap B^3(x,CK)= \varnothing.
    \end{cases}   
\end{equation}
As a result,
\begin{equation}\label{eqn_blur_scale_K_3D}
    |f_{\theta,\mathrm{low}}(x)|\lesssim  K^{-2}\#\{T_\theta\in\mathbb T_\theta: T_\theta\cap B^3(x,CK)\neq \varnothing\}.
\end{equation}
Intuitively, this follows from the uncertainty principle, since we now view the tubes $T_\theta$ blurred by scale $K$. For heuristic purposes, one can pretend $\check{\eta}_{\theta,low}=\frac{1}{\delta^{-1}K^2}\mathbb{1}_{\tilde{T}}$ where $\tilde{T}$ is a $\delta^{-1}\times K\times K$ tube with same orientation as $T_{\theta}$, and the convolution $\psi_{T_{\theta}}*\mathbb{1}_{\tilde{T}}(x)$ can capture at most the volume of a $\delta^{-1}\times 1\times 1$ tube. 
The rigorous proof is left as an exercise, since it is similar to the argument in Section \ref{sec_step_4}, except that the treatment of weights in 2D should be generalised to 3D in the natural way. (In particular, note that the $K^{s-2}$ in \eqref{eqn_low_part_upper_bound_3D} below would be similarly replaced by $K^{\frac {s-2}2}$ with a suitably chosen exponent $E=E(s)$ for the weights. This will also slightly affect our choice of $K$ in \eqref{eqn_defn_K_3D} below.)

Using \eqref{eqn_blur_scale_K_3D} and the $s$-dimensional condition \eqref{eqn_s_dim_3D}, for every $x\in H$ we have
\begin{equation}\label{eqn_low_part_upper_bound_3D}
    |f_{\mathrm{low}}(x)|\leq \sum_{\theta\in \Theta}|f_{\theta,\mathrm{low}}(x)|\lesssim \# \Theta K^{-2} K^s\sim \# \Theta K^{s-2}.
\end{equation}

\subsection{Choosing \texorpdfstring{$K$}{Lg} such that the low part is negligible}

Recall from \eqref{eqn_f_bound_3D} that $|f(x)|\gtrapprox \#\Theta$, that is, $|f(x)|\ge |\log \delta|^{-\alpha}\#\Theta$ for some $\alpha=O(1)$. We now choose 
\begin{equation}\label{eqn_defn_K_3D}
    K:=|\log \delta|^{\frac {\alpha+1} {2-s}},
\end{equation}
so that for $\delta$ small enough we have for every $x\in H$ that $|f_{\mathrm{low}}(x)|<|f(x)|/2$. Thus we need only consider the contributions from the high part and the mixed parts.

\subsection{Analysis of \texorpdfstring{$\delta^{1/2}$}{Lg}-mixed part}\label{sec_step_6_3D}

The main tool for the analysis of mixed parts is the Bourgain-Demeter $\ell^2(L^6)$ decoupling for the cone, stated as in Theorem \ref{thm_Bourgain-Demeter_decoupling}. We further divide this section into a few subsections.

\subsubsection{Construction of covering planks}

To use Theorem \ref{thm_Bourgain-Demeter_decoupling}, the idea is to construct a cover of the truncated cone
    \begin{equation}
        \Gamma_{K^{-1}}:=\{r\mathbf e_3(\theta):K^{-1}\le r\le 1,0\le \theta\le 1\}
    \end{equation}
    by essentially disjoint slabs $P^+_{\theta,\delta^{1/2}}$, where $P^+_{\theta,\delta^{1/2}}:=P_{\theta,\delta^{1/2}}\cap \{\xi_3>0\}$, and we define $P^-_{\theta,\delta^{1/2}}$ in a similar way. In the following we make this precise.   

    For each $C\ge 1$, denote the slightly enlarged plank
    \begin{equation}    CP^+_{\theta,\delta^{1/2}}:=\left\{\sum_{i=1}^3 \xi_i \mathbf{e}_i(\theta):|\xi_1|\le C\delta,\,\,|\xi_2|\le C\delta^{1/2},\,\,C^{-1}K^{-1}\le \xi_3\le C\right\}.
    \end{equation}

    {\it Remark on notation.} For simplicity, from now on we abuse notation and simply denote $CP^+_{\theta,\delta^{1/2}}$ by $CP_{\theta,\delta^{1/2}}$.

    Strictly speaking, the object we need to cover is not $\Gamma_{K^{-1}}$, but rather $\cup_{\theta\in \Theta}P_{\theta,\delta^{1/2}}$. For this purpose, we need the following lemma.

\begin{lem}\label{lem_same_and_distinct}
We have the following observations.
\begin{enumerate}   
    \item \label{item_1_same_and_distinct} 
If $|\theta-\theta'|\leq \delta^{1/2}$, then $P_{\theta,\delta^{1/2}}$ and $P_{\theta',\delta^{1/2}}$ are essentially the same in the sense that $P_{\theta,\delta^{1/2}}\sub C_1 P_{\theta',\delta^{1/2}}$ for some suitable constant $C_1$. 
    \item \label{item_2_same_and_distinct} If $C_2K\delta^{1/2}\le |\theta-\theta'|\le C_2^{-1}$ where $C_2=C_2(C_1)$ is large enough, then $C_1P_{\theta,\delta^{1/2}}$ and $C_1P_{\theta',\delta^{1/2}}$ are disjoint.
  
\end{enumerate}
\end{lem}
The proof is given in Section \ref{sec_covering_lemmas} after the main proof. 

Let $\Omega$ be the $\delta^{1/2}$-lattice contained in $[0,1]$. Lemma \ref{lem_same_and_distinct} immediately implies the following.
\begin{lem}\label{lem_cover}
     The enlarged planks $\{C_1P_{\omega,\delta^{1/2}}:\omega\in \Omega\}$ cover the union of planks $\cup_{\theta\in \Theta}P_{\theta,\delta^{1/2}}$. More precisely, for each $\theta\in \Theta$, take $\omega\in \Omega$ to be the closest to $\theta$ (so that $|\omega-\theta|\le \delta^{1/2}$). Then we have
    \begin{equation}
        P_{\theta,\delta^{1/2}}\sub C_1 P_{\omega,\delta^{1/2}}.
    \end{equation}
    Moreover, the planks $\{C_1P_{\omega,\delta^{1/2}}:\omega\in \Omega\}$ are $O(K^{2})$ overlapping.
\end{lem}

\subsubsection{Applying decoupling inequality}

First, we may assume each $f_{\theta,\delta^{1/2}}$ has Fourier support on $P_{\theta,\delta^{1/2}}$ only. With a loss of a constant factor, we may assume without loss of generality that all $\Theta$ lie in an interval of length less than $C_2^{-1}$ where we recall the constant $C_2$ as in Lemma \ref{lem_same_and_distinct}. 

For $\omega\in \Omega$ and $\theta\in \Theta$, we introduce a notation: $\theta\prec \omega$ if $\theta-\omega\in (-\delta^{1/2}/2,\delta^{1/2}/2]$. We also refer to this as ``$\theta$ is attached to $\omega$". Each $\theta\in \Theta$ can be attached to a unique $\omega\in \Omega$. With this, we define
\begin{equation}
    f_\omega=\sum_{\theta\prec \omega}f_{\theta,\delta^{1/2}}.
\end{equation}
Thus $f_{\delta^{1/2}}=\sum_{\omega\in \Omega}f_\omega$. For each $\omega\in \Omega$, the Fourier support of $f_\omega$ is equal to $\cup_{\theta\prec \omega}P_{\theta,\delta^{1/2}}$, which is contained in $C_1 P_{\omega,\delta^{1/2}}$ by Lemma \ref{lem_cover}. Since $C_1 P_{\omega,\delta^{1/2}}$ are $K^{O(1)}$ overlapping and $K\lesssim_\eps \delta^{-\eps}$ by \eqref{eqn_defn_K_3D}, we may 
apply Bourgain-Demeter's decoupling Theorem \ref{thm_Bourgain-Demeter_decoupling}\footnote{The case $p=6$ is sharp for decoupling, but we shall apply this argument to a general $2\le p\le 6$, which allows us to see what would happen if we did not have sharp decoupling.} to get
\begin{equation}
    \norm{f_{\delta^{1/2}}}_p\lesssim_\eps \delta^{-\eps}\norm{\norm{f_\omega}_p}_{\ell^2(\omega\in \Omega)}.
\end{equation}
We actually only need a slightly weaker $\ell^p(L^p)$ decoupling inequality, which follows from the above via H\"older's inequality:
\begin{equation}\label{eqn_lp_decoupling}
    \norm{f_{\delta^{1/2}}}_p\lesssim_\eps \delta^{-\eps}(\#\{\omega\in \Omega:f_\omega\neq 0\})^{1/2-1/p}\norm{\norm{f_\omega}_p}_{\ell^p(\omega\in \Omega)}.
\end{equation}

Note that a priori, each $\omega$ may have a different number of $\theta\in \Theta$ for which $\theta\prec \omega$. However, after a standard process which we call dyadic pigeonholing, we may assume there is a number $N$ such that for each $\omega$, the number of $\theta$'s for which $\theta\prec \omega$ is between $N$ and $2N$. In this process, we only pay the price of a logarithmic factor $O(|\log \delta|)$. The details are given in Section \ref{sec_dyadic_pigeonholing} below. 

For simplicity, we will denote $\#\omega=\#\{\omega\in \Omega:f_\omega\neq 0\}$, and so we have 
\begin{equation}\label{eqn_after_dyadic_pigeonholing}
    \#\Theta\sim N\#\omega.
\end{equation}

\subsubsection{Analysis of each \texorpdfstring{$f_\omega$}{Lg}}

Fix $\omega$. By the triangle and H\"older's inequalities, we have
\begin{equation}
    \norm {f_\omega}_p\le \sum_{\theta\in \Theta:\theta\prec \omega}\norm {f_{\theta,\delta^{1/2}}}_p\leq \left(\sum_{\theta\in \Theta:\theta\prec \omega}\norm {f_{\theta,\delta^{1/2}}}^p_p\right)^{1/p}N^{1-1/p}.
\end{equation}
We remark that this is the best we can do in general, since all planks $P_{\theta,\delta^{1/2}}$ where $\theta\prec \omega$ are essentially the same, in the sense similar to Lemma \ref{lem_same_and_distinct}. Thus, combined with \eqref{eqn_lp_decoupling} and \eqref{eqn_after_dyadic_pigeonholing}, we arrive at
\begin{equation}\label{deltahalfcontrol}
    \begin{split}
        \norm{f_{\delta^{1/2}}}_p
    &\lesssim_\eps \delta^{-\eps}(\#\omega)^{\frac 1 2-\frac 1 p}\norm{\norm{f_{\theta,\delta^{1/2}}}_p}_{\ell^p(\theta\in \Theta)}N^{1-\frac 1 p}\\
    &\sim \delta^{-\eps} (\#\Theta)^{\frac 1 2-\frac 1 p}\norm{\norm{f_{\theta,\delta^{1/2}}}_p}_{\ell^p(\theta\in \Theta)}N^{\frac 1 2}.
    \end{split}
\end{equation}

Now we use the assumption that $\Theta$ is a $\delta$-separated $(\delta,t)$-set to get

\begin{equation}
    N\lesssim \#\left\{\theta\in \Theta\colon \theta\in \left(\omega-\frac{\delta^{1/2}}{2},\omega+\frac{\delta^{1/2}}{2}\right]\right\}\lesssim  \delta^{-t/2}.
\end{equation}
Also, for each $f_{\theta,\delta^{1/2}}$, by the elementary property of Fourier cutoff and rapid decay of $\psi_{T_\theta}$, we have
\begin{equation}
    \norm {f_{\theta,\delta^{1/2}}}_p\lesssim \norm{f_\theta}_p\sim ((\#\mathbb T_\theta)\delta^{-1})^{1/p},
\end{equation}
and by plugging these back in inequality (\ref{deltahalfcontrol}),
\begin{equation}\label{eqn_final_mixed_thm7}
    \norm {f_{\delta^{1/2}}}_p\lesssim_\eps \delta^{-\eps-\frac t 4}(\#\Theta)^{\frac{1}{2}-\frac{1}{p}}\left(\sum_{\theta\in \Theta}(\#\mathbb{T}_{\theta})\delta^{-1}\right)^{\frac 1 p}.
\end{equation}
Using $\#\mathbb T_\theta\lesssim \delta^{-s}$, we have
\begin{equation}\label{eqn_final_mixed}
    \norm {f_{\delta^{1/2}}}_p\lesssim_\eps \delta^{-\eps-\frac {s+1}{p}-\frac t 4}(\#\Theta)^{\frac 1 2}.
\end{equation}

\subsubsection{Comparison with the high part}

We now make a comparison between the bounds $\norm {f_{\mathrm{high}}}_{L^p(H)}$ given by \eqref{eqn_high_contribution_3D} and $\norm {f_{\delta^{1/2}}}_p$. If $p>4$, using the lower bound $\#\Theta\gtrapprox \delta^{-t}$, we see that for small enough $\delta$,
\begin{equation}
    (\#\Theta)^{\frac {p-1}{p}}K^{\frac 2 p} \delta^{-\frac {s+1}p}\ge \delta^{-\eps-\frac {s+1}{p}-\frac t 4}(\#\Theta)^{\frac 1 2},
\end{equation}
which means that the upper bound for the $L^p$ norm of the $\delta^{1/2}$-mixed part is dominated by the upper bound for the $L^p$ norm of the high part of $f$ in $H$.

\subsection{Analysis of \texorpdfstring{$\lambda$}{Lg}-mixed parts}\label{sec_step_7_3D}
The main tool for this part is still decoupling, but more complicated than the previous step. Fix a dyadic scale $\lambda\in [\delta^{1/2},K^{-1}]$.

For technical reasons that arise in the proof of Lemma \ref{lem_same_and_distinct_lambda} in Section \ref{sec_covering_lemmas} below, we first perform a preliminary reduction. Let $C_0$ be a large dyadic constant to be determined. If $\lambda \in [\delta^{1/2},C_0\delta^{1/2}]$, then the same method of Section \ref{sec_step_6_3D} can be used. Thus we assume $\lambda>C_0\delta^{1/2}$ throughout the rest of this subsection.

\subsubsection{Construction of covering planks}

First, since $P_{\theta,\lambda}$ consists of four similar parts, we may assume without loss of generality that each $f_{\theta,\lambda}$ is Fourier supported on one part:
\begin{equation}    
Q_{\theta,\lambda}:=\left\{\sum_{i=1}^3 \xi_i \mathbf e_i(\theta):|\xi_1|\le \delta,\lambda/2\le \xi_2\le \lambda,K^{-1}\le\xi_3\le 1\right\}.
\end{equation}
Similarly, for $C\ge 1$, we define the slightly enlarged planks
\begin{equation}    
CQ_{\theta,\lambda}:=\left\{\sum_{i=1}^3 \xi_i \mathbf e_i(\theta):|\xi_1|\le C\delta,(2C)^{-1}\lambda\le \xi_2\le C\lambda,C^{-1}K^{-1}\le \xi_3\le C\right\}.
\end{equation}
Note that this is the same as the $S_\tau$ defined on Page 15 of \cite{GGGHMW2022}.

Define $\Sigma$ to be the $\lambda$-lattice contained in $[0,1]$. For each $\sigma\in \Sigma$, consider the following plank at scale $\lambda$:
\begin{equation}
    P_{\sigma,\lambda}=\left\{\sum_{i=1}^3 \xi_i \mathbf e_i(\sigma):|\xi_1|\le \lambda^2,|\xi_2|\le \lambda,K^{-1}\le \xi_3\le 1 \right\}.
\end{equation}
Also recall that for $C\ge 1$,
\begin{equation}\label{eqn_CP_sigma_lambda}
    CP_{\sigma,\lambda}=\left\{\sum_{i=1}^3 \xi_i \mathbf e_i(\sigma):|\xi_1|\le C\lambda^2,|\xi_2|\le C\lambda,C^{-1}K^{-1}\le \xi_3\le C \right\}.
\end{equation}

Since we have an extra intermediate scale $\lambda$, we will need the following lemma, which is rephrased from Lemma 3 of \cite{GGGHMW2022}.
\begin{lem}\label{lem_same_and_distinct_lambda}
    By choosing $C_1$ and then $C_2$ to be large enough, we have
    \begin{enumerate}
        \item \label{item_1_same_and_distinct_lambda} If $|\theta-\theta'|\le \lambda^{-1}\delta$, then $Q_{\theta,\lambda}$ and $Q_{\theta',\lambda}$ are essentially the same in the sense that $Q_{\theta,\lambda}\sub C_1Q_{\theta',\lambda}$.
        \item \label{item_2_same_and_distinct_lambda} If $C_2\lambda^{-1}\delta\le |\theta-\theta'|\le C_2^{-1}\lambda$, then $C_1Q_{\theta,\lambda}$ and $C_1Q_{\theta',\lambda}$ are disjoint.
    \end{enumerate}
\end{lem}
The proof is slightly different from Lemma \ref{lem_same_and_distinct}, so we include it in Section \ref{sec_covering_lemmas}. The following lemma is almost identical to Lemma \ref{lem_same_and_distinct}, so we omit the proof.

\begin{lem}\label{lem_cover_lambda}
The following statements are true.
\begin{enumerate}
\item \label{item_1_cover_lambda} If $|\theta-\theta'|\le \lambda$, then $C_1P_{\theta,\lambda}$ and $C_1P_{\theta',\lambda}$ are essentially the same in the sense that $C_1P_{\theta,\lambda}\sub C_3P_{\theta',\lambda}$ where $C_3=C_3(C_1)$ is large enough.
\item If \label{item_2_cover_lambda} $C_4K\lambda\le |\theta-\theta'|\le C_4^{-1} K^{-1}$ where $C_4=C_4(C_3)$ is large enough, then $C_3P_{\theta,\lambda}$ and $C_3P_{\theta',\lambda}$ are disjoint.
\end{enumerate}       
\end{lem}
The key difference from the $\delta^{1/2}$-mixed case is that the natural separation in this case is $\lambda^{-1}\delta$ which is smaller than $\lambda$, the aperture of $Q_{\theta,\lambda}$. Moreover, in terms of the canonical scale of decoupling, the thickness $\delta$ of $Q_{\theta,\lambda}$ does not match its aperture $\lambda$ which was supposed to be $\delta^{1/2}$. When $\lambda=\delta^{1/2}$ as in the previous case, we had $\lambda^{-1}\delta=\delta^{1/2}=\lambda$ which made the scenario simpler.

Motivated by Lemma \ref{lem_same_and_distinct_lambda}, we define $\mathcal T$ to be the $\lambda^{-1}\delta$-lattice in $[0,1]$. Now we have three subsets of $[0,1]$:
\begin{equation}\label{eqn_intermediate}
    \Theta\sub \delta \Z\cap [0,1],\quad \mathcal T=(\lambda^{-1}\delta)\Z\cap [0,1],\quad \Sigma=\lambda \Z\cap [0,1].
\end{equation}
We will define relationships between their elements. For each $\theta\in \Theta$, we attach it to a $\tau\in \mathcal T$ if $\theta-\tau\in (\lambda^{-1}\delta/2,\lambda^{-1}\delta/2]$, which we denote by $\theta\prec \tau$. Thus each $\theta$ is attached to a unique $\tau$. With this, we define
\begin{equation}
    f_\tau=\sum_{\theta\prec \tau}f_{\theta,\lambda}.
\end{equation}
Similarly, for each $\tau\in \mathcal T$, we attach it to a $\sigma\in \Sigma$ if $\tau-\sigma\in (\lambda/2,\lambda/2]$, which we denote by $\tau\prec \sigma$. Thus each $\tau$ is attached to a unique $\sigma$. With this, we define
\begin{equation}
    f_\sigma=\sum_{ \tau\prec \sigma}f_{\tau}.
\end{equation}
We also write $\theta\prec \sigma$ if there is a $\tau$ such that $\theta\prec \tau$ and $\tau\prec \sigma$. Thus each $\theta$ is attached to a unique $\sigma$. Thus we have
\begin{equation}
    f_\lambda=\sum_{\sigma\in \Sigma}f_\sigma=\sum_{\sigma\in \Sigma}\sum_{ \tau\prec \sigma}f_\tau=\sum_{\sigma\in \Sigma}\sum_{ \tau\prec \sigma}\sum_{\theta\prec \tau}f_{\theta,\lambda}.
\end{equation}
\subsubsection{Applying decoupling inequalities}

Before estimating $\int |f_\lambda|^p$, we may first apply a similar preliminary reduction to that of the $\delta^{1/2}$-mixed part to reduce to the case that $\mathrm{diam}(\Theta)\le C_4^{-1}K^{-1}$. Next, we may apply a dyadic pigeonholing argument to subsets of $\Theta,\mathcal T,\Sigma$ (still denoted by $\Theta,\mathcal T,\Sigma$ by abuse of notation), so that there are numbers $M,N$ such that 
\begin{equation}\label{eqn_after_pigeonholing_mixed}
    \#\{\theta\in \Theta:\theta\prec \tau\}\sim M,\quad \#\{\tau\in \mathcal T:\tau\prec \sigma\}\sim N.
\end{equation}
We only pay the price of logarithmic losses here; the detail is similar to the proof in Section \ref{sec_dyadic_pigeonholing}, which we ask the reader to verify.

We first apply Theorem \ref{thm_Bourgain-Demeter_decoupling} to decouple $f_\lambda$ into different $f_\sigma$'s. The detail is as follows. First, the Fourier support of each $f_\tau$ is equal to $\cup_{\theta\prec \tau}Q_{\theta,\lambda}$ 
and in turn contained in $C_1Q_{\tau,\lambda}$ by Part \ref{item_1_same_and_distinct_lambda} of Lemma \ref{lem_same_and_distinct_lambda}. Thus the Fourier support of each $f_\sigma$ is contained in $\cup_{\tau\prec \sigma}C_1 Q_{\tau,\lambda}$, which is contained in $C_3 P_{\sigma,\lambda}$ by Part \ref{item_1_cover_lambda} of Lemma \ref{lem_cover_lambda} and the trivial fact that $C_1Q_{\tau,\lambda}\sub C_1P_{\tau,\lambda}$. Also, since $\text{diam}(\Theta)\leq C_4^{-1}K^{-1}$, by Part \ref{item_2_cover_lambda} of Lemma \ref{lem_cover_lambda}, $\{C_3 P_{\sigma,\lambda}:\sigma\in \Sigma\}$ has bounded overlap.  

We thus apply Theorem \ref{thm_Bourgain-Demeter_decoupling} with respect to the planks $C_3P_{\sigma,\lambda}$ at scale $\lambda^{2}$ in place of $\delta$, followed by H\"older's inequality to get
\begin{equation}\label{eqn_decoupling_to_sigma}
    \norm{f_{\lambda}}_p\lesssim_\eps \lambda^{-\eps}\norm{\norm{f_\sigma}_p}_{\ell^2(\sigma\in \Sigma)}\leq \lambda^{-\eps}(\#\sigma)^{\frac 1 2-\frac 1 p}\norm{\norm{f_\sigma}_p}_{\ell^p(\sigma\in \Sigma)},
\end{equation}
where, for simplicity, we denoted $\#\sigma:=\#\{\sigma\in \Sigma:f_\sigma\ne 0\}$.

Now we need to analyse each $f_\sigma$, whose Fourier transform has support contained in $\cup_{\tau\prec \sigma}C_1 Q_{\tau,\lambda}$. By Part \ref{item_2_same_and_distinct_lambda} of Lemma \ref{lem_same_and_distinct_lambda}, the collection $\{C_1 Q_{\tau,\lambda}:\tau\prec\sigma\}$ has bounded overlap. Thus we can apply the flat decoupling Theorem \ref{thm_flat_decoupling} to get (recall \eqref{eqn_after_pigeonholing_mixed}):
\begin{equation}\label{eqn_decoupling_to_tau}
\norm{f_{\sigma}}_p\lessapprox N^{1-\frac 2 p}\norm{\norm{f_\tau}_p}_{\ell^p(\tau\prec \sigma)}.
\end{equation}

\subsubsection{Analysis of each \texorpdfstring{$f_\tau$}{Lg}}

Fix $\tau$. We use triangle and H\"older's inequalities (recall \eqref{eqn_after_pigeonholing_mixed}):
\begin{equation}\label{eqn_to_each_theta}
    \norm{f_{\tau}}_p\lesssim M^{1-\frac 1 p}\norm{\norm{f_{\theta,\lambda}}_p}_{\ell^p(\theta\prec \tau)}.
\end{equation}
This is the best we can do in general since for each $\theta\prec \tau$, the planks $Q_{\theta,\lambda}$ are essentially the same.

Combining the above inequalities \eqref{eqn_decoupling_to_sigma} through \eqref{eqn_to_each_theta} gives
\begin{equation}\label{eqn_before_each_f_theta_lambda}    \norm{f_{\lambda}}_p\lesssim_\eps \delta^{-\eps}(\#\sigma)^{\frac 1 2-\frac 1 p}M^{1-\frac 1 p}N^{1-\frac 2 p}\norm{\norm{f_{\theta,\lambda}}_p}_{\ell^p(\theta\in \Theta)}.
\end{equation}
Now for each $\theta$, by the elementary property of Fourier cutoff and rapid decay of each $\psi_{T_\theta}$, we have 
\begin{equation}\label{eqn_each_f_theta_lambda}    \norm{f_{\theta,\lambda}}_p\lesssim \norm{f_\theta}_p\sim (\#\mathbb T_\theta \delta^{-1})^{1/p}.
\end{equation}
Now by bounded overlap, we observe that
\begin{equation}\label{eqn_numbers_relation}
    MN\sim \#\{\theta\in \Theta:\theta\prec \sigma\}\lesssim (\lambda/\delta)^t,\quad \#\Theta\sim (\#\sigma) MN.
\end{equation}
where we have used the fact that $\Theta$ is a $\delta$-separated $(\delta,t)$-set in the first relation. We also have $M\lesssim \lambda^{-t}$ by the $(\delta,t)$-set condition. Using this, noting $p\ge 2$ and \eqref{eqn_before_each_f_theta_lambda}, \eqref{eqn_each_f_theta_lambda} and \eqref{eqn_numbers_relation}, we thus have
\begin{align}
    \norm{f_{\lambda}}_p
    &\lesssim_\eps \delta^{-\eps}(\#\Theta)^{\frac 1 2-\frac 1 p}(MN)^{\frac 1 2-\frac{1}{p}}M^{\frac 1 p}\delta^{-\frac {1}p}\left(\sum_{\theta\in \Theta}\#\mathbb T_\theta\right)^{\frac 1 p}\nonumber\\
    &\lesssim\delta^{-\eps-\frac {1}p}\left(\sum_{\theta\in \Theta}\#\mathbb T_\theta\right)^{\frac 1 p}(\#\Theta)^{\frac 1 2-\frac 1 p}(\lambda/\delta)^{t(\frac 1 2-\frac 1 p)}\lambda^{-\frac t p}\nonumber\\
    &=\delta^{-\eps-\frac {1}p-t(\frac 1 2-\frac 1 p)}\left(\sum_{\theta\in \Theta}\#\mathbb T_\theta\right)^{\frac 1 p}(\#\Theta)^{\frac 1 2-\frac 1 p}\lambda^{t(\frac 1 2-\frac 2 p)}.\label{eqn_f_lambda_upper_bound_thm7}
\end{align}
We remark that when $\lambda\sim\delta^{1/2}$, this reduces to the same expression in \eqref{eqn_final_mixed} in the $\delta^{1/2}$-mixed part. Using $\#\mathbb T_\theta\lesssim \delta^{-s}$, we have
\begin{equation}
\norm{f_{\lambda}}_p\lesssim_\eps \delta^{-\eps-\frac {s+1}p-t(\frac 1 2-\frac 1 p)}(\#\Theta)^{\frac 1 2}\lambda^{t(\frac 1 2-\frac 2 p)}.
\end{equation}

\subsubsection{Comparison with the high part}

We now make a comparison between the bounds to $\norm {f_{\mathrm{high}}}_p^p$ given by \eqref{eqn_high_contribution_3D} and the bounds for  $$\sum_{\lambda \text{ dyadic};\lambda\in (\delta^{1/2},K^{-1}]}\norm {f_{\lambda}}_p^{p}$$
obtained from the previous step.

If $p>4$, using the lower bound $\#\Theta\gtrapprox \delta^{-t}$, we see that for small enough $\delta$,
\begin{equation}
    (\#\Theta)^{\frac {p-1}{p}}K^{\frac{2}{p}} \delta^{-\frac {s+1}p}\ge \delta^{-\eps-\frac {s+1}p-t(\frac 1 2-\frac 1 p)}(\#\Theta)^{\frac 1 2}\lambda^{t(\frac 1 2-\frac 2 p)},
\end{equation}
which means that the upper bound for the $\lambda$-mixed part is dominated by the upper bound of the high part.

\subsection{Conclusion}

Since Theorem \ref{thm_Bourgain-Demeter_decoupling} holds at the sharp exponent $p=6$, by the discussion at the end of Sections \ref{sec_step_6_3D} and \ref{sec_step_7_3D}, we know that the high part dominates over both the low part and the mixed parts. Since the number of dyadic $\lambda\in [\delta^{1/2},K^{-1}]$ is $O(|\log \delta|)$, by \eqref{eqn_high_upper_bound_3D}, we have
\begin{equation}
    \int_H |f|^p\lesssim |\log \delta|(\#\Theta)^{p-1}K^2 \delta^{-1-s}.
\end{equation}
On the other hand, since $f(x)\gtrsim \# \Theta$ on $H$, we trivially have the lower bound
\begin{equation}
    \int_H |f|^p\gtrsim (\# \Theta)^p |H|\sim (\# \Theta)^p\delta^{-a}.
\end{equation}
Recalling \eqref{eqn_defn_K_3D} and $\#\Theta\gtrapprox \delta^{-t}$, we have $\delta^{-a-t}\lesssim_\eps \delta^{-1-s-\eps}$, as required.

\subsection{Proof of dyadic pigeonholing}\label{sec_dyadic_pigeonholing}
We now prove the statement leading to \eqref{eqn_after_dyadic_pigeonholing}. Recall that by $\Omega$ we denote the $\delta^{1/2}$-lattice contained in $[0,1]$. For each $\omega\in \Omega$, by the $\delta$-separation condition of $\Theta$, we have the trivial bound $\#\{\theta\prec \omega\}\leq \delta^{-1/2}$. 

Let $J=\log_2 (\delta^{-1/2})$. For each integer $0\le j\le J$, we denote the set
\begin{equation}
    \Omega_j:=\{\omega\in \Omega:\#\{\theta\in \Theta:\theta\prec \omega\}\in [2^j,2^{j+1})\}.
\end{equation}
 Then it is easy to see that
\begin{equation}   
\sum_{j=1}^J 2^j\#\Omega_j\le \#\Theta\le \sum_{j=1}^J 2^{j+1}\#\Omega_j
\end{equation}
since we only consider those $\omega$ such that $\{\theta\in \Theta:\theta\prec \omega\}\ne \varnothing$.

Now by the pigeonholing principle, there exists some $i\in [1,J]$ such that 
\begin{equation}
   \sum_{j=1}^J 2^j\#\Omega_j\le J2^{i}\#\Omega_{i}.
\end{equation}
This gives
\begin{equation}
    2^{i}\#\Omega_i\leq\#\Theta\le J2^{i+1}\#\Omega_i,
\end{equation}
which is as required once we take $N=2^i$, since $J\sim |\log \delta|$.

\subsection{Proof of covering lemmas}\label{sec_covering_lemmas}
We now give proofs of the covering Lemmas \ref{lem_same_and_distinct} and \ref{lem_same_and_distinct_lambda}.

\begin{proof}[Proof of Lemma \ref{lem_same_and_distinct}]

We first prove Part \ref{item_1_same_and_distinct}. Let $\xi=\sum_{i=1}^3 \xi_i \mathbf e_i(\theta)\in P_{\theta,\delta^{1/2}}$. Then using Proposition \ref{prop_error_e_i},
\begin{align*}
    |\xi_1'|:&=|\xi\cdot \mathbf e_1(\theta')|\lesssim \delta\cdot 1+\delta^{1/2}\cdot \delta^{1/2}+1\cdot \delta\sim \delta,\\
    |\xi_2'|:&=|\xi\cdot \mathbf e_2(\theta')|\lesssim \delta\cdot \delta^{1/2}+\delta^{1/2}\cdot 1+1\cdot \delta^{1/2}\sim \delta^{1/2},\\
    \xi_3':&=\xi\cdot \mathbf e_3(\theta')=O(\delta^2)+O(\delta^{1/2}\cdot \delta^{1/2})+\xi_3 (1+O(\delta))\\
    &\in [K^{-1}/2,2].
\end{align*}

Next, we prove Part \ref{item_2_same_and_distinct} by contradiction. Let $\xi=\sum_{i=1}^3 \xi_i \mathbf e_i(\theta)\in C_1P_{\theta,\delta^{1/2}}$ and also $\xi=\sum_{i=1}^3 \xi'_i \mathbf e_i(\theta')\in C_1P_{\theta',\delta^{1/2}}$. 
It turns out that we only need to study $|\xi'_2|=|\xi\cdot \mathbf e_2(\theta')|\leq C_1\delta^{1/2}$. First, since $\xi\in C_1P_{\theta,\delta^{1/2}}$, we have $\xi_1=O(\delta)$ and $\xi_2=O(\delta^{1/2})$, and so 
\begin{equation}
    \xi_2'=\xi\cdot \mathbf e_2(\theta')=\xi_3 \mathbf e_3(\theta)\cdot  \mathbf e_2(\theta')+O(\delta^{1/2}),
\end{equation}
and so by $\xi_3\gtrsim K^{-1}$ we have
\begin{equation}
    |\mathbf e_3(\theta)\cdot  \mathbf e_2(\theta')|\lesssim K\delta^{1/2}.
\end{equation}

By Taylor expansion for $\mathbf{e}_3$ around $\theta'$, we get
\begin{equation}
\mathbf{e}_3(\theta)=\mathbf{e}_3(\theta')+\mathbf{e}_3'(\theta')(\theta-\theta')+o(\theta-\theta')    
\end{equation}
where $\frac{|o(\theta-\theta')|}{|\theta-\theta'|}\rightarrow 0$ as $\theta\rightarrow \theta'$. Then
\begin{equation*}
    \begin{split}
        K\delta^{1/2}\gtrsim &|\mathbf{e}_3(\theta)\cdot \mathbf{e}_2(\theta')|\\
=&|\{\mathbf{e}_3(\theta')+\mathbf{e}_3'(\theta')(\theta-\theta')+o(\theta-\theta')\}\cdot \mathbf{e}_2(\theta')|\\
=&|\{-\uptau(\theta')\mathbf{e}_2(\theta')(\theta-\theta')+o(\theta-\theta')\}\cdot \mathbf{e}_2(\theta')|\\
\gtrsim &|\uptau(\theta')||\theta-\theta'|-|o(\theta-\theta')|
    \end{split}
\end{equation*}
By the condition $\frac{|o(\theta-\theta')|}{|\theta-\theta'|}\rightarrow 0$, combined with $\uptau\sim 1$, we can get 
\begin{equation}
    \begin{split}
        K\delta^{1/2}\gtrsim |\theta-\theta'|
    \end{split}
\end{equation}
which leads to the contradiction by the assumption on $|\theta-\theta'|$.

\end{proof}

\begin{proof}[Proof of Lemma \ref{lem_same_and_distinct_lambda}]
    For the first part, let $\xi=\sum_{i=1}^3 \xi_i \mathbf e_i(\theta)\in Q_{\theta,\lambda}$. Then using Proposition \ref{prop_error_e_i} with $\lambda^{-2}\delta^{2}$ in place of $\delta$,

\begin{align*}
    |\xi_1'|:&=|\xi\cdot \mathbf e_1(\theta')|\lesssim \delta\cdot 1+\lambda \cdot \lambda^{-1}\delta+1\cdot \lambda^{-2}\delta^2\sim \delta,\\
    \xi_2':&=\xi\cdot \mathbf e_2(\theta')= O(\delta)+\xi_2(1+O(\lambda^{-2}\delta^2))+O(\lambda^{-1}\delta)\in [\lambda/4,4\lambda],\\
    &\text{if $C_0$ chosen at the beginning of Section \ref{sec_step_7_3D} is large enough,}\\
    \xi_3':&=\xi\cdot \mathbf e_3(\theta')=O(\lambda^{-2}\delta^{3})+O(\delta)+\xi_3 (1+O(\lambda^{-2}\delta^2))\\
    &\in [K^{-1}/2,2].
\end{align*}

Next, we prove the disjointedness by contradiction. Let $\xi=\sum_{i=1}^3 \xi_i \mathbf e_i(\theta)\in C_1Q_{\theta,\lambda}$ and also $\xi=\sum_{i=1}^3 \xi'_i \mathbf e_i(\theta')\in C_1Q_{\theta',\lambda}$.

It turns out that we only need to study $|\xi'_1|=|\xi\cdot \mathbf e_1(\theta')|\le C_1\delta$. Using Taylor expansion for $\mathbf{e}_{1}(\theta')$ around $\theta$, Lemma \ref{lem_frame_derivatives} and the assumption that $\xi\in C_1Q_{\theta,\lambda}$, this implies
\begin{equation}
    |\xi_2(\theta'-\theta)+\xi_2 O(|\theta'-\theta|^2)+\xi_3O(\theta'-\theta)^2|=O(\delta),
\end{equation}
and so using $|\xi_3|\lesssim 1$, $|\theta-\theta'|\le C_2^{-1}\ll 1$ and $\xi_2\sim \lambda$, we get
\begin{equation}
    \lambda|\theta'-\theta|=O(\delta+|\theta'-\theta|^2).
\end{equation}
Thus we have either $|\theta'-\theta|\lesssim \lambda^{-1}\delta$, or $|\theta'-\theta|\gtrsim \lambda$, but both contradict the assumption if $C_2$ is large enough.

\end{proof}

\section{Proof of Theorem 7}

The goal of this section is to prove Theorem 7 in \cite{GGGHMW2022} and its corollaries 2 and 3. Corollary 3 was a necessary ingredient for the proof of Theorem 8. One can keep in mind the following implications where the theorem numbers correspond to the ones in \cite{GGGHMW2022}:
\begin{center}
    Theorem 7 $\Rightarrow $ Corollary 2 $\Rightarrow $ Corollary 3 $\Rightarrow $ Theorem 8 $\Rightarrow $ Theorem 2. 
\end{center}

\begin{thm}[Theorem 7 in \cite{GGGHMW2022}]\label{thm_thm7}
    Fix $0\le a\le 2$. For each $\eps>0$ and each $0<c\le 1$, there exists a constant $b(\eps,a,c)>0$ 
    so that the following holds. Let $\delta>0$. Let $\mu$ be a finite nonzero Borel measure supported in the unit ball in $\R^3$ with $\rho_a(\mu)<\infty$ (recall \eqref{eqn_Frostman_mu_intro}).
    
    Let $\Theta$ be a maximally $\delta$-separated subset of $[0,1]$. For each $\theta\in \Theta$, let $\mathbb T_\theta$ be a collection of disjoint $\delta\times \delta\times 1$ tubes pointing in the direction $\gamma(\theta)$, and let $\mathbb T=\cup_{\theta\in \Theta} \mathbb T_\theta$. Suppose that
    \begin{equation}\label{eqn_lower_bound_supp_mu}
        \sum_{T\in \mathbb T}1_T(x)\ge c\delta^{\eps-1},\quad \forall x\in \mathrm{supp}(\mu).
    \end{equation}
    Then
    \begin{equation}\label{eqn_thm7_main}
        \#\mathbb T\ge b(\eps,a,c)\mu(\R^3)\rho_a(\mu)^{-1}\delta^{-1-a+100\sqrt \eps}.
    \end{equation}    
\end{thm}
The main argument is a high-low decomposition, combined with an induction on scales argument. Decoupling is also used in the analysis of high part. Note that we have rescaled the physical space back to the unit ball, and so the frequency scale will be in $B^3(0,\delta^{-1})$.

\subsection{A bootstrap inequality}

To facilitate the induction, we first prove the following bootstrap inequality.

\begin{lem}\label{lem_bootstrap}
    With the statement of Theorem \ref{thm_thm7}, let $M(\delta,\eps,c)=M_a(\delta,\eps,c)$ be the smallest constant for which
    \begin{equation}\label{eqn_defn_M_delta}
        M(\delta,\eps,c)\#\mathbb T\ge \mu(\R^3)\rho_a(\mu)^{-1}\delta^{-1-a}.
    \end{equation}
    Then there exists a small $\delta_0=\delta_0(\eps,a,c)\in (0,1)$ such that for all $\delta\in (0,\delta_0)$ we have
    \begin{equation}\label{eqn_bootstrap_thm7}
        M(\delta,\eps,c)\leq \delta^{-2\eps}M\left(\delta^{1-\sqrt \eps},\frac\eps{1-\sqrt \eps},c_0 c\right)+c^{-4}A(\eps) \delta^{-10\sqrt \eps},
    \end{equation}
    where $c_0\in (0,1)$ is an absolute constant and $A(\eps)$ is decreasing in $\eps$.
\end{lem}

{\it Remark.} By \eqref{eqn_lower_bound_supp_mu} we see that $c\delta^{\eps-1}\leq \#\Theta$ so 
\begin{equation}
    M(\delta,\eps,c)\le c^{-1}\mu(\R^3)\rho_a(\mu)^{-1}\delta^{-a-\eps},
\end{equation}
which is in particular less than $\delta^{-3}$ for $\delta$ small enough.

\begin{proof}[Proof of Theorem \ref{thm_thm7} assuming Lemma \ref{lem_bootstrap}]
    The base cases when $\delta\ge \delta_0$ or $\eps>1/2$ are trivial, so it suffices to assume $\delta<\delta_0$ and $\eps<1/2$ and iterate the bootstrap inequality \eqref{eqn_bootstrap_thm7} in Lemma \ref{lem_bootstrap}.

    Let $\delta<\delta_0$. We define three sequences $\delta_n,\eps_n,c_n$, $n\in \N$ as follows:
    \begin{gather}
        \delta_1:=\delta,\quad \eps_1:=\eps,\,c_1=c\\
        \eps_{n+1}:=\frac {\eps_n}{1-\sqrt{\eps_n}},\\
        \delta_{n+1}:=\delta_{n}^{1-\sqrt {\eps_n}},\\
        c_{n+1}:=c_0^{n}c.
    \end{gather}
    Then for each $N\in \N$ such that $\delta_{N}<\delta_0$ and $\eps_N<1/2$, we have, by direct computation
    \begin{equation}\label{eqn_iteration_02}
        M(\delta,\eps,c)
        \le \delta^{-2N\eps} M(\delta_{N+1},\eps_{N+1},c_{N+1})+\sum_{n=1}^N c_n^{-4}A(\eps_n)\delta^{-2(n-1)\eps}\delta_n^{-10{\sqrt{\eps_n}}}.
    \end{equation}

It is easy to see that $\eps_n$ is increasing and bounded from below by $\eps$. Also, it diverges to $\infty$. By the definition of the $\delta_n$, we have
$$\delta_{n+1}=\delta_n^{\eps_n/\eps_{n+1}}=(\delta_{n-1})^{\frac{\eps_{n-1}}{\eps_n}\frac{\eps_n}{\eps_{n+1}}}=\delta_{n-1}^{\frac{\eps_{n-1}}{\eps_{n+1}}},$$ and by iterating this a few times
$\delta_{n}=\delta^{\frac{\eps}{\eps_{n}}}$.
Therefore 
\begin{equation}\label{deltanbound}
    \delta_n^{-10\sqrt{\eps_n}}= \delta^{-10\frac{\eps}{\sqrt{\eps_n}}}\leq \delta^{-10\sqrt{\eps}}.
\end{equation}
If we let $N$ be the smallest integer such that $\eps_{N+1}\ge 1/2>\eps_N$, then we have the following claim:
\begin{equation}\label{eqn_decrease_rate}
    N\le 5\eps^{-1/2}.
\end{equation}

Assuming this, using $M(\delta_{N+1},\eps_{N+1},c_{N+1})\lesssim 1$ and inequality \eqref{deltanbound}, we can estimate from \eqref{eqn_iteration_02} that
\begin{align*}
    M(\delta,\eps,c)
    &\lesssim \delta^{-10\sqrt \eps}\left(1+\sum_{n=1}^N (c_0^{n-1}c)^{-4}A(\eps_n)\delta^{-10\sqrt \eps }\right)\\
    &\le \delta^{-10\sqrt \eps}\left(1+N(c_0^{N-1}c)^{-4}A(\eps)\delta^{-10\sqrt\eps}\right)\\
    &\lesssim_{\eps,c} \delta^{-100\sqrt\eps},
\end{align*}
which proves Theorem \ref{thm_thm7}. It remains to prove \eqref{eqn_decrease_rate}. Write $x_n=\eps_n^{-1}$, so we have the iterative formula
\begin{equation}\label{eqn_iterative_x_n}
    x_{n+1}=x_n-\sqrt {x_n}.
\end{equation}
Since $\eps_{N+1}=\frac{\eps_N}{1-\sqrt{\eps_N}}<\frac{\sqrt{2}}{2(\sqrt{2}-1)}$, one has $x_{N+1}=\eps_{N+1}^{-1}> 2-\sqrt{2}>1$.

Thus $2< x_n\le \eps^{-1}$ for every $1\le n\le N$, $x_{N+1}\in (2-\sqrt 2,2]$, and $(x_n)$ is decreasing from $1$ to $N+1$. The function $f(t)=t-\sqrt t$ is strictly increasing for $t>1$, and so \eqref{eqn_iterative_x_n} implies
\begin{equation}
    x_n=x_{n+1}+\frac 1 2+\frac {\sqrt {4x_{n+1}+1}} 2:=f^{-1}(x_{n+1}).
\end{equation}
In reverse order, we write $y_1=x_{N+1}$ and $y_{n+1}=f^{-1}(y_n)$, so $y_{N+1}= \eps^{-1}$.

The heuristic idea is as follows.
Since $f^{-1}(y)>y+\sqrt y$, if we view $y_n$ as a continuous function in $n$, then the iterative formula roughly says $\frac {dy}{dn}>y^{1/2}$, which gives $y_n>n^2/4$. 

The detail is as follows. We actually prove that $y_n>\frac {n^2y_1}{14}$. The base case $n=1$ is trivial and the case $n=2$ follows from $y_2=x_N>2$ and $y_1=x_{N+1}\leq 2$. Assume it holds for some $n\ge 2$, then we have
\begin{align*}
    y_{n+1}-\frac {(n+1)^2y_1}{14}
    &=f^{-1}(y_n)-\frac {(n+1)^2y_1}{14}\\
    &> y_n+\sqrt {y_n}-\frac {(n+1)^2y_1}{14}\\
    &\ge\frac {n^2y_1}{14}+\sqrt{\frac {n^2y_1}{14}}-\frac {(n+1)^2y_1}{14}\\
    &=\frac{\sqrt {y_1}}{14}((\sqrt {14}-2\sqrt {y_1})n-\sqrt {y_1}),
\end{align*}
which is positive since $y_1=x_{N+1}\leq2$ and $n\geq 2$. Using $y_N>N^2y_1/14$, we thus have
\begin{equation}\label{controlofN}
    N<\sqrt{\frac{14y_N}{y_1}}\le\sqrt {\frac{14y_N}{2-\sqrt 2}}\le 5\sqrt {y_N}\leq5\eps^{-1/2}.
\end{equation}

\end{proof}

The remaining of this section is devoted to the proof of Lemma \ref{lem_bootstrap}.

\subsection{Construction of wave packets} 

This step is the same as Section \ref{sec_step_1_3D} in the proof of Theorem \ref{thm_thm5} except rescaling.

Fix $\theta\in \Theta$. We define a frequency slab $P_\theta$ centred at the origin that has dimensions $\delta^{-1}\times \delta^{-1}\times 1$, and its shortest direction is parallel to $\mathbf e_1(\theta)$. That is, $P_\theta$ is dual to every $T_\theta\in \mathbb T_\theta$. We may also assume without loss of generality that $\mathbb T_\theta$ is a subcollection of the lattice that tiles $[-10,10]^3$. 
    
    For each $T_\theta\in \mathbb T_\theta$, define a Schwartz function $\psi_{T_\theta}:\R^3\to \R$ Fourier supported on $P_\theta$ such that $\psi_{T_\theta}\gtrsim 1$ on $T_\theta$ and decays rapidly outside $T_\theta$. With this, we define
    \begin{equation}
        f_\theta=\sum_{T_\theta\in \mathbb T_\theta}\psi_{T_\theta},\quad f=\sum_{\theta\in \Theta}f_\theta.
    \end{equation}
    Thus $f_\theta$ is also Fourier supported on $P_\theta$, but physically essentially supported on $\cup \mathbb T_\theta$. The function $f$ is Fourier supported on $\cup_{\theta\in \Theta}P_\theta$, which is a union of $\delta^{-1}\times \delta^{-1}\times 1$ slabs of the same size centered at $0$, but with normal orientations prescribed by $\mathbf e_1(\theta)$.

    Thus by definition, for every $x\in \mathrm{supp}(\mu)$ we have
    \begin{equation}\label{eqn_f_bound_thm7}
       c\delta^{\eps-1} \leq f(x).
    \end{equation}

    In particular, for every $p<\infty$,
    \begin{equation}\label{eqn_lower_bound_f_thm7}
        \mu(\R^3)c^p \delta^{p(\eps-1)}\le\int |f|^p d\mu.
    \end{equation}

\subsection{High-low decomposition} 
For each $\delta>0$, $\eps>0$, let
\begin{equation}\label{eqn_defn_K_thm7}
        K=K(\delta,\eps)=10^{-4}\delta^{-\sqrt \eps}.
\end{equation}
With this, decompose $P_\theta$ into:
    \begin{enumerate}
    \item The high part:
        \begin{equation}            P_{\theta,\mathrm{high}}:=\left\{\sum_{i=1}^3 \xi_i \mathbf e_i(\theta):|\xi_1|\le 1,K^{-1}\delta^{-1}\le |\xi_2|\le \delta^{-1},|\xi_3|\le \delta^{-1}\right\}.
        \end{equation}        
        \item The low part:
        \begin{equation}            P_{\theta,\mathrm{low}}:=\left\{\sum_{i=1}^3 \xi_i \mathbf e_i(\theta):|\xi_1|\le 1,|\xi_2|\le K^{-1}\delta^{-1},|\xi_3|\le K^{-1}\delta^{-1}\right\}.
        \end{equation}        
        \item The $\delta^{1/2}$-mixed part:
        \begin{equation}            P_{\theta,\delta^{1/2}}:=\left\{\sum_{i=1}^3 \xi_i \mathbf e_i(\theta):|\xi_1|\le 1,|\xi_2|\le \delta^{-1/2},K^{-1}\delta^{-1}\le|\xi_3|\le \delta^{-1}\right\}.
        \end{equation}
        \item For dyadic numbers $\lambda\in (\delta^{1/2},K^{-1}]$, the $\lambda$-mixed part:
        \begin{equation}            P_{\theta,\lambda}:=\left\{\sum_{i=1}^3 \xi_i \mathbf e_i(\theta):|\xi_1|\le 1,\frac\lambda 2\delta^{-1}\le |\xi_2|\le \lambda \delta^{-1},K^{-1}\delta^{-1}\le|\xi_3|\le \delta^{-1}\right\}.
        \end{equation}       
    \end{enumerate}
    We define decompositions as in  \eqref{eqn_P_theta_decomposition} through \eqref{eqn_f_decomposition}. With this, we have the following pigeonholing argument:

    \begin{claim}\label{twopossiblecases}
        Either one of the following is correct:
        \begin{itemize}
            \item High case: there exists a Borel set $F$ with $\mu(F)\ge \mu(\R^3)/2$, such that
            \begin{equation}
                \left|f_{\mathrm{high}}(x)+\sum_{\text{dyadic } \lambda\in [\delta^{1/2},K^{-1}]}f_\lambda (x)\right|\ge \frac{|f(x)|}2\ge   \frac{c\delta^{\eps-1}}2,\quad \forall x\in F.
            \end{equation}
            \item Low case: there exists a Borel set $F$ with $\mu(F)\ge \mu(\R^3)/2$, such that
            \begin{equation}\label{eqn_low_case_thm7}
                \left|f_{\mathrm{low}}(x)\right|\ge \frac{|f(x)|}2\ge \frac{c \delta^{\eps-1}}2,\quad \forall x\in F.
            \end{equation}
        \end{itemize}
    \end{claim}
    \begin{proof}
        Define $F=\{x\in \mathrm{supp}(\mu):|f_{\mathrm{low}}(x)|\ge |f(x)|/2\}$. If $\mu(F)\ge \mu(\R^3)/2$, then we are in the latter case. Otherwise, $\mu(\supp(\mu)\backslash F)\ge \mu(\R^3)/2$, and for $x\in \supp(\mu)\backslash F$, we are in the former case with $\supp(\mu)\backslash F$ in place of $F$.
    \end{proof}

    \subsection{Analysis of high and mixed parts}

    Assume that we are in the high case of Claim \ref{twopossiblecases}. Note that we have incorporated the mixed parts into the high part which will make our argument much simpler.

    By the lower bound in \eqref{eqn_f_bound_thm7} and the triangle inequality with the fact that we only have $O(|\log \delta|)$ many dyadic scales $\lambda$, we have
    \begin{equation}\label{eqn_60}
        \int |f_{\mathrm{high}}|^pd\mu+\sum_{\text{dyadic } \lambda\in [\delta^{1/2},K^{-1}]}|f_\lambda|^p d\mu\gtrapprox c^p\mu(\R^3)\delta^{p(\eps-1)}.
    \end{equation}
    One good thing about the high case and the mixed cases is that we can use the uncertainty principle. More precisely, we have the following lemma.
    \begin{lem}
        Let $f$ be a bounded integrable function and $\eta$ be an $L^1$-normalised Schwartz function Fourier supported on a ball $B\sub \R^d$ of radius $\delta^{-1}$. Let $a\in (0,d]$ and $\mu$ be a finite Borel measure such that $\rho_a(\mu)<\infty$. Then for each $p\in [1,\infty)$ we have
        \begin{equation}\label{eqn_locally_constant}
            \int |f*\eta(x)|^p d\mu(x)\lesssim \rho_a(\mu)\delta^{a-d}\int |f*\tilde\eta(x)|^p dx,
        \end{equation}
    where $\tilde \eta$ is a Schwartz function whose Fourier transform is supported on $2B$ and equals $1$ on $B$.
    \end{lem}
    \begin{proof}[Proof of lemma]
        Since $f*\tilde\eta*\eta=f* \eta$,  it suffices to prove \eqref{eqn_locally_constant} with the integral on the right hand side replaced by $\int |f|^p dx$, and then replace $f$ by $f*\tilde \eta$. Next, by H\"older's inequality we have $$|f*\eta(x)|^p\lesssim |f|^p*|\eta|(x)\cdot \|\eta\|_{L^1}^{p-1}\lesssim |f|^p*|\eta|(x),$$ and so we may assume $p=1$. By Minkowski's inequality, it then suffices to prove
\begin{equation}\label{fractalintegral}
            \int |\eta(y)|d\mu(y)\lesssim \rho_a(\mu)\delta^{a-d}.
        \end{equation}
        To check that, use that from the support assumption on $\eta$ one can write $\eta(y)=\delta^{-d}\eta_0(\delta^{-1}y)$ where $\eta_0$ is Fourier supported in a ball of radius $1$. Then 
        $$\int |\eta(y)|d\mu(y)\lesssim \delta^{-d}\left\{\int_{|y|\leq \delta} \|\eta_0\|_{\infty}d\mu(y)+\sum_{i\geq 0}\int_{|y|\sim \delta 2^i} |\eta(\delta^{-1} y)|d\mu(y)\right\}$$
        so (\ref{fractalintegral}) follows by the rapid decay of $\eta_0$ and the definition of $\rho_a(\mu)$.
    \end{proof}
    Using the lemma to bound the left hand side of \eqref{eqn_60} (strictly speaking, $f_{\mathrm{high}}$ and $f_\lambda$ should be replaced by functions Fourier supported on a slightly larger slab, but that does not affect the estimates), we have
    \begin{equation}\label{eqn_high_lower_bound_thm_7}
        \int |f_{\mathrm{high}}|^pdx+\sum_{\text{dyadic } \lambda\in [\delta^{1/2},K^{-1}]}|f_\lambda|^p dx\gtrapprox c^p\mu(\R^3)\rho_a(\mu)^{-1}\delta^{p(\eps-1)+3-a}.
    \end{equation}
    Suppose $p\ge 4$. We can just use \eqref{eqn_high_contribution_3D_thm7}, \eqref{eqn_final_mixed_thm7} and \eqref{eqn_f_lambda_upper_bound_thm7} with $t=1$ and $\#\Theta=\delta^{-1}$, noting there is a scaling difference, to get
    \begin{align*}
        &\int |f_{\mathrm{high}}|^pdx+\sum_{\text{dyadic } \lambda\in [\delta^{1/2},K^{-1}]}|f_\lambda|^p dx\\
        &\lesssim_\eps \delta^3\left(K^2 \delta^{-p+1}\left(\sum_{\theta\in \Theta}\#\mathbb T_\theta\right)+\sum_{\text{dyadic } \lambda\in [\delta^{1/2},K^{-1}]}\delta^{-\eps-p+1}\left(\sum_{\theta\in \Theta}\#\mathbb T_\theta\right)\lambda^{\frac p 2-2}\right)\\
        &\lesssim \delta^{-2\sqrt\eps-p+4}\left(\sum_{\theta\in \Theta}\#\mathbb T_\theta\right)\\
        &= \delta^{-2\sqrt\eps-p+4}\#\mathbb T,
    \end{align*}
    where in the second to last line we used \eqref{eqn_defn_K_thm7}.
    
    We may take, say, $p=4$. Using this and \eqref{eqn_high_lower_bound_thm_7}, we thus have, for some large constant $A(\eps)$ independent of $\delta,c$,
    \begin{equation}
        \#\mathbb T\ge A(\eps)^{-1}c^4\mu(\R^3)\rho_a(\mu)^{-1}\delta^{10\sqrt \eps-1-a}.
    \end{equation}
    We have just absorbed the term $\delta^{-4\eps-2\sqrt{\eps}}$ and the logarithmic loss into the term $\delta^{-10\sqrt \eps}$, using $\delta<\delta_0$ if $\delta_0$ is small enough. This gives in the high case
    \begin{equation}
        M(\delta,\eps,c)\leq A(\eps) c^{-4}\delta^{-10\sqrt \eps}.
    \end{equation}    
    
    \subsection{Analysis of low part}
    
    \subsubsection{Maximally incomparable collection of thickened tubes}
    
    For each $T_\theta\in \mathbb T_\theta$, we can view the essential support of $\psi_{T_\theta}*\eta_{\theta,\mathrm{low}}^\vee$ as a thickened tube $\tilde T_\theta$ with dimensions $1\times K\delta\times K\delta$. Let $\tilde {\mathbb T}_\theta$ and thus $\tilde {\mathbb T}$ denote the collection of these thickened tubes. 
    
    With this, we say two tubes $\tilde T$, $\tilde T'$ are $C$-{\it comparable} if one is contained in the $C$-time dilation of the other.
    We then choose from $\tilde {\mathbb T}$ a maximally $1000$-incomparable subcollection; we still denote this by $\tilde {\mathbb T}$ slightly abusing notation. 

    We also perform a preliminary reduction here: for each $\tilde T=\tilde T_\theta\in \tilde {\mathbb T}$, we replace it by some concentric $\tilde T_{\theta'}$ where $\theta'\in K\delta \Z\cap [0,1]$ is the closest lattice point to $\theta$. Note that $\tilde T_\theta$ and $\tilde T_{\theta'}$ are $10$-comparable since $|\gamma'|=1$. Also, the collection of slightly distorted tubes $\{\tilde T_{\theta'}\}$ is $10$-incomparable, since the original collection is $1000$-incomparable. 

    Abusing notation again, we may thus assume without loss of generality that $\tilde {\mathbb T}=\cup_{\theta\in K\delta \Z\cap [0,1]}\tilde {\mathbb T}_\theta$. Also, by losing an absolute constant we may assume $\mathrm{diam}(\Theta)\ll_\gamma 1$, so that we have the relation $|\gamma(\theta)-\gamma(\theta')|\sim |\theta-\theta'|$ for all $\theta,\theta'\in \Theta$.
    
    We have the following observations.
    \begin{enumerate}
        \item \label{obs_1} If two tubes $\tilde T_\theta$, $\tilde T'_\theta$ are from one single $\tilde {\mathbb T}_\theta$, then they are $\ge K\delta$ apart.
        \item \label{obs_2} If two tubes $\tilde T_\theta$, $\tilde T_{\theta'}$ intersect, then $|\theta-\theta'|\ge|\gamma(\theta)-\gamma(\theta')|\ge K\delta$.
        \item \label{obs_3} Every $T\in \mathbb T$ is contained in the 10000-time dilation of some $\tilde T\in \tilde {\mathbb T}$, with possibly different orientations (with difference less than $O(K\delta)$). 
        \item \label{obs_4} As a result of Observations \ref{obs_2} and \ref{obs_3}, each $T\in \mathbb T$ is contained in at most $O(1)$ thickened tubes of the form $10000\tilde T$.
        \item \label{obs_5} The 10000-dilation of every $\tilde T\in \tilde {\mathbb T}$ contains at most $O(K^3)$ tubes $T\in \mathbb T$. Indeed, each $\tilde T$ contains at most $O(K^2)$ tubes $T$ from each single direction. Meanwhile, since $\Theta$ is $\delta$-separated, we have $\{\gamma(\theta):\theta\in \Theta\}$ is also $\delta$-separated. Thus by Observation \ref{obs_3}, each $\tilde T$ could only contain tubes $T\in \mathbb T$ from at most $O(K)$ directions, from which the result follows. 
        \item \label{obs_6} For each $x\in \R^3$, it is contained in at most $O(K^{-1}\delta^{-1})$ thickened tubes from $\tilde {\mathbb T}$. Indeed, for each direction $\theta\in \Theta$, Observation \ref{obs_1} implies that $x$ is contained in at most $1$ thickened tube from $\tilde {\mathbb T}_\theta$. Observation \ref{obs_2} implies that if $|\theta-\theta'|< K\delta$, i.e. $|\gamma(\theta)-\gamma(\theta')|< K\delta$, then $x$ cannot be contained in both $\tilde T_\theta$ and $\tilde T_{\theta'}$. Thus there are at most $O(K^{-1}\delta^{-1})$ directions $\theta$ for which there is a thickened tube in that direction that contains $x$.
    \end{enumerate} 

    \subsubsection{A pigeonholing argument}

    Consider Observation \ref{obs_3} above. A priori, each $10000\tilde T$ may contain different numbers of $T\in \mathbb T$. We now use a pigeonholing argument to reduce to the case that each $10000\tilde T$ contains at least a large portion of tubes in $T\in \mathbb T$.

    Write $\tilde {\mathbb T}$ as a disjoint union
    \begin{equation}
        \tilde {\mathbb T}=\tilde {\mathbb T}_{\mathrm{heavy}}\cup \tilde {\mathbb T}_{\mathrm{light}},
    \end{equation}
where $\tilde {\mathbb T}_{\mathrm{light}}$ is the collection of thickened tubes $\tilde {\mathbb T}$ whose 10000-dilation contains $\le c'K^{3-\sqrt \eps}$ tubes from $\mathbb T$. Here $c'$ is a small constant to be chosen right below. Note that, using Observation \ref{obs_6}, 
\begin{align*}
    \norm{\sum_{\tilde T\in \tilde {\mathbb T}_{\mathrm{light}}}\sum_{T\in \mathbb T:T\sub 10000\tilde T}\psi_{T}*\eta^\vee_{\mathrm{low}}}_\infty
    &\le \sum_{\tilde T\in \tilde {\mathbb T}_{\mathrm{light}}}\sum_{T\in \mathbb T:T\sub 10000\tilde T}\norm{\psi_{T}*\eta^\vee_{\mathrm{low}}}_\infty\\
    &\lesssim (\delta^{-1}K^{-1})(c'K^{3-\sqrt \eps})K^{-2}\\
    &=c'\delta^{-1+\eps}.
\end{align*}
By the rapid decay of $\psi_T*\eta^\vee_{\mathrm{low}}$ and Observations \ref{obs_3} and \ref{obs_4} above, we have
\begin{align*}
    \left|\sum_{\tilde T\in \tilde {\mathbb T}}\sum_{T\in \mathbb T:T\sub 10000\tilde T}\psi_{T}*\eta^\vee_{\mathrm{low}}(x)\right|
    &\sim \sum_{\tilde T\in \tilde {\mathbb T}}\sum_{T\in \mathbb T:T\sub 10000\tilde T}\left|\psi_{T}*\eta^\vee_{\mathrm{low}(x)}\right|\\
    &\sim \sum_{T\in \mathbb T}\left|\psi_{T}*\eta^\vee_{\mathrm{low}}(x)\right|\\
    &\sim \left|\sum_{T\in \mathbb T}\psi_{T}*\eta^\vee_{\mathrm{low}}(x)\right|\\    &=\left|f_{\mathrm{low}}(x)\right|\gtrsim c\delta^{\eps-1},
\end{align*}
for all $x\in F$, using \eqref{eqn_low_case_thm7}. Thus the triangle inequality gives
\begin{equation}
    \left|\sum_{\tilde T\in \tilde {\mathbb T}_{\mathrm{heavy}}}\sum_{T\in \mathbb T:T\sub 10000\tilde T}\psi_{T}*\eta^\vee_{\mathrm{low}}(x)\right|\gtrsim c\delta^{\eps-1},\quad \forall x\in F,
\end{equation}
if $c'$ is chosen to be small enough. Again by rapid decay of $\psi_T*\eta^\vee_{\mathrm{low}}$, we have
\begin{equation}
    \sum_{\tilde T\in \tilde {\mathbb T}_{\mathrm{heavy}}}\sum_{T\in \mathbb T:T\sub 10000\tilde T}1_T(x)\gtrsim cK^2\delta^{\eps-1},\quad \forall x\in F.
\end{equation}
Using Observation \ref{obs_5} above, this implies that 
\begin{equation}\label{eqn_10_tilde_T_lower_bound}
    \sum_{\tilde T\in \tilde {\mathbb T}_{\mathrm{heavy}}}1_{10000\tilde T}(x)\gtrsim cK^{-1}\delta^{\eps-1},\quad \forall x\in F.
\end{equation}

\subsubsection{Analysis at new scale}

Our goal in this step is to let $\tilde {\mathbb T}$ satisfy the hypothesis of Theorem \ref{thm_thm7} at a larger scale $10000K\delta=\delta^{1-\sqrt \eps}$.

Let $\Theta'$ denote the lattice $K\delta \Z\cap [0,1]$. Recall that by our choice, $\tilde {\mathbb T}_{\mathrm{heavy}}$ can be written as $\cup_{\theta'\in \Theta'}\tilde {\mathbb T}_{\theta'}$, where each $\tilde {\mathbb T}_{\theta'}$ is a collection of disjoint $K\delta\times K\delta\times 1$ tubes pointing in the direction $\gamma(\theta')$. By a trivial tiling argument and losing an absolute constant, we can further assume $\Theta'=10^4 K\delta\Z \cap [0,1]$.  

Write
\begin{equation}
    \delta^{-1+\eps}(10^4K)^{-1}=(10^4 K\delta)^{-1+\tilde \eps},\quad \tilde \eps:=\frac \eps {1-\sqrt \eps}.
\end{equation}
Then we have, for some absolute constant $c_0$,
\begin{equation}
    \sum_{\tilde T\in \tilde {\mathbb T}_{\mathrm{heavy}}}1_{10^4\tilde T}(x)\geq c_0c (10^4 K\delta)^{-1+\tilde \eps},\quad \forall x\in \supp(\mu).
\end{equation}
Now we can apply the definition of $M(10^4 K\delta,\tilde \eps,c_0c)$ in \eqref{eqn_defn_M_delta} to get that
\begin{equation}
    \#\tilde {\mathbb T}_{\mathrm{heavy}}\ge M(10^4 K\delta,\tilde \eps,c_0c)^{-1}\mu(\R^3)\rho_a(\mu)^{-1}(10^4 K\delta)^{-1-a}.
\end{equation}
By the definition of $\tilde {\mathbb T}_{\mathrm{heavy}}$, we have
\begin{align*}
    \# \mathbb T
    &\ge c(10^4 K)^{3-\sqrt \eps}M(10^4 K\delta,\tilde \eps,c_0c)^{-1}\mu(\R^3)\rho_a(\mu)^{-1}(10^4 K\delta)^{-1-a}\\
    &\ge c (10^4 K)^{2-a-\sqrt \eps}M(10^4 K\delta,\tilde \eps,c_0c)^{-1}\mu(\R^3)\rho_a(\mu)^{-1}\delta^{-1-a}.
\end{align*} 
Thus in the low case, we have
\begin{align*}
    M(\delta,\eps,c)
    &\le M(10^4 K\delta,\tilde \eps,c_0c) c^{-1}(10^4 K)^{a-2+\sqrt \eps}\\
    &\le M(10^4 K\delta,\tilde \eps,c_0c)c^{-1}\delta^{-\eps}\\
    &\le M(\delta^{1-\sqrt{\eps}},\tilde \eps,c_0c)\delta^{-2\eps},
\end{align*}
where we have used $a-2\le 0$ and used $\delta<\delta_0$ taking $\delta_0$ to be small enough. Thus the result follows.

Combining the low case and the high and mixed case, we are done.

\subsection{Corollary 2}
We now restate and prove Corollary 2 in \cite{GGGHMW2022}.

\begin{cor}\label{cor_cor_2}
    Let $a\in (0,3]$ and $a_1\in [0,\min\{2,a\})$. Let $\mu$ be a finite non-zero Borel measure supported on the unit ball in $\R^3$ with $\rho_a(\mu)\le 1$. Let $\delta\in (0,1)$ and denote by $\Theta$ the $\delta$-lattice in $[0,1]$. For each $\theta\in \Theta$, let $\mathbb S_\theta$ be a disjoint collection of at most $c\mu(\R^3)\delta^{-a_1}$ squares of side length $\delta$ in $\pi_\theta(\R^3)$. Then there exists $\eps\in (0,1)$, depending only on $\gamma,c,a,a_1$, such that
    \begin{equation}\label{eqn_70}
        \delta \sum_{\theta\in \Theta}\pi_{\theta\#} \mu\left(\bigcup \mathbb S_\theta\right)\le C(a,a_1)\mu(\R^3)\delta^\eps.
    \end{equation}
    
\end{cor}
\begin{proof}
    We argue by contradiction. Suppose that there exist $a,a_1$ such that for every $\eps\in (0,1)$ and every large constant $C>10$, there exist $\delta>0$ and a disjoint collection $\mathbb S_\theta$ of at most $\mu(\R^3)\delta^{-a_1}$ squares of radius $\delta$ in $\pi_\theta(\R^3)$ for each $\theta\in \Theta$, such that 
    \begin{equation}\label{eqn_70_contradiction}
        \delta \sum_{\theta\in \Theta}\pi_{\theta\#} \mu\left(\bigcup \mathbb S_\theta\right)>C\mu(\R^3)\delta^\eps.
    \end{equation}

    Note that for each $S_\theta\in \mathbb S_\theta$, $\pi_\theta^{-1}(S_\theta)\cap B^3(0,1)$ is a $\delta$-tube. We denote these tubes by
    \begin{equation}
        \mathbb T_\theta:=\{\pi_\theta^{-1}(S_\theta)\cap B^3(0,1):S_\theta\in \mathbb S_\theta\},
    \end{equation}
    and write $\mathbb T:=\cup_{\theta\in \Theta}\mathbb T_\theta$. Thus \eqref{eqn_70_contradiction} becomes
    \begin{equation}\label{eqn_70_contradiction_2}
        \delta\int \sum_{T\in \mathbb T}1_T(x)d\mu(x)=\delta \sum_{T\in \mathbb T}\mu(T)=\delta \sum_{\theta\in \Theta} \mu\left(\bigcup_{T_\theta\in \mathbb T_\theta}T_\theta\right)>C\mu(\R^3)\delta^\eps.
    \end{equation}    
    We perform a pigeonholing argument here. Denote $S(x):=\delta \sum_{T\in \mathbb T}1_T(x)$, which is bounded above by $\delta \#\Theta=1$. Define
    \begin{equation}
        F:=\left\{x\in \mathrm{supp}(\mu):\delta \sum_{T\in \mathbb T}1_T(x)\ge \frac{C\delta^\eps}2\right\}.
    \end{equation}
    We thus have
    \begin{align*}
        \int S(x)d\mu(x)
        &=\int_F S(x)d\mu(x)+\int_{F^c}S(x)d\mu(x)\\
        &\le \int_F 1 d\mu(x)+\int_{F^c}\frac {C\delta^{\eps}}2d\mu(x)\\
        &\le \mu(F)+\frac {C\delta^{\eps}}2\mu(F^c).
    \end{align*}
    Using \eqref{eqn_70_contradiction_2}, this implies that 
    \begin{equation}
        \mu(F)\ge \frac {C\delta^{\eps}}2\mu(\R^3).
    \end{equation}
    Now we define a new measure $\nu$ by
    \begin{equation}
        \nu(A):=\mu(A\cap F),
    \end{equation}
    supported on $F\sub B^3(0,1)$, with $\rho_{\min\{2,a\}}(\nu)\le\rho_a(\nu)\le \rho_a(\mu)\le 1$. Also, the definition of $F$ means that
    \begin{equation}
        \sum_{T\in \mathbb T}1_T(x)\ge \frac{C}2\delta^{\eps-1},\quad \forall x\in \mathrm{supp}(\nu).
    \end{equation}
    Now we may apply Theorem \ref{thm_thm7} with $\nu$ in place of $\mu$ and $\min\{2,a\}$ in place of $a$ to obtain
    \begin{equation}\label{eqn_number_T_lower_bound_cor_2}
        \#\mathbb T\gtrsim_{\eps,a}\nu(\R^3)\delta^{-1-\min\{2,a\}+100\sqrt \eps}\ge \frac {C\delta^{\eps}}2\mu(\R^3)\delta^{-1-\min\{2,a\}+100\sqrt \eps}.
    \end{equation}
    However, since each $\#\mathbb S_\theta\le c\mu(\R^3)\delta^{-a_1}$ for each $\theta\in \Theta$, we have on the other hand
    \begin{equation}
        \#\mathbb T\le c\mu(\R^3)\delta^{-1-a_1}.
    \end{equation}
    Compared with \eqref{eqn_number_T_lower_bound_cor_2}, we have there is a constant $C_{\eps,c,a}$ such that
    \begin{equation}
        C\le C_{\eps,c,a}\delta^{\min\{2,a\}-a_1-\eps-100\sqrt \eps}.
    \end{equation}
    Choosing $\eps>0$ small enough such that $\min\{2,a\}-a_1-\eps-100\sqrt \eps>0$ and then $C>C_{\eps,c,a}$, we arrive at a contradiction.
\end{proof}

\subsection{Corollary 3}
We now restate and prove Corollary 3 in \cite{GGGHMW2022}.
\begin{cor}\label{cor_cor_3}
    Let $a\in [2,3]$ and $a_1\in [0,2)$. Let $\mu$ be a finite non-zero Borel measure supported on the unit ball in $\R^3$ with $\rho_a(\mu)<\infty$. Let $\delta\in (0,1)$ and denote by $\Delta$ the $\delta^{1/2}$-lattice in $[0,1]$. For each $\theta\in \Delta$, let $\mathbb P_\theta$ be a disjoint collection of at most $\mu(\R^3)\delta^{-\frac {a_1+1}2}$ rectangles of dimensions $\delta\times \delta^{1/2}$ in $\pi_\theta(\R^3)$ whose longer sides point in the $\gamma'(\theta)=\mathbf e_2(\theta)$ direction. Then there exists $\eps\in (0,1)$, depending only on $\gamma,a,a_1$, such that
    \begin{equation}\label{eqn_78}
        \delta^{1/2} \sum_{\theta\in \Delta}\pi_{\theta\#} \mu\left(\bigcup \mathbb P_\theta\right)\le C(a,a_1)\rho_a(\mu)\mu(\R^3)\delta^\eps.
    \end{equation}
\end{cor}
\begin{proof}
    We may assume $\rho_a(\mu)= 1$ without loss of generality. To apply Corollary \ref{cor_cor_2}, we need to construct for each $\theta'\in \Theta=\delta \Z\cap [0,1]$ a collection of squares $\mathbb S_{\theta'}$ of radius $\delta$ in $\pi_{\theta'}(\R^3)$. We proceed as follows. For each $\theta'\in \Theta$, we choose $\theta\in \Delta$ to be the closest to $\theta'$, so that $|\theta'-\theta|<\delta^{1/2}$. We denote this relation by $\theta'\prec \theta$. We then consider the set of planks
    \begin{equation}
      \mathbb{U}_{\theta}:=  \{U_\theta=\pi_{\theta}^{-1}(P_{\theta})\cap B^3(0,1):P_{\theta}\in \mathbb P_\theta\}.
    \end{equation}
    We partition each $100U_\theta$ into $(\delta/2)\times (\delta/2)\times 1$ tubes $T_\theta$ with direction $\gamma(\theta)$ and denote the collection of all these tubes by $\mathbb T_\theta$. Then we define
    \begin{equation}
        \mathbb S_{\theta'}:=\{\pi_{\theta'}(T_\theta):T_\theta\in \mathbb T_\theta\}.
    \end{equation}
    Strictly speaking, each element of $\mathbb S_{\theta'}$ is a possibly slightly distorted parallelogram, but we can slightly enlarge it to become a square of side length $\delta$. The slightly enlarged collection of squares will be at most $100$-overlapping since $|\gamma(\theta')-\gamma(\theta)|\le |\theta'-\theta|<\delta^{1/2}$; by abuse of notation we will still denote it by $\mathbb S_{\theta'}$. We also have
    \begin{equation}
        \# \mathbb S_{\theta'}=\# \mathbb T_\theta\sim \delta^{-1/2}\#\mathbb P_\theta\lesssim \mu(\R^3)\delta^{-\frac {a_1+2}2}.
    \end{equation}     
    Since $\frac{a_1+2}2<2=\min\{2,a\}$, we may apply Corollary \ref{cor_cor_2} to obtain
    \begin{equation}
        \delta \sum_{\theta'\in \Theta}\pi_{\theta'\#} \mu\left(\bigcup \mathbb S_{\theta'}\right)\lesssim \mu(\R^3)\delta^\eps,
    \end{equation}
    and so

    \begin{align*}
        \mu(\R^3)\delta^\eps
        &\gtrsim \delta \sum_{\theta\in \Delta}\sum_{\theta'\prec \theta}\mu\left(\pi_{\theta'}^{-1}(\bigcup_{T_\theta\in \mathbb T_\theta}\pi_{\theta'}(T_\theta))\right)\\
        &=\delta \sum_{\theta\in \Delta}\sum_{\theta'\prec \theta}\mu\left(\pi_{\theta'}^{-1}(\pi_{\theta'}(\cup_{U_{\theta}\in \mathbb{U}_{\theta}} 100U_\theta))\right)\\
        &\geq \delta \sum_{\theta\in \Delta}\sum_{\theta'\prec \theta}\mu\left(\pi_{\theta}^{-1}(\pi_{\theta}(\cup_{U_{\theta}\in \mathbb{U}_{\theta}}U_\theta))\right)\\
        &\sim \delta^{1/2}\sum_{\theta\in \Delta}\mu\left(\pi_{\theta}^{-1}(\cup_{P_{\theta}\in \mathbb{P}_{\theta}}P_\theta)\right)\\
        &=\delta^{1/2} \sum_{\theta\in \Delta}\pi_{\theta\#} \mu\left(\bigcup \mathbb P_\theta\right),
    \end{align*}    
    where the third line follows from a simple geometric observation.
\end{proof}

\section{Proof of Theorem 8}\label{sec_proof_thm8}
We now prove Theorem \ref{thm_thm8}. 

To avoid technical issues and to make the presentation slightly cleaner, we only restrict ourselves to the case of a standard light cone. More precisely, we consider $I_0:=[0,\sqrt 2\pi]$ and $\mathbf e_1=\gamma: I_0\to \mathbb S^2$ given by
\begin{equation}\label{lightcone1}
    \gamma(\theta)=\frac 1 {\sqrt 2}\left(\cos \sqrt 2\theta,\sin \sqrt 2 \theta,1\right).
\end{equation}
In this case, it is easy to see that
\begin{equation}\label{lightcone2}
    \begin{split}
        \mathbf e_2(\theta)&=\left(-\sin \sqrt 2\theta,\cos \sqrt 2 \theta,0\right),\\
    \mathbf e_3(\theta)&=\frac 1 {\sqrt 2}\left(-\cos \sqrt 2\theta,-\sin \sqrt 2 \theta,1\right),\\
    \uptau(\theta)&=1.
    \end{split}
\end{equation}

The main idea is to decompose the frequency space into ``good" and ``bad" parts according to the cone $\Gamma=\{r\mathbf e_3(\theta):r\in \R,\theta\in I_0\}$, and estimate them separately. 

We record some important geometric facts about the light cone. 

\begin{prop}\label{prop_light_cone}
    Let $\xi=(x_1,x_2,x_3)\in \R^3$. Then there exists a $\xi':=r\mathbf e_3(\omega)\in \Gamma$ such that $|\xi-\xi'|=\mathrm{dist}(\xi,\Gamma)$, and also $\xi-\xi'$ is parallel to $\mathbf e_1(\omega)$. As a result, $\mathrm{dist}(\xi,\Gamma)=|\xi\cdot \mathbf e_1(\omega)|$. 
    
    Meanwhile, we also have
    $\mathrm{dist}(\xi,\Gamma)=2^{-1/2}|x_3-\sqrt{x_1^2+x_2^2}|$.
    \begin{proof}
        This is easily seen using the cylindrical coordinates in $\R^3$.
    \end{proof}

\end{prop}

\subsection{Construction of covering planks}

Our goal is to construct planks in three steps. We give the main idea first. Given $j\ge k\ge 0$, Step 1 is to construct a canonical family $\Lambda_{j,k}$ of planks $P_{\theta,j,k}$ that cover the set
\begin{equation}
    \{\xi\in \R^3:|\xi|\sim 2^j,\mathrm{dist}(\xi,\Gamma)\sim 2^{j-k}\},
\end{equation}
except for $k=j$, in which case we want to cover 
$$ \{\xi\in \R^3:|\xi|\sim 2^j,\mathrm{dist}(\xi,\Gamma)\lesssim 1\}.$$
Each plank $P_{\theta,j,k}$ essentially has normal, aperture and radial lengths $2^{j-k}$, $2^{j-k/2}$, $2^{j}$, respectively.

In Step 2, we define $\Lambda_j=\cup_{k\le j}\Lambda_{j,k}$, so that $\Lambda_j$ forms a canonical family of planks that cover the $2^j$ neighbourhood of $\{\xi\in \R^3:|\xi|\sim 2^j\}$.

In Step 3, we define $\Lambda=\cup_{j\in \N} \Lambda_j$, so that $\Lambda$ forms a covering of the entire $\R^3$.

We will then define wave packets adapted to the planks above, and perform subsequent Fourier analysis.

Now we come to the details and define the planks $P_{\theta,j,k}$. Fix $\theta\in I_0$. For $0\le k<j$, we define

\begin{align}
    P_{\theta,j,k}
    &:=\Bigg\{\sum_{i=1}^3\xi_i \mathbf e_i(\theta): 2^{j-k-10}\le |\xi_1|\le  2^{j-k+10}, \nonumber\\
    &|\xi_2|\le  2^{j-\frac k 2-100},  2^{j-10}\le |\xi_3|\le 2^{j+10} \Bigg\}.
\end{align}
For $0\le k=j$, we define
\begin{align}
    P_{\theta,j,j}&=\Bigg\{\sum_{i=1}^3\xi_i \mathbf e_i(\theta):|\xi_1|\le 2^{10}, |\xi_2|\le 2^{\frac j 2-100}, 2^{j-10}\le|\xi_3|\le 2^{j+10} \Bigg\}.
\end{align}
Strictly speaking, each $P_{\theta,j,k}$ for $k<j$ is a union of four planks. However, by symmetry we may assume $\xi_1>0$ and $\xi_3>0$, and abusing notation we still denote this by $P_{\theta,j,k}$. Similar for $P_{\theta,j,j}$.

Heuristically, $\xi_1\sim 2^{j-k}$ (or $|\xi_1|\lesssim 1$ when $j=k$) is the distance between the plank $P_{\theta,j,k}$ and $\Gamma$, $|\xi_2|\lesssim 2^{j-k/2}$ is the angular aperture, and $\xi_3\sim 2^j$ is the distance from the origin. The precise statement is as follows.
\begin{lem}\label{lem_distance_planks}
    Fix $\theta\in I_0$. Let $0\le k\le j$ and $\xi\in P_{\theta,j,k}$. Then the distance between $\xi$ and the cone $\Gamma$ is at most $2^{j-k+10}$. Furthermore, if $j>k$, then the distance between $\xi$ and the cone $\Gamma$ is at least $2^{j-k-40}$.
   
\end{lem}
\begin{proof}

Fix $\xi\in P_{\theta,j,k}$ and write $\xi=\sum_{i=1}^3 \xi_i \mathbf e_i(\theta)=(x_1,x_2,x_3)$ where $x_i$ can be found using (\ref{lightcone1}) and (\ref{lightcone2}). Then Proposition \ref{prop_light_cone} allow us to compute directly
\begin{align*}
\text{dist}(\xi,\Gamma)=
    &\frac{1}{2}\bigg|((\xi_1+\xi_3)-\bigg((\xi_1\cos \sqrt 2\theta-\sqrt 2\xi_2 \sin \sqrt 2\theta-\xi_3 \cos \sqrt 2\theta)^2\\
    &+((\xi_1\sin \sqrt 2\theta+\sqrt 2\xi_2 \cos \sqrt 2\theta-\xi_3 \sin \sqrt 2\theta)^2\bigg)^{1/2}\bigg|\\
    &=\frac{1}{2}\left|\xi_1+\xi_3-\sqrt{2\xi_2^2+(\xi_3-\xi_1)^2}\right|\\
    &=|2\xi_1\xi_3-\xi_2^2|\left|\xi_1+\xi_3+\sqrt{2\xi_2^2+(\xi_3-\xi_1)^2}\right|^{-1}.
\end{align*}
In particular, the term inside the absolute value is positive for $\xi\in P_{\theta,j,k}$, and is a increasing function in $|\xi_2|$. For $j>k$, using the upper and lower bounds of $|\xi_i|$, we thus have 
\begin{equation}
    \mathrm{dist}(\xi,\Gamma)\in [2^{j-k-40},2^{j-k+10}],
\end{equation}
where the upper bound was obtained by considering $\xi_2=0$. The case $j=k$ is similar.
 
\end{proof}

Now we are ready to define the covering planks. Let $c=100^{-100}$. For each $k\ge 0$, define the lattice 
\begin{equation}\label{eqn_defn_Theta_k}
    \Theta_k:=c2^{-k/2}\N\cap I_0.
\end{equation}
Define
\begin{equation}
    \Lambda_{j,k}:=\{P_{\theta,j,k}:\theta\in \Theta_k\},
\end{equation}
and then
\begin{equation}\label{eqn_defn_Lambda}
    \Lambda_j:=\bigcup_{k=0}^j \Lambda_{j,k},\quad \Lambda:=\bigcup_{j=0}^\infty \Lambda_j. 
\end{equation}
Note that we always have
\begin{equation}
    \mathrm{dist}(\xi,\Gamma)\le 2^{-1/2}|\xi|.
\end{equation}

Indeed, write $\xi=(x_1,x_2,x_3)$. By radial symmetry we may assume $x_2=0$, and by dilation symmetry we can assume $x_1^2+x_3^2=1$. Then the distance is maximized exactly when $x_1=0$ or $x_3=1$, and the distance is $2^{-1/2}$.

For $C\ge 1$, we also define for $j>k\ge 0$
\begin{align}
    C\Gamma_{j,k}
    &:=\{\xi\in \R^3:C^{-1}2^{j-1}\le |\xi|\le C 2^{j},\nonumber\\
    &\mathrm{dist}(\xi,\Gamma)\in [C^{-1}2^{-k-3/2}|\xi|,C  2^{-k-1/2}|\xi|]\},
\end{align}
and for $j=k$
\begin{equation}
    C \Gamma_{j,j}:=\{\xi\in \R^3:C^{-1} 2^{j-1}\le |\xi|\le C 2^{j},\mathrm{dist}(\xi,\Gamma)\in [0,C2^{-j-1/2}|\xi|]\}.
\end{equation}

When $C=1$, we simply denote $\Gamma_{j,k}$. We also denote 
$$C\Gamma_j=\{\xi\in \R^3\colon C^{-1}2^{j-1}\leq |\xi| \leq C2^{j} \}.$$

\begin{lem}\label{lem_planks_j_k}
For all $j\ge k\ge 0$, we have
    \begin{equation}
        \Gamma_{j,k}\sub \bigcup_{\theta\in \Theta_k} P_{\theta,j,k}\sub A\Gamma_{j,k},
    \end{equation}    
    for some absolute constant $A$. Moreover, the family $\Lambda_{j,k}$ has $O(1)$-overlap.
\end{lem}
\begin{proof}
    The second inclusion follows directly from Lemma \ref{lem_distance_planks} if $A$ is large enough. The $O(1)$-overlap follows from essentially the same proof of that of Part \ref{item_2_same_and_distinct} of Lemma \ref{lem_same_and_distinct} in Section \ref{sec_covering_lemmas}.
    
    To prove the first inclusion, let $\xi\in \Gamma_{j,k}$. Take $\xi':=r\mathrm e_3(\omega)\in \Gamma$ that attains $\mathrm{dist}(\xi,\Gamma)$. Take $\theta\in \Theta_k$ be closest to $\omega$, so $|\theta-\omega|\le c 2^{-k/2}$. It suffices to prove that $\xi\in P_{\theta,j,k}$.

   We first consider the normal direction. For $j>k$, by Proposition \ref{prop_light_cone}, we have
   \begin{equation}
       |\xi\cdot \mathbf e_1(\omega)|=|(\xi-\xi')\cdot \mathbf e_1(\omega)|=\mathrm{dist}(\xi,\Gamma)\ge 2^{-k-3/2}|\xi|\ge 2^{j-k-5/2}.
   \end{equation}
   Also,
   \begin{equation}
       |\xi\cdot \mathbf e_2(\omega)|=|(\xi-\xi')\cdot \mathbf e_2(\omega)|\le |\xi-\xi'|\leq 2^{-k-1/2}|\xi|.
   \end{equation}   
   Thus by Taylor expansion,
    \begin{align*}
        |\xi\cdot \mathbf e_1(\theta)|
        &\ge |\xi\cdot \mathbf e_1(\omega)|-|\xi\cdot \mathbf e_2(\omega)||\omega-\theta|-2|\xi||\omega-\theta|^2\\
        &\ge 2^{j-k-5/2}-c 2^{j-3k/2-1/2}-2c^2 2^{j-k}\\
        &\ge 2^{j-k-10}.
    \end{align*}
    Similarly, for $j\ge k$, we also get an upper bound $|\xi\cdot \mathbf e_1(\theta)|\le 2^{j-k+10}$.  
    
    Now we come to the tangential direction. By Proposition \ref{prop_light_cone}, we have
    \begin{equation}
        \xi\cdot \mathbf e_2(\omega)
        =(\xi-\xi')\cdot \mathbf e_2(\omega)+\xi'\cdot \mathbf e_2(\omega)=0.
    \end{equation}
    By Taylor expansion,
    \begin{equation}
        |\xi\cdot \mathbf e_2(\theta)|\le 2|\xi| |\theta-\omega|\le 2c 2^{j-k/2}\le 2^{-100}2^{j-k/2}.
    \end{equation}
    Lastly, we consider the radial direction. By Proposition \ref{prop_light_cone}, we have 
    \begin{equation}
        |\xi\cdot \mathbf e_3(\omega)|=|\xi'\cdot \mathbf e_3(\omega)|=|\xi'|.
    \end{equation}
    Also, recalling that $\mathrm{dist} (\xi,\Gamma)\le 2^{-1/2}|\xi|$, we have
    \begin{equation}
        |\xi'|\ge |\xi|-|\xi'-\xi|\ge (1-2^{-1/2})|\xi|\ge 2^{j-5}.
    \end{equation}
    Thus by Taylor expansion,
    \begin{equation}
        |\xi\cdot \mathbf e_3(\theta)|
        \ge |\xi\cdot \mathbf e_3(\omega)|-|\xi| |\theta-\omega|
        \ge 2^{j-5}-c2^{j-k/2}\ge 2^{j-10}.
    \end{equation}
    Similarly, we can get an upper bound $|\xi\cdot \mathbf e_3(\theta)|\le 2^{j+10}$.
\end{proof}

\begin{lem}\label{lem_planks_j}
    We have
    \begin{equation}
        \Gamma_j\sub \bigcup \Lambda_j \sub A\Gamma_j.
    \end{equation}
    Moreover, the family $\Lambda_j$ has $O(1)$-overlap.
\end{lem}
\begin{proof}
The inclusion follows by Lemma \ref{lem_planks_j_k} and taking union $\cup_{k=0}^j$. The bounded overlap of $\Lambda_j$ follows readily from the bounded overlap of each $\Lambda_{j,k}$ established in Lemma \ref{lem_planks_j_k} and the bounded overlap of $\{A\Gamma_{j,k}:0\le k\le j\}$.
\end{proof}

\begin{lem}
    The family of planks $\Lambda$ covers $\R^3\backslash B^3(0,1/2)$ and has $O(1)$-overlap.
\end{lem}
\begin{proof}
The bounded overlap of $\Lambda$ follows from the bounded overlap of each $A\Lambda_j$ and the bounded overlap of $\{\Gamma_j:j\ge 0\}$.

To prove the covering, by Lemma \ref{lem_planks_j}, it suffices to prove that
\begin{equation}
    \bigcup_{j=0}^\infty \bigcup_{k=0}^j \Gamma_{j,k}\supseteq \R^3\backslash B^3(0,1/2).
\end{equation}
To see this, given any $\xi\in \R^3\backslash B^3(0,1/2)$. Take $j\ge 0$ such that $|\xi|\in [2^{j-1},2^{j}]$. Also, since $\mathrm{dist}(\xi,\Gamma)\le 2^{-1/2}|\xi|$, take $k'\ge 0$ such that 
\begin{equation}
    \mathrm{dist}(\xi,\Gamma)\in [2^{-k'-3/2}|\xi|,2^{-k'-1/2}|\xi|].
\end{equation}
If $j\ge k'$, then we just take $k=k'$. If $j<k'$, then we take $k=j$.
\end{proof}

\subsection{Construction of wave packets}
Consider the collection $\Lambda$ of (infinitely many) planks we constructed in \eqref{eqn_defn_Lambda}. 

We first define a smooth partition of unity adapted to the family $\Lambda$: for each $P\in \Lambda$, let $\psi_P$ be a bump function that $\sim 1$ on $P$ and is supported in $2P$. Also, let $\psi_0$ be a bump function adapted to the ball $B^3(0,1)$. Moreover, they can be chosen in a way that 
\begin{equation}
    \psi_0+\sum_{P\in \Lambda}\psi_P=1.
\end{equation}
For technical issues, we need to introduce a tiny exponent $\beta>0$ to be determined.

For each $P\in \Lambda$, let $\mathbb T_P$ be a tiling of $\R^3$ by planks $T$ nearly dual to $P=P_{\theta,j,k}$, that is, having dimensions
\begin{equation}\label{eqn_wave_packet_T}
    2^{k-j+k\beta},\quad 2^{k/2-j+k\beta}, \quad 2^{-j+k\beta}
\end{equation}
in the normal, tangential and radial directions, respectively. Strictly speaking, they are not duals due to the tiny factor $\beta$. This is intended so that we can use repeated integration by parts argument to prove Lemma \ref{lem_almost_orthogonality} below.

For each $T\in \mathbb T_P$, we choose a Schwartz function $\eta_T$ which $\sim 1$ on $T$ and is Fourier supported on the plank centred at $0$, with sides parallel to $P$ but with dimensions $2^{j-k-k\beta}$, $2^{j-k/2-k\beta}$ and $2^{j-k\beta}$, respectively, such that we also have
\begin{equation}
    \sum_{T\in \mathbb T_P}\eta_T=1.
\end{equation}
Now for each $P\in \Lambda$ and each $T\in \mathbb T_P$, we define the wave packet
\begin{equation}
    M_T\mu(x):=\eta_T(x)(\mu*\psi_P^\vee)(x)=\eta_T(x)\int \psi_P^\vee(x-y)d\mu(y).
\end{equation}
We record some estimates on the wave packets.
\begin{lem}\label{lem_rapid_decay_plank}
    For $P\in \Lambda_{j,k}$, $T\in \mathbb T_P$ and any $\theta\in I_0$, we have
    \begin{equation}
    \begin{split}
        \norm {\pi_{\theta\#}M_T\mu}_{L^1(\mathcal H^2)}&\lesssim_{N,\beta}\int w_{T,N}(y)d\mu(y)\\
        &\lesssim_{N,\beta} 2^{-kN}\mu(\R^3)+ \mu(T),\quad \forall N\ge 1,
    \end{split}
    \end{equation}
    where $w_{T,N}(x)$ is any weight function adapted to $T$ (refer to \eqref{eqn_weight_T_2D}).
\end{lem}
\begin{proof}
Given a continuous function $f:\R^3\to \C$ of compact support and $\theta\in I_0$, we may compute, using the definition of push-forward measure via Riesz representation theorem
\begin{equation}
    \pi_{\theta\#}f(x)=\int_{\R}f(x+t\gamma(\theta))dt,\quad \forall x\in \pi_{\theta\#}(\R^3).
\end{equation}
Thus
\begin{align*}
    \pi_{\theta\#}M_T\mu(x)
    &=\int_{\R}M_T\mu(x+t\gamma(\theta))dt\\
    &=\int_{\R}\int \eta_T (x+t\gamma(\theta))\psi_P^\vee (x+t\gamma(\theta)-y)d\mu(y)dt.
\end{align*}
By Fubini's theorem and writing $z=x+t\gamma(\theta)$, we have
    \begin{equation}
        \int |\pi_{\theta\#}M_T\mu(x)|d\mathcal H^2(x)= \int \int_{\R^3} \eta_T(z)\psi_P^\vee(z-y)dz d\mu(y).
    \end{equation}
    Since $\psi_P^\vee$ is an $L^1$-normalized Schwartz function essentially supported on a slightly smaller rectangle than $T$ (due to the $\beta$ factor), an application of Proposition \ref{prop_weights_convolution} (generalized to $\R^3$ in an apparent way) gives $|\psi_P^\vee*\eta_T(y)|\lesssim_N w_{T,N}(y)$ for any choice of $N$. Thus the result follows.

\end{proof}

The next is an almost orthogonality lemma.
\begin{lem}\label{lem_almost_orthogonality}
    For each $P\in \Lambda_{j,k}$, denote by $\theta_P$ the angle in $I_0$ such that $P=P_{\theta_P,j,k}$. Then for any $\theta\in I_0$ such that $|\theta-\theta_P|\ge C2^{-k/2}$ (here $C$ is a large enough absolute constant) and any $x\in \pi_\theta(B^3(0,1))$, we have
    \begin{equation}
        \sum_{T\in \mathbb{T}_P}|\pi_{\theta\#}M_T \mu(x)|\lesssim_N 2^{-kN}|P|\mu(\R^3),\quad \forall N\ge 1.
    \end{equation}    
\end{lem}

\begin{proof}
It suffices to prove that for any rectangle $T\in \mathbb T_{P}$ centred at $c_T$ and any $x\in \pi_\theta(B^3(0,1))$, we have
\begin{equation}
        |\pi_{\theta\#}M_T \mu(x)|\lesssim_{\beta,N}2^{-kN}(1+|x-c_T|2^{j-k-k\beta})^{-N}|P|\mu(\R^3),\quad \forall N\ge 1.
    \end{equation}
    Using the inverse Fourier transform formula, we compute directly that
    \begin{align*}
        \pi_{\theta\#}M_T \mu(x)
        &=\int \int \eta_T(x+t\gamma(\theta)) \widehat \mu(\xi)\psi_P(\xi) e^{i(x+t\gamma(\theta))\cdot \xi}  dt d\xi\\
        &\le\int |\widehat \mu(\xi)\psi_P(\xi)|  \left|\int \eta_T(x+t\gamma(\theta)e^{it\gamma(\theta))\cdot \xi}dt\right| d\xi.
    \end{align*}
    Using $\|\hat \mu\|_\infty\le \mu(\R^3)$, it suffices to prove that
    \begin{equation}\label{eqn_rapid_decay_lem_5}
        \left|\int \eta_T(x+t\gamma(\theta))e^{it\gamma(\theta)\cdot \xi}dt\right|\lesssim_{\beta,N}2^{-kN}(1+|x-c_T|2^{j-k-k\beta})^{-N},\quad \forall N\ge 1.
    \end{equation}
    To prove this, we first prove the following claim.
    \begin{claim}\label{claim_exponent_rapid_decay}
        Let $\theta\in I_0$ and $|\theta-\theta_P|\ge C2^{-k/2}$. If $C$ is large enough, then we have \begin{equation}
        |\xi\cdot \gamma(\theta)|\gtrsim  (\theta-\theta_P)^2 2^{j},\quad \forall \xi\in P.
    \end{equation}
    \end{claim}
    \begin{proof}[Proof of claim]
        Write $\xi=\sum_{i=1}^3 \xi_i \mathbf e_i(\theta_P)\in P$. For $i=3$, by Taylor expansion at the centre $\theta_P$,
    \begin{align*}
        \mathbf e_3(\theta_P)\cdot \gamma(\theta)
        &=\mathbf e_3(\theta_P)\cdot \bigg(\mathbf e_1(\theta_P)+\mathbf e_2(\theta_P)(\theta-\theta_P)\\
        &+(-\mathbf e_1(\theta_P)+\uptau(\theta_P)\mathbf e_3(\theta_P))(\theta-\theta_P)^2\bigg)+o((\theta-\theta_P)^2)\\
        &\sim (\theta-\theta_P)^2.
    \end{align*}
    Since $\xi_3\sim 2^{j}$, we have 
    \begin{equation}
        |\xi_3\mathbf e_3(\theta_P)\cdot \gamma(\theta)|\gtrsim (\theta-\theta_P)^2 2^{j}.
    \end{equation}
    We also have 
    \begin{equation}
        |\xi_1\mathbf e_1(\theta_P)\cdot \gamma(\theta)|\lesssim 2^{j-k},
    \end{equation}
    and
    \begin{equation}
        |\xi_2\mathbf e_2(\theta_P)\cdot \gamma(\theta)|\lesssim  2^{j-k/2}|\theta-\theta_P|. 
    \end{equation}
    If $C$ is large enough, then
    \begin{equation}
        |\xi\cdot \gamma(\theta)|\sim |\xi_3\mathbf e_3(\theta_P)\cdot \gamma(\theta)|\gtrsim (\theta-\theta_P)^2 2^{j}.
    \end{equation}
    \end{proof}
    We also need to estimate the derivative of $\eta_T(x+t\gamma(\theta))$ with respect to $t$. Without loss of generality we may write
    \begin{align*}
        \eta_T(x)
        &=a_1(2^{j-k-k\beta}(x-c_T)\cdot \mathbf e_1(\theta_P))a_2(2^{j-k/2-k\beta}(x-c_T)\cdot \mathbf e_2(\theta_P))\\
        &\cdot a_3(2^{j-k\beta}(x-c_T)\cdot \mathbf e_3(\theta_P)),
    \end{align*}      
    for some suitable bump functions $a_1,a_2,a_3$ supported on $(-1,1)$. Thus for $x\in B^3(c_T,2^{k-j+k\beta})$, we have
    \begin{align*}
        &\left|\frac {d}{dt}\eta_T(x+t\gamma(\theta))\right|\\
        &\lesssim  2^{j-k-k\beta }|\mathbf e_1(\theta_P)\cdot \mathbf e_1(\theta)|+ 2^{j-k/2-k\beta }|\mathbf e_2(\theta_P)\cdot \mathbf e_1(\theta)|\\
        &+2^{j-k\beta }|\mathbf e_3(\theta_P)\cdot \mathbf e_1(\theta)|\\
        &\lesssim 2^{j-k-k\beta}+ 2^{j-k/2-k\beta}|\theta-\theta_P|+2^{j-k\beta} |\theta-\theta_P|^2\\
        &\lesssim 2^{j-k\beta} |\theta-\theta_P|^2,
    \end{align*}
    and for $x\notin B^3(c_T,2^{k-j+k\beta})$ we have rapid decay in $x$ since $\eta$ is Schwartz.
    
    With Claim \ref{claim_exponent_rapid_decay}, we can apply integration by parts repeatedly to estimate the left hand side of \eqref{eqn_rapid_decay_lem_5}. Each time we integrate by parts, we gain a factor $\lesssim 2^{-k\beta}$, from which the result follows.

\end{proof}

\subsection{Good part and bad part}
The main idea is to divide the wave packets into two parts, called the good part and the bad part. We will prove an $L^1$ estimate for the bad part and an $L^2$ estimate for the good part.

Let $a>2$ be as in Theorem \ref{thm_thm8}. Let $\eps,a_1$ be small numbers to be determined later.

For each $P\in \Lambda$, define the set of heavy planks in $\mathbb{T}_{P}$
\begin{equation}
    \mathbb T_{P,b}:=\left\{T\in \mathbb T_P: \mu(T)\ge 2^{-k\frac{a_1+1}2-a(j-k)}\right\},\quad \mathbb T_{P,g}:=\mathbb T_{P}\backslash \mathbb T_{P,b}.
\end{equation}
Decompose $\mu=\mu_b+\mu_g:=\mu_b+\mu_0+\mu_{g1}+\mu_{g_2}$, where
\begin{align}\label{eqn_defn_good_bad}
    \mu_0&:=\mu*\psi_0^\vee,\\
    \mu_b
    &:=\sum_{j\ge 0}\sum_{k\in [j\eps,j]}\sum_{P\in \Lambda_{j,k}}\sum_{T\in \mathbb T_{P,b}}M_T\mu,\\
    \mu_{g_1}&:=\sum_{j\ge 0}\sum_{k\in [j\eps,j]}\sum_{P\in \Lambda_{j,k}}\sum_{T\in \mathbb T_{P,g}}M_T\mu,\\    
    \mu_{g_2}&:=\sum_{j\ge 0}\sum_{k\in [0,j\eps)}\sum_{P\in \Lambda_{j,k}}\sum_{T\in \mathbb T_{P}}M_T\mu.
\end{align}
We remark that the wave packets of $\mu_b$ are those that have heavy mass and are not too far away from the cone $\Gamma$ (as $k$ is not too small). 

Note that each $M_T\mu$ is a function but $\mu_{g},\mu_b$ are in general measures. We also remark that $\mu_g=\mu_{g,\beta,a,a_1,\eps}$ and $\mu_b=\mu_{b,a,\beta,a,a_1,\eps}$ depend on parameters $\beta,a,a_1,\eps$, but for simplicity we just omit them.

We now state the main estimates we want to prove. 
\begin{lem}[Lemma 6 of \cite{GGGHMW2022}]\label{lem_lem6}
    Let $a>2$, $\eps\ll 1$ and $a_1<2$. For all Borel measures $\mu$ supported on the unit ball in $\R^3$ with $\rho_a(\mu)\le 1$, it holds that
    \begin{equation}
        \sum_{j\ge 0}\sum_{k\in [j\eps,j]}\sum_{P\in \Lambda_{j,k}}\sum_{T\in \mathbb T_{P,b}}\int \int |\pi_{\theta \#}M_T\mu |d \mathcal H^2 d\theta\lesssim 1.
    \end{equation}
\end{lem}

Let $\{\phi_{n}:n\ge 1\}$ be an approximation to the identity in $\R^3$, and then define $\mu^{(n)}_{g}:=\mu_g*\phi_{n}$ to be a mollified version of $\mu_g$, such that each $\mu^{(n)}_{g}$ is a function.
\begin{lem}[Lemma 7 of \cite{GGGHMW2022}]\label{lem_lem7}
    Let $a>2$. Then for $\eps>0$ and $2-a_1>0$ small enough depending on $a$, all Borel measures $\mu$ supported on the unit ball in $\R^3$ with $\rho_a(\mu)\le 1$, it holds that
    \begin{equation}
        \int \int |\pi_{\theta \#}\mu^{(n)}_{g}|^2 d \mathcal H^2 d\theta\lesssim 1,
    \end{equation}
    where the constant is independent of $n$.
\end{lem}

We now prove Theorem \ref{thm_thm8} assuming Lemmas \ref{lem_lem6} and \ref{lem_lem7}.
\begin{proof}[Proof of Theorem \ref{thm_thm8}]
    Take $\eps,2-a_1$ sufficiently small such that both Lemmas \ref{lem_lem6} and \ref{lem_lem7} hold. 

    By Lemma \ref{lem_lem6}, in particular, for a.e. $\theta$ we have
    \begin{equation}
        \sum_{j\ge 0}\sum_{k\in [j\eps,j]}\sum_{P\in \Lambda_{j,k}}\sum_{T\in \mathbb T_{P,b}}\int|\pi_{\theta \#}M_T\mu|d\mathcal H^2<\infty,
    \end{equation}
    which implies that $\pi_{\theta \#}\mu_b$ is absolutely continuous with respect to $\mathcal H^2$.
    
    By Lemma \ref{lem_lem7}, the sequence of functions $\{\pi_{\theta \#}\mu^{(n)}_{g}\}$ is bounded in $L^2(\mathcal H^2\times I_0)$, and so by weak compactness, there exists an $L^2(\mathcal H^2\times I_0)$ function $h$ such that $\pi_{\theta \#}\mu^{(n)}_{g}\to h$ weakly in $L^2(\mathcal H^2\times I_0)$ passing to a subsequence. However, by properties of approximate identity we know that $\mu^{(n)}_{g}\to \mu_g$ weakly in $L^2(\mathcal H^2)$, whence $\pi_{\theta \#}\mu^{(n)}_{g}\to \pi_{\theta \#}\mu_g$ weakly in $L^2(\mathcal H^2\times I_0)$. By uniqueness of limits, we have $h=\pi_{\theta \#}\mu_g$, and so $\pi_{\theta \#}\mu_g$ is actually a function in $L^2(\mathcal H^2\times I_0)$. Since $\mathcal H^2|_{B^2(0,1)}\times I_0$ is a finite measure space, H\"older's inequality implies that 
    \begin{equation}
        \int \int |\pi_{\theta \#}\mu_g| d \mathcal H^2 d\theta\lesssim 1.
    \end{equation}
    In particular, for a.e. $\theta\in I_0$ we have 
    \begin{equation}
        \int |\pi_{\theta \#}\mu| d \mathcal H^2<\infty,
    \end{equation}
    and so $\pi_{\theta \#}\mu_g$ is absolutely continuous with respect to $\mathcal H^2$. Hence $\pi_{\theta \#}\mu$ is absolutely continuous with respect to $\mathcal H^2$.
\end{proof}

\subsection{Controlling bad part}
We now prove Lemma \ref{lem_lem6}. The main ingredient is Corollary \ref{cor_cor_3}. We have
\begin{align}
    &\int \sum_{j\ge 0}\sum_{k\in [j\eps,j]}\sum_{P\in \Lambda_{j,k}}\sum_{T\in \mathbb T_{P,b}}\int |\pi_{\theta\#} M_T\mu|d\mathcal H^2d\theta\nonumber\\
    &=\sum_{j\ge 0}\sum_{k\in [j\eps,j]} \int\int \sum_{ \stackrel{P\in \Lambda_{j,k}}{|\theta_P-\theta|<C2^{-k/2}} }\sum_{T\in \mathbb T_{P,b}}|\pi_{\theta\#} M_T\mu| d\mathcal H^2d\theta\label{eqn_proof_lem6_01}\\
    &+\sum_{j\ge 0}\sum_{k\in [j\eps,j]} \int\int \sum_{ \stackrel{P\in \Lambda_{j,k}}{|\theta_P-\theta|\ge C2^{-k/2}} }\sum_{T\in \mathbb T_{P,b}}|\pi_{\theta\#} M_T\mu| d\mathcal H^2d\theta\label{eqn_proof_lem6_02}.
\end{align}
\subsubsection{Far part}
To bound \eqref{eqn_proof_lem6_02}, we use Lemma \ref{lem_almost_orthogonality} to get
\begin{align*}
    &\sum_{ \stackrel{P\in \Lambda_{j,k}}{|\theta_P-\theta|\ge C2^{-k/2}} }\sum_{T\in \mathbb T_{P}}|\pi_{\theta\#} M_T\mu(x)|\\
    &\lesssim_{N,\beta} (\#\Lambda_{j,k})2^{-kN+3j}\mu(\R^3)\\
    &\lesssim_{N,\beta} 2^{-kN+3j},
\end{align*}
where we have used $\#\Lambda_{j,k}\lesssim 2^{k/2}$ and $\mu(\R^3)\lesssim 1$ since $\rho_a(\mu)\leq 1$.

Thus we have
\begin{align*}
    \eqref{eqn_proof_lem6_02}
    &\lesssim_{N,\beta}\sum_{j\ge 0}\sum_{k\in [j\eps,j]} \int \int 2^{-kN+3j}d\mathcal H^2d\theta\\
    &\lesssim \sum_{j\ge 0}2^{3j-3j\eps N}\\
    &\lesssim 1,
\end{align*}
if we choose $N>10\eps^{-1}$. 
\subsubsection{Near part}\label{sec_near_part}
To estimate \eqref{eqn_proof_lem6_01}, we apply Lemma \ref{lem_rapid_decay_plank} and get
\begin{equation}\label{eqn_Sep_22_02}
    \sum_{ \stackrel{P\in \Lambda_{j,k}}{|\theta_P-\theta|\le C2^{-k/2}} }\sum_{T\in \mathbb T_{P,b}}\int|\pi_{\theta\#} M_T\mu| d\mathcal H^2\lesssim_{\beta,N} \sum_{ \stackrel{P\in \Lambda_{j,k}}{|\theta_P-\theta|\le C2^{-k/2}} }\sum_{T\in \mathbb T_{P,b}}2^{-kN}+ \mu(T).
\end{equation}
The rapidly decaying term $2^{-kN}$ can be dealt with easily, since we also have
\begin{equation}
    \# \mathbb T_{P,b}\leq 2^{k\frac {a_1+1} 2+a(j-k)}
\end{equation}
by the heavy mass in our definition of $\mathbb T_{P,b}$. The remaining argument is divided into the following steps.

\underline{Step 1: Discretization.}

To estimate the other term, we define
\begin{equation}\label{eqn_Sep_22_03}
    F(\theta,C):=\sum_{ \stackrel{P\in \Lambda_{j,k}}{|\theta_P-\theta|\le C2^{-k/2}} }\sum_{T\in \mathbb T_{P,b}}\mu(T)=\sum_{ \stackrel{P\in \Lambda_{j,k}}{|\theta_P-\theta|\le C2^{-k/2}} }\mu(\cup\mathbb T_{P,b}).
\end{equation}
It is easy to see that for $|\theta-\theta'|\le C2^{-k/2}$, we have
\begin{equation}
    F(\theta,C)\le F(\theta',2C).
\end{equation}
Using this and denoting $\Theta'_k$ to be the $C2^{-k}$ lattice on $[0,1]$, we have
\begin{align}
    \int F(\theta,C)d \theta
    &=\sum_{\theta'\in \Theta'_k}\int_{\theta'-C2^{-k-1}}^{\theta'+C2^{-k-1}}F(\theta,C)d\theta\nonumber\\
    &\leq \sum_{\theta'\in \Theta'_k}\int_{\theta'-C2^{-k-1}}^{\theta'+C2^{-k-1}}F(\theta',2C)d\theta\nonumber\\
    &\sim\sum_{\theta'\in \Theta'_k} 2^{-k}F(\theta',2C)\nonumber\\
    &\sim 2^{-k} \sum_{P\in \Lambda_{j,k}}\mu(\cup\mathbb T_{P,b}),\label{eqn_Sep_22_04}
\end{align}
where in the last line we have used that $\Theta_k$ is a $c2^{-k/2}$ lattice.

\underline{Step 2: Rescaling.}

Now we seek to apply Corollary \ref{cor_cor_3} with $\delta^{1/2}\sim 2^{-k/2}$, where we take $\Delta=\Theta_k$. 

Let $B_l$ be finitely overlapping balls of radius $2^{k-j+k\beta}$ that cover $B^3(0,1)$. For each $\theta$ and $l$ we have
\begin{equation}\label{eqn_Sep_22_01}
    \mu(\cup \mathbb T_{P,b})\le \sum_{B_l}\mu(\cup \mathbb T_{P,l,b}),
\end{equation}
where
\begin{equation}
    \mathbb T_{P,l,b}:=\{T\in\mathbb T_{P,b}:T\cap B_l\ne \varnothing\}.
\end{equation}
It is easy to see that $\cup \mathbb T_{P,l,b}\sub 2^{1000}B_l$ where $2^{1000}B_l$ denotes the concentric dilation of $B_l$; also $2^{1000}B_l$'s have finite overlap.

Now we perform a rescaling. Let $\lambda(x):=2^{-j+k-k\beta}x$ and $\mu_{j,k}$ be the measure on $\R^3$ defined by
\begin{equation}
    \mu_{j,k}(A):=\mu(\lambda (A)).
\end{equation}
Also, for each $l$, define $\tilde B_l=\lambda(2^{1000}B_l)$ and
\begin{equation}
    \mu_{j,k,l}:=2^{a(j-k)}\mu_{j,k}1_{\tilde B_l},\quad \tilde {\mathbb T}_{P,l,b}:=\{\lambda(T):T\in \mathbb T_{P,l,b}\}.
\end{equation}
We can check that the set $\cup \tilde{\mathbb T}_{P,l,b}$ is a disjoint union of planks of dimensions $2^{-k}\times 2^{-k/2}\times 1$ pointing in the directions $\mathbf e_i(\theta_P)$, $i=1,2,3$ respectively, and it is contained in $\tilde B_l$. If we project every plank along the direction $\mathbf e_3(\theta_P)$, then we obtain a disjoint family of planar rectangles of dimensions $2^{-k}\times 2^{-k/2}$ pointing in the directions $\mathbf e_i(\theta_P)$, $i=1,2$ respectively.

Meanwhile, in view of the heavy mass as in the definition $\mathbb T_{P,b}$, we have $\mu_{j,k,l}(\lambda (T))\gtrsim 2^{-k\frac {a_1+1}2}$, and so we have 
\begin{equation}
    \# (\tilde{\mathbb T}_{P,l,b}) \lesssim 2^{k\frac {a+1}2}\mu_{j,k,l}(\R^3).
\end{equation}
It is also direct to see that $\rho_{a}(\mu_{j,k,l})\lesssim 2^{ka\beta}$. 

\underline{Step 3: Applying Corollary \ref{cor_cor_3}.}

Now we use Corollary \ref{cor_cor_3} with $\delta^{1/2}=c2^{-k/2}$ to get some small exponent $\eps'>0$, depending only on $a,a_1$, such that
\begin{equation}
    \sum_{P\in \Lambda_{j,k}}\pi_{\theta_P \#}\mu_{j,k,l}\left(\bigcup\pi_{\theta_P \#} \tilde{\mathbb T}_{P,l,b}\right)\lesssim 2^{-\frac k 2-k\eps'+ka\beta}\mu_{j,k,l}(\R^3).
\end{equation}
But by direct computation,
\begin{equation}
    \pi_{\theta_P \#}\mu_{j,k,l}\left(\bigcup\pi_{\theta_P \#} \tilde{\mathbb T}_{P,l,b}\right)=\mu_{j,k,l}\left(\bigcup \tilde{\mathbb T}_{P,l,b}\right)\gtrsim 2^{a(j-k)}\mu(\cup \mathbb T_{P,l,b}).
\end{equation}
Summing over $B_l$ and using \eqref{eqn_Sep_22_01}, we thus have
\begin{equation}
    \sum_{P\in \Lambda_{j,k}}\mu(\cup \mathbb T_{P,b})\le \sum_{B_l}\mu(\cup \mathbb T_{P,l,b})\lesssim 2^{-\frac k 2-k\eps'}\mu(\R^3)\le 2^{-\frac k 2-k\eps'+ak\beta}.
\end{equation}
Continuing estimating from \eqref{eqn_Sep_22_02} and using \eqref{eqn_Sep_22_03} and \eqref{eqn_Sep_22_04}, we have
\begin{equation}
    \sum_{ \stackrel{P\in \Lambda_{j,k}}{|\theta_P-\theta|\le C2^{-k/2}} }\sum_{T\in \mathbb T_{P,b}} \mu(T)
    \lesssim 2^{ak\beta-k\eps'}.
\end{equation}
Now it suffices to choose $\beta=\eps'/100$ and sum over $k\in [j\eps,j]$ and then $j\in \N$. This concludes the $L^1$ estimate of the bad part.

\subsection{Controlling good part}
We now prove Lemma \ref{lem_lem7}. By the functional analytic proof of Theorem \ref{thm_thm8}, we may abuse notation and treat $\mu_g$ as a function as well. 

Fix $\theta$ and $x\in \pi_{\theta\#}(\R^3)$. We have
\begin{equation}
    \pi_{\theta\#}\mu_g(x)=\int \mu_g (x+t\gamma(\theta))dt.
\end{equation}
With the coordinate system $\{\mathbf e_i(\theta):i=1,2,3\}$ and slightly abusing notation, any $x\in \pi_{\theta\#}(\R^3)$ can be written as $(0,x_2,x_3)$. Thus the above equality can be rewritten as
\begin{equation}
    \pi_{\theta\#}\mu_g(x_2,x_3)=\int\mu_g(t,x_2,x_3)dt.
\end{equation}
Doing the Fourier transform in the $(\mathbf e_2,\mathbf e_3)$-plane, we have
\begin{equation}
    (\pi_{\theta\#}\mu_g)^\wedge(\eta_2,\eta_3)=\widehat {\mu_g}(0,\eta_2,\eta_3)=\widehat {\mu_g}(\eta_2 \mathbf e_2(\theta)+\eta_3 \mathbf e_3(\theta)).
\end{equation}
By Plancherel's theorem,
\begin{equation}\label{eqn_good2_after_Plancherel}
    \int\int_{\R^2} |\pi_{\theta\#}\mu_g|^2 d\mathcal H^2 d\theta=\int\iint|\widehat {\mu_g}(\eta_2 \mathbf e_2(\theta)+\eta_3 \mathbf e_3(\theta))|^2 d\eta d\theta.
\end{equation}
By the triangle inequality, we will estimate $\mu_0$, $\mu_{g1}$ and $\mu_{g2}$ respectively.

We first settle the easiest case, namely, for $\mu_0$. We have $\widehat{\mu_0}=\widehat \mu \psi_0$. Since $\sup|\widehat \mu|\leq \mu(\R^3)<\infty$, the integral over $|\eta|\lesssim 1$ and $\theta\in I_0$ is clearly finite.

\subsubsection{Preparation}
To estimate the other two terms, we first perform a change of variables
\begin{equation}
    \xi=(\xi_1,\xi_2,\xi_3)=\eta_2 \mathbf e_2(\theta)+\eta_3 \mathbf e_3(\theta), 
\end{equation}
which has Jacobian 
\begin{equation}
    |\det (\mathbf e_2,\mathbf e_3,\eta_2 (-\mathbf e_1+\uptau \mathbf e_3)-\eta_3 \uptau \mathbf e_2)|=|\eta_2\det (\mathbf e_1,\mathbf e_2,\mathbf e_3)|= |\eta_2|.
\end{equation}
Thus we have
\begin{equation}\label{eqn_lem7_change_variables}
    \int \iint |\widehat {\mu_g}(\eta_2 \mathbf e_2(\theta)+\eta_3 \mathbf e_3(\theta) )|^2 d\eta d\theta= \int_{\R^3} |\widehat {\mu_g}(\xi)|^2 |\eta_2|^{-1}d\xi.
\end{equation}
We also need the following lemma.
\begin{lem}\label{lem_claim1}
    Let $P\in \Lambda_{j,k}$. If there exist $\theta\in I_0$ and $(\eta_2,\eta_3)$ satisfying
    \begin{equation}
        \eta_2 \mathbf e_2(\theta)+\eta_3 \mathbf e_3(\theta)\in \supp (\psi_P),
    \end{equation}
    as well as $|\theta-\theta_P|\ll 1$, then it holds that 
    \begin{equation}
        |\theta-\theta_P|\lesssim 2^{-\frac{k}2},\quad |\eta_2|\lesssim 2^{j-\frac k 2}, \quad |\eta_3|\sim 2^{j}.
    \end{equation}
    Furthermore, if $j>k$, then it holds that 
    \begin{equation}
        |\theta-\theta_P|\sim 2^{-\frac{k}2},\quad |\eta_2|\sim 2^{j-\frac k 2}, \quad |\eta_3|\sim 2^{j}.
    \end{equation}
\end{lem}
\begin{proof}
    Write $\xi=\eta_2 \mathbf e_2(\theta)+\eta_3 \mathbf e_3(\theta)$. Since $\xi\in \supp (\psi_P)=2P$, we have (without loss of generality)
    \begin{align}
        &|\xi\cdot \mathbf e_1(\theta_P)|\lesssim 2^{j-k},\label{item_01_claim1}\\
        &|\xi\cdot \mathbf e_2(\theta_P)|\ll 2^{j-\frac k 2},\label{item_02_claim1}\\
        &\xi\cdot \mathbf e_3(\theta_P)\sim 2^{j}\label{item_03_claim1}.
    \end{align}
    In particular, $|\xi|\sim |\eta_2|+|\eta_3|\sim 2^{j}$. Assume towards contradiction that $|\theta-\theta_P|\gg 2^{-k/2}$. Using \eqref{item_02_claim1}, by Taylor expansion for $\mathbf e_2(\theta_P)$ at $\theta$ we have
    \begin{equation}\label{eqn_Sep_12}
        |\eta_2+\eta_3(\theta_P-\theta)+O(\theta_P-\theta)^2|\ll 2^{j-\frac k 2}.
    \end{equation}
     Using \eqref{item_03_claim1}, by Taylor expansion for $\mathbf e_1(\theta_P)$ at $\theta$ we have
     \begin{align*}
        2^{j-k}
        &\gtrsim |\xi\cdot [\mathbf e_1(\theta)+\mathbf e_2(\theta)(\theta_P-\theta)+(-\mathbf e_1(\theta)+\uptau(\theta)\mathbf e_3(\theta))(\theta_P-\theta)^2\\
        &+|\xi|o((\theta_P-\theta)^2)]|\\
        &=\eta_2 (\theta_P-\theta)+\eta_3(\theta_P-\theta)^2(1/2+o(1)).
    \end{align*}
    Using $|\theta-\theta_P|\gg 2^{-k/2}$, we have
    \begin{equation}
        |\eta_2 +\eta_3(\theta_P-\theta)(1/2+o(1))|\ll 2^{j-\frac k 2}.
    \end{equation}    
    Combining this with \eqref{eqn_Sep_12} and using the triangle inequality, we have
    \begin{equation}
        |\eta_2|\ll 2^{j-k/2},
    \end{equation}
    and plugging this into \eqref{eqn_Sep_12} again, we get
    \begin{equation}
        \eta_3\lesssim |\theta_P-\theta|^{-1}2^{j-k/2}\ll 2^j,
    \end{equation}
    using $|\theta-\theta_P|\ll 1$. This leads to a contradiction since $|\eta_2|+|\eta_3|\sim 2^{j}$. Thus we must have $|\theta-\theta_P|\lesssim 2^{-k/2}$.

    By Taylor approximation, it is easy to get $|\eta_2|\lesssim 2^{j-k/2}$ and $\eta_3\sim 2^{j}$.

    Lastly, we come to the case $j>k$. Assume towards contradiction that $|\theta-\theta_P|\ll 2^{-k/2}$. Then by Taylor expansion,
    \begin{equation}
        |\xi\cdot \mathbf e_1(\theta)-\xi\cdot \mathbf e_1(\theta_P)|\le |\xi\cdot \mathbf e_2(\theta_P)||\theta-\theta_P|\ll 2^{j-k},
    \end{equation}
    and so $|\xi\cdot \mathbf e_1(\theta)|\sim 2^{j-k}$, contradicting the assumption that $\xi\cdot \mathbf e_1(\theta)=0$. Thus $|\theta-\theta_P|\gtrsim 2^{-k/2}$. By Taylor approximation, it is easy to get $|\eta_2|\sim 2^{j-k/2}$ and $\eta_3\sim 2^{j}$.

\end{proof}

\subsubsection{Estimate of \texorpdfstring{$\mu_{g2}$}{Lg}}

For $\mu_{g2}$ we have $k\le j\eps$, namely, the planks $P$ are extremely far away from the cone $\Gamma$. We estimate \eqref{eqn_good2_after_Plancherel} with $\mu_{g2}$ in place of $\mu_g$. Since $k\ge j\eps$, Lemma \ref{lem_claim1} implies that
\begin{equation}
    |\eta_2|\gtrsim |\xi|^{1-\eps/2}.
\end{equation}
Thus, by \eqref{eqn_lem7_change_variables}, we have
\begin{equation}
    \int_{\R^3} |\widehat {\mu_g}(\xi)|^2 |\eta_2|^{-1}d\xi\lesssim \int_{\R^3} |\widehat {\mu}(\xi)|^2 |\xi|^{\eps/2-1}d\xi,
\end{equation}
which is finite if we choose $\eps$ small enough that $a-\eps/2>2$, since $\rho_a(\mu)\lesssim 1$ and $a>2$ (see, for instance, Theorems 2.7 and 3.10 of \cite{Mattilabook}).

\subsubsection{Estimate of \texorpdfstring{$\mu_{g1}$}{Lg}}
This part estimates the contribution from the planks that have light mass and are not too far away from the cone. By definition, we have
\begin{equation}
    \widehat {\mu_g}(\xi)=\sum_{j\ge 0}\sum_{k\in [j\eps,k]}\sum_{P\in \Lambda_{j,k}}\sum_{T\in \mathbb T_{P,g}}\widehat {M_T\mu}
\end{equation}
where we have
\begin{equation}
    \widehat {M_T\mu}(\xi)= \widehat{\eta_T} *(\widehat \mu(\xi) \psi_P(\xi)),
\end{equation}
which is supported on $3P$ where $P\in \Lambda_{j,k}$. By Plancherel's identity, we have
\begin{equation}
    \int_{\R^3} |\widehat {\mu_g}(\xi)|^2 |\eta_2|^{-1}d\xi=\sum_{j\ge 0}\sum_{k\in [j\eps,k]}\sum_{P\in \Lambda_{j,k}}\int_{\R^3} \bigg|\sum_{T\in \mathbb T_{P,g}} \widehat {M_T\mu}(\xi)\bigg|^2 |\eta_2|^{-1} d\xi.
\end{equation}
We further divide the argument into the following steps.

\underline{Step 1: Using orthorgonality.}

Fix $k\in [j\eps,k]$. We consider the cases $k<j$ and $k=j$ separately. In the former case, using Lemma \ref{lem_claim1}, we have
\begin{align*}
    \int_{\R^3} \bigg|\sum_{T\in \mathbb T_{P,g}} \widehat {M_T\mu}(\xi)\bigg|^2 |\eta_2|^{-1} d\xi
    &\sim 2^{-j+k/2}\int_{\R^3} \bigg|\sum_{T\in \mathbb T_{P,g}} \widehat {M_T\mu}(\xi)\bigg|^2  d\xi\\
    &=2^{-j+k/2}\int_{\R^3} \bigg|\sum_{T\in \mathbb T_{P,g}} M_T\mu(x)\bigg|^2  dx\\
    &\sim  2^{-j+k/2}\sum_{T\in \mathbb T_{P,g}}\int_{\R^3} | M_T\mu(x)|^2  dx.
\end{align*}
Summing over $P\in \Lambda_{j,k}$, we thus have
\begin{equation}\label{eqn_113}
    \sum_{P\in \Lambda_{j,k}}\int_{\R^3} \bigg|\sum_{T\in \mathbb T_{P,g}} \widehat {M_T\mu}(\xi)\bigg|^2 |\eta_2|^{-1} d\xi\lesssim 2^{-j+\frac k 2}\sum_{P\in \Lambda_{j,k}}\sum_{T\in \mathbb T_{P,g}}\int_{\R^3} | M_T\mu(x)|^2  dx.
\end{equation}

For $j=k$ we have \eqref{eqn_113} holds as well. To see this, fix $P$ and we go back to the form
\begin{equation}
    \int \iint \bigg|\sum_{T\in \mathbb T_{P,g}}\widehat {M_T\mu}(\eta_2 \mathbf e_2(\theta)+\eta_3 \mathbf e_3(\theta) )\bigg|^2 d\eta d\theta.
\end{equation}
By Lemma \ref{lem_claim1}, we must have $|\theta-\theta_P|\lesssim 2^{-j/2}$. For every such $\theta$, we apply Plancherel's identity in the $(\eta_2,\eta_3)$ plane to get
\begin{equation}\label{eqn_Sep_23}
    \iint \bigg|\sum_{T\in \mathbb T_{P,g}}\widehat {M_T\mu}(\eta_2 \mathbf e_2(\theta)+\eta_3 \mathbf e_3(\theta) )\bigg|^2 d\eta
    =\iint |F(x_2,x_3)|^2 dx_2dx_3,
\end{equation}
where
\begin{align*}
    F(x_2,x_3)
    &:=\sum_{T\in \mathbb T_{P,g}}\iint\widehat {M_T\mu}(\eta_2 \mathbf e_2(\theta)+\eta_3 \mathbf e_3(\theta) )e(x_2 \eta_2+x_3 \eta_3)d\eta_2 d\eta_3\\
    &=\sum_{T\in \mathbb T_{P,g}}\int_{\R^3}M_T\mu(y)\iint e(-y\cdot(\eta_2 \mathbf e_2(\theta)+\eta_3 \mathbf e_3(\theta)) )dy\\
    &\cdot e(x_2 \eta_2+x_3 \eta_3)d\eta_2 d\eta_3\\
    &=\sum_{T\in \mathbb T_{P,g}}\int_{\R^3}M_T\mu(y)\delta_0(x_2-y\cdot \mathbf e_2(\theta))\delta_0(x_3-y\cdot \mathbf e_3(\theta))dy\\
    &=\sum_{T\in \mathbb T_{P,g}}\int_{\R}M_T\mu(t\mathbf e_1(\theta)+x_2\mathbf e_2(\theta)+x_3\mathbf e_2(\theta) )dt.
\end{align*}
Plugging back into \eqref{eqn_Sep_23}, we get
\begin{align*}
    &\iint |F(x_2,x_3)|^2 dx_2dx_3\\
    =&\iint F(x_2,x_3)\overline{F(x_2,x_3)} dx_2dx_3\\
    =&\sum_{T\in \mathbb T_{P,g}}\sum_{T'\in \mathbb T_{P,g}}\iint\int_{\R}\int_{\R}M_T\mu(t\mathbf e_1(\theta)+x_2\mathbf e_2(\theta)+x_3\mathbf e_3(\theta))\\
    &\cdot \overline{M_{T'}\mu(t'\mathbf e_1(\theta)+x_2\mathbf e_2(\theta)+x_3\mathbf e_3(\theta))} dt dt' dx_2 dx_3.
\end{align*}
For a fixed $T$, recall that $M_T\mu$ is essentially supported on $T$. Since $|\theta-\theta_P|\lesssim 2^{-j/2}$, we can show that 
\begin{equation}
    |M_T\mu(t\mathbf e_1(\theta)+x_2\mathbf e_2(\theta)+x_3\mathbf e_3(\theta))| \lesssim |\tilde M_T\mu(t\mathbf e_1(\theta_P)+x_2\mathbf e_2(\theta_P)+x_3\mathbf e_3(\theta_P))|
\end{equation}
where $\tilde M_T:=w_T (\mu*\psi_P^\vee)$ and $w_T$ is a suitable weight function adapted to $T$. For fixed $x_2,x_3$, we thus have $T,T'$ must only differ at most in the $\mathbf e_1(\theta_P)$ position (otherwise we have rapid decay which can be treated as $0$; we leave the rigorous argument to the reader; cf. Section \ref{sec_weight}). However, since we are in the case $k=j$ and we may assume $T\sub B^3(0,1)$ (otherwise we have rapid decay), we must have $T=T'$. Thus we can continue estimating
\begin{align*}
    &\iint |F(x_2,x_3)|^2 dx_2dx_3\\
    &\lesssim \sum_{T\in \mathbb T_{P,g}} \int_{\R}\int_{\R} \left|\int_{\R} M_T\mu(t\mathbf e_1(\theta_P)+x_2\mathbf e_2(\theta_P)+x_3\mathbf e_3(\theta_P))dt\right|^2 dx_2 dx_3 \\
    &\lesssim \sum_{T\in \mathbb T_{P,g}}\int_{\R} \int_{\R} \int_{\R}\left|M_T\mu(t\mathbf e_1(\theta_P)+x_2\mathbf e_2(\theta_P)+x_3\mathbf e_3(\theta_P))\right|^2 dtdx_2 dx_3,
\end{align*}
where in the last line we have used the fact that $t$ is essentially supported on the unit interval. Integrating in $\theta$, using the discretization and summing over $P\in \Lambda_{j,k}$, we have \eqref{eqn_113}. 

We continue to estimate \eqref{eqn_113} and do not distinguish $j>k$ and $j=k$ anymore. Thus it suffices to estimate the term
\begin{equation}
    \sum_{P\in \Lambda_{j,k}}\sum_{T\in \mathbb T_{P,g}}\int_{\R^3} | M_T\mu(x)|^2  dx.
\end{equation}

\underline{Step 2: Preparation for refined Strichartz inequality.}

We hope to apply Theorem \ref{thm_Strichartz} stated in Section \ref{sec_Strichartz} below. Before that we need to do some preparation work. 

Write 
\begin{equation}\label{eqn_defn_f_T}
    f_T:=(\eta_T M_T\mu)*\psi_P^\vee,
\end{equation}
so that
\begin{equation}
    \sum_{P\in \Lambda_{j,k}}\sum_{T\in \mathbb T_{P,g}}\int_{\R^3} | M_T\mu(x)|^2  dx=\int \sum_{P\in \Lambda_{j,k}}\sum_{T\in \mathbb T_{P,g}}f_T d\mu. 
\end{equation}
Similar to the proof of Step 2 of Section \ref{sec_near_part}, we cover the unit ball by $O(1)$-overlapping small balls $B_l$ of radius $2^{k-j+k\beta}$ and let $\nu_l$ be the restriction of $\mu$ to $B_l$. By Cauchy-Schwarz, we have
\begin{equation}
    \int \sum_{P\in \Lambda_{j,k}}\sum_{T\in \mathbb T_{P,g}}f_T d\mu\leq \sum_{l}\mu(B_l)^{1/2}\left(\int \left|\sum_{P\in \Lambda_{j,k}}\sum_{T\in \mathbb T_{P,g}}f_T \right|^2 d\nu_l\right)^{1/2}.
\end{equation}
Let $\zeta_j$ be a non-negative bump function such that $\widehat {\zeta_j}(\xi)=1$ for $|\xi|\le 2^{j+k\beta+10}$. By the Fourier support condition of $f_T$, we have
\begin{equation}
    \int \left|\sum_{P\in \Lambda_{j,k}}\sum_{T\in \mathbb T_{P,g}}f_T \right|^2 d\nu_l=\int \left|\sum_{P\in \Lambda_{j,k}}\sum_{T\in \mathbb T_{P,g}}f_T \right|^2 d(\nu_l*\zeta_j).
\end{equation}
By dyadic pigeonholing (see Section \ref{sec_dyadic_pigeonholing} for the technique), we can find a subset
\begin{equation}
    \mathbb W_l\sub \bigcup_{P\in \Lambda_{j,k}}\{T\in \mathbb T_{P,g}:T\cap B_l\ne \varnothing\}
\end{equation}
such that $\norm {f_T}_2$ is a constant $N$ up to a factor of $2$ as $T$ varies over $\mathbb W_l$, and
\begin{equation}
    \int\left| \sum_{P\in \Lambda_{j,k}}\sum_{T\in \mathbb T_{P,g}}f_T \right|^2 d(\nu_l*\zeta_j)\lesssim j \int \left|\sum_{T\in \mathbb W_l}f_T\right|^2 d(\nu_l*\zeta_j).
\end{equation}
Now we take some $p\ge 2$ to be determined. By H\"older's inequality, we have
\begin{equation}
    \int \left|\sum_{T\in \mathbb W_l}f_T\right|^2 d(\nu_l*\zeta_j)\le \norm{\sum_{T\in \mathbb W_l}f_T}_p^2 \left(\int(\nu_l*\zeta_j)^{\frac p {p-2}}\right)^{1-\frac 2 p}.
\end{equation}
We can view $\sum_{T\in \mathbb W_l}f_T$ as essentially supported on $B^3(0,1)$ by the rapid decay outside. Thus, by dyadic pigeonholing again, there is some dyadic $M\ge 1$ and a disjoint union $Y$ of balls $Q$ of radius $2^{-j}$, such that each $Q\sub Y$ intersects $\sim M$ planks $T$ as they vary over $\mathbb W_l$, and that
\begin{equation}
    \norm{\sum_{T\in \mathbb W_l}f_T}_p^2 \left(\int(\nu_l*\zeta_j)^{\frac p {p-2}}\right)^{1-\frac 2 p}\lesssim j\norm{\sum_{T\in \mathbb W_l}f_T}_{L^p(Y)}^2 \left(\int_Y (\nu_l*\zeta_j)^{\frac p {p-2}}\right)^{1-\frac 2 p}.
\end{equation}
Combining all the above estimates, we have in particular
\begin{align}
    &\sum_{P\in \Lambda_{j,k}}\sum_{T\in \mathbb T_{P,g}}\int_{\R^3} | M_T\mu(x)|^2  dx\nonumber\\
    &\lesssim_\beta 2^{j\beta} \sum_{l}\mu(B_l)^{1/2}\norm{\sum_{T\in \mathbb W_l}f_T}_{L^p(Y)}\left(\int_Y (\nu_l*\zeta_j)^{\frac p {p-2}}\right)^{\frac 1 2-\frac 1 p}.\label{eqn_before_Strichartz}
\end{align}

\underline{Step 3: Applying refined Strichartz inequality.}

Now we apply Theorem \ref{thm_Strichartz} (to be proved later) to the rescaled functions $f_T(2^{-j}\cdot)$ to get
\begin{align}
    \norm{\sum_{T\in \mathbb W_l}f_T}_{L^p(Y)}
    &\lesssim_\beta 2^{j\beta+(3j-\frac {3k}2)(\frac 1 2-\frac 1 p)}\left(\frac{M}{\# \mathbb W_l}\right)^{\frac 1 2-\frac 1 p}\left(\sum_{T\in \mathbb W_l}\norm{f_T}_2^2\right)^{1/2}\nonumber\\
    &\sim 2^{j\beta+(3j-\frac {3k}2)(\frac 1 2-\frac 1 p)}\left(\frac{M}{\# \mathbb W_l}\right)^{\frac 1 2-\frac 1 p} (\# \mathbb W_l)^{\frac 1 2} N\nonumber\\
    &=2^{j\beta+(3j-\frac {3k}2)(\frac 1 2-\frac 1 p)}M^{\frac 1 2-\frac 1 p}(\# \mathbb W_l)^{\frac 1 2}N.\label{eqn_Sep_23_02}
\end{align}
We also need to estimate the term involving $\nu*\zeta_j$. Using the assumption that $\rho_a(\mu)\le 1$, we have $\rho_a(\nu)\le 1$, and so by the rapid decay of $\zeta_j$ outside $B^3(0,2^{-j-k\beta})$,
\begin{align*}
    \left|\int \zeta_j(x-y)d\nu_l(y)\right|
    &\lesssim \nu_l(B^3(x,2^{-j-k\beta}))2^{3(j+k\beta)}\lesssim 2^{3(j+k\beta)-a(j-k\beta)}.
\end{align*}
Hence by the light mass in the definition of $\mathbb T_{P,g}$, we have
\begin{align*}
    \int_Y (\nu_l*\zeta_j)^{\frac p {p-2}}
    &\lesssim \int_Y \nu_l*\zeta_j\cdot 2^{\frac{6(j+k\beta)-2a(j-k\beta)}{p-2}} \\
    &\sim 2^{\frac{6(j+k\beta)-2a(j-k\beta)}{p-2}} M^{-1} \sum_{T\in \mathbb W_l}\int_{T}\nu_l*\zeta_j\\
    &\lesssim 2^{\frac{6(j+k\beta)-2a(j-k\beta)}{p-2}} M^{-1} (\# \mathbb W_l) \cdot 2^{-k\frac {a_1+1}{2}-a(j-k)}.
\end{align*}
Using \eqref{eqn_before_Strichartz} and \eqref{eqn_Sep_23_02}, we have
\begin{align}
    &\sum_{P\in \Lambda_{j,k}}\sum_{T\in \mathbb T_{P,g}}\int_{\R^3} | M_T\mu(x)|^2  dx\label{eqn_Sep_23_04}\\
    &\lesssim_\beta 2^{j\beta}\sum_{l}\mu(B_l)^{1/2}2^{j\beta+(3j-\frac {3k}2)(\frac 1 2-\frac 1 p)}M^{\frac 1 2-\frac 1 p}(\# \mathbb W_l)^{\frac 1 2}N\nonumber\\
    &\cdot \left(2^{\frac{6(j+k\beta)-2a(j-k\beta)}{p-2}} M^{-1} (\# \mathbb W_l) \cdot 2^{-k\frac {a_1+1}{2}-a(j-k)}\right)^{\frac 1 2-\frac 1 p}\nonumber\\
    &\lesssim 2^{O(j\beta)}2^{j\frac{3-a}2}2^{k(\frac 1 2-\frac 1 p)(a-2-\frac {a_1}2)}\sum_{l}\mu(B_l)^{1/2}(\# \mathbb W_l)^{1/2}N\nonumber\\
    &\le 2^{O(j\beta)}2^{j\frac{3-a}2}2^{k(\frac 1 2-\frac 1 p)(a-2-\frac {a_1}2)} \mu(\R^3)^{\frac 1 2}\left(\sum_l (\# \mathbb W_l) N^2\right)^{\frac 1 2},\label{eqn_Sep_23_05}
\end{align}
where we have used Cauchy-Schwarz in the last inequality. Recalling that $N\sim \norm{f_T}_2$, we have 
\begin{align*}
    \sum_{l}  (\# \mathbb W_l)N^2
    &\sim \sum_{l}\sum_{T\in \mathbb W_l}\norm{f_T}^2_2\\
    &\lesssim \sum_{l}\sum_{T\in \mathbb W_l}\norm{M_T\mu}^2_2\\
    &\lesssim \sum_{P\in \Lambda_{j,k}}\sum_{T\in \mathbb T_{P,g}}\int_{\R^3} | M_T \mu|^2.
\end{align*}
Plugging this back into \eqref{eqn_Sep_23_05} and noting this coincides with the term \eqref{eqn_Sep_23_05} at the beginning, we obtain
\begin{equation}
  \sum_{P\in \Lambda_{j,k}}\sum_{T\in \mathbb T_{P,g}}\int_{\R^3} | M_T\mu(x)|^2  dx\lesssim_\beta 2^{O(j\beta)}2^{j(3-a)}2^{k(1-\frac 2 p)(a-2-\frac {a_1}2)} \mu(\R^3).
\end{equation}

\underline{Step 4: Conclusion.}

Using $\mu(\R^3)\lesssim 1$, \eqref{eqn_113}, \eqref{eqn_lem7_change_variables} and \eqref{eqn_good2_after_Plancherel}, we have
\begin{equation}
  \int\int_{\R^2} |\pi_{\theta\#}\mu_g|^2 d\mathcal H^2 d\theta\lesssim_\beta \sum_{j=0}^\infty \sum_{k\in [j\eps,j]}2^{O(j\beta)}  2^{j(3-a)}2^{k(1-\frac 2 p)(a-2-\frac {a_1}2)} 2^{-j+\frac k 2}.
\end{equation}
We may assume $a$ and $a_1$ are very close to $2$. Thus we have 
\begin{equation}
    (a-2-\frac {a_1}2)(1-\frac 2 p)+\frac 1 2=-\frac 1 2+\frac 2 p+\beta',
\end{equation}
where $\beta'>0$ can be chosen to be arbitrarily close to $0$. If we choose any $p>4$ and $\beta'\ll p-4$, then we have
\begin{align*}
    \sum_{k\in [j\eps,j]}\int\int_{\R^2} |\pi_{\theta\#}\mu_g|^2 d\mathcal H^2 d\theta
    &\lesssim_\beta 2^{O(j\beta)}2^{j(3-a)-j}2^{j\eps(1-\frac 2 p)(-\frac 1 2+\frac 2 p+\beta'))+\frac {j\eps}2}\\
    &\lesssim 2^{O(j\beta+j\eps)}2^{j(2-a)}.
\end{align*}
Choosing $\beta$ and $\eps$ small enough (depending on $a,a_1$ only), we have
\begin{equation}
    O(\beta)+O(\eps)+2-a <0,
\end{equation}
and so summing over $j$ gives the claim. This finishes the estimate of $\mu_{g1}$ and concludes the proof of Lemma \ref{lem_lem7} and hence Theorem \ref{thm_thm8}. 

\section{Refined Strichartz inequality}\label{sec_Strichartz}
    We begin our study of the \textit{refined Strichartz inequality}. This type of estimate was first introduced in the paper \cite{guth2018falconers} and the subsequent argument follows section 4 of \cite{guth2018falconers} closely. The reader is encouraged to compare the two papers to understand the general strategy behind proving this theorem. For this exposition we omit some of the computations and focus on the general proof strategy. The reader is referred to the original paper \cite{GGGHMW2022} for details.
    
    Recall that we had that for each $\theta$, $e_1(\theta)= \gamma(\theta), e_2(\theta) = \gamma'(\theta), e_3(\theta)=\gamma(\theta) \times \gamma'(\theta)$ forms an orthonormal basis for $\R^3$.
    
    For each $R \geq 1$, let $\Pi_{R} = \{j R^{-1/2}: j \in \Z\} \cap [0,a]$. For each $\theta \in \Pi_R$ we let
    \begin{equation*}
        \tau_{R}(\theta) = \{x_1e_1(\theta) +x_2e_2(\theta) + x_3 e_3(\theta): 1 \leq x_1 \leq 2, |x_2| \leq R^{-1/2}, |x_3| \leq R\}.    
    \end{equation*}
    For the rest of this section we fix $R$ and thus we may abbreviate $\tau_{R}(\theta)$ by $\tau(\theta)$.  Let $\mathcal{P}_{R^{-1}}= \{ \tau(\theta): \theta \in \Pi_R\}$. If $\tau = \tau( \theta) \in \mathcal{P}_{R^{-1}}$ for some fixed $\tau$ then we define $\theta_{\tau} = \theta$ for this choice of $\tau$. Let
    \begin{equation*}
        T^{\circ}_{\tau, 0} =\{x_1e_3(\theta_{\tau})+ x_2e_2(\theta_{\tau}) + x_3e_1(\theta_{\tau}): |x_1| \leq R, |x_2| \leq R^{1/2}, |x_3| \leq 1 \}.
    \end{equation*}
    Notice that $T^{\circ}_{\tau,0}$ are planks that are dual to $\tau(\theta_\tau)$. We denote by $\mathbb{T}^{\circ}_{\tau}$ the collection of translates of $T^{\circ}_{\tau, 0}$ that cover $B(0,R)$. For a fixed small constant $\delta>0$, we let $T_{\tau ,0} = R^{\delta}T^{\circ}_{\tau, 0}$, and $\mathbb{T}_{\tau} := \{ R^{\delta} T : T \in \mathbb{T}^{\circ}_{\tau} \}$. For a given plank in the covering $T \in \mathbb{T}_{\tau}$, we set $\tau(T) := \tau$.
    \begin{defn}
        Fix $T \in \mathbb{T}_{\tau}$. We say that a function $f_T: \mathbb{R}^3 \rightarrow \mathbb{C}$ is a $T-$ function if $\hat{f_T}$ is supported on $\tau(T)$ and
        \begin{equation*}
            \norm{f_T}_{L^{\infty}(B(0,R) \setminus T)} \lesssim_{\delta} R^{-1000}\norm{f_T}_{2}.
        \end{equation*}
    \end{defn}
    We are now ready to state our \textit{refined Strichartz inequality}.
    \begin{thm}\label{thm_Strichartz}
        Let $\gamma:[a,b] \rightarrow S^2 $ be a non-degenerate $C^2$ curve (i.e. with $\det(\gamma, \gamma', \gamma'')$ non vanishing. As a consequence let $B$ be such that 
        \begin{equation}\label{cond1}
            |\det(\gamma, \gamma', \gamma'')| \geq B^{-1}
        \end{equation}
        and
        \begin{equation}\label{cond2}
            \norm{\gamma}_{C^2([a,b])} \leq B.
        \end{equation}
        Let $R \geq 1$ and suppose that $f= \sum_{T \in \mathbb{W} }f_T$, where each $f_T$ is a $T-$ function and 
        \begin{equation*}
            \mathbb{W} \subseteq \bigcup_{\tau \in \mathcal{P}_{R^{-1}}} \mathbb{T}_{\tau}
        \end{equation*}
    Also assume that for all $T, T' \in \mathbb{W}$,
    \begin{equation*}
        \norm{f_T}_2 \sim \norm{f_{T'}}_2.
    \end{equation*}
    Let $Y$ be a disjoint union of unit balls in $B(0,R)$ each of which intersects $M$ sets of the form $2T$ where $T \in \mathbb{W}$. Then for $2 \leq p \leq 6$, we have
    \begin{equation*}
        \norm{f}_{L^p(Y)} \lesssim_{\epsilon, \delta} R^{\epsilon}\left(\frac{M}{|\mathbb{W}|} \right)^{1/2-1/p}\left(\sum_{T \in \mathbb{W}}\norm{f}_{p}^{2}\right)^{1/2}.
    \end{equation*}
    \end{thm}
    \begin{proof}
        Assume that $[a,b] =[-1,1]$. We fix $\epsilon \in (0,1/2), \delta_0 = \epsilon^{100}$ and $\delta \in (0, \delta_0)$. We choose $R$ to be sufficiently large (the specific choice does not matter for this presentation).

        We assume that the following theorem holds for $K^2$ cubes instead of unit cubes where $K=R^{\delta^2}$ and for all scales $\Tilde{R}= \frac{R}{K^2} = R^{1- 2 \delta^2}$ and for all curves $\gamma$ satisfying the two conditions (\ref{cond1}) and (\ref{cond2}), and for all $B \geq 1$. We shall show that this assumption implies the theorem. As a consequence of inducting on scales, this implies the full theorem, for an introduction to the concept of inducting on scales we refer the reader to \cite{inductiononscales}. 
        
        For each $\tau \in \mathcal{P}_{R^{-1/2}}$, let $\kappa = \kappa(\tau) \in \mathcal{P}_K^{-1}$ which minimizes $|\theta_{\tau}-\theta_{\kappa}|$. For each such $\kappa$, let
        \begin{align*}
            \square_{\kappa, 0}
            &=\bigg\{x_1 \frac{(\gamma \times \gamma')(\theta_{\kappa})}{|(\gamma \times \gamma')(\theta_{\kappa})|} + x_2 \frac{\gamma'(\theta_{\kappa})}{|\gamma'(\theta_{\kappa})|} +x_3\gamma(\theta_{\kappa}):\\ 
            & |x_1| \leq R^{1+\delta}, |x_2| \leq R^{1+\delta}/K, |x_3| \leq R^{1+\delta}/K^2\bigg\}.
        \end{align*}
        and
        \begin{align*}
            \mathcal{P}_{\kappa}
            &=\bigg\{\square = a \gamma(\theta_{\kappa}) + b \frac{\gamma'(\theta_{\kappa})}{|\gamma'(\theta_{\kappa})|} +\square_{\kappa, 0} : \\ 
            & a \in ((1/10)R^{1+\delta}K^{-2})\mathbb{Z}, b \in ((1/10)R^{1+\delta}K^{-1})\mathbb{Z}\bigg\}.
        \end{align*}
        We let $\mathbb{P}= \bigcup_{\kappa \in \mathcal{P}_{K^{-1}}} \mathbb{P}_{\kappa}$. Given any $\tau$ and corresponding $\kappa = \kappa(\tau)$, a routine computation shows that
        \begin{equation}\label{1}
            |\inp{\gamma \times \gamma'(\theta_{\tau})} {\gamma'(\theta_{\kappa})}| \leq B^{-7}K^{-1}
        \end{equation}
        and
        \begin{equation}\label{2}
            |\inp{(\gamma \times \gamma')(\theta_{\tau})}{\gamma(\theta_{\kappa})}| \leq B^{-7}K^{-2}.
        \end{equation}
    From here we can see that for each $T \in \mathbb{T}_{\tau}$ there are $ \sim 1$ sets $\square \in \mathbb{P}_{\kappa(\tau)}$ with $T \cap 10 \square \neq \phi$ and moreover $T \subseteq 100\square$ whenever $T \cap 10\square \neq \phi$ and moreover $T \subseteq 100 \square$ whenever $T \cap 10 \square \neq \phi$. For each such $T$, let $\square = \square(T) \in \mathbb{P}_{\kappa}$ be some choice such that $T \cap 10 \square \neq \phi$, and let $\mathbb{W}_{\square}$ be the set of $T$'s that associated to the set $\square$.

    For each $\kappa$ and each $\square \in \mathbb{P}_{\kappa}$, let $\{\mathcal{Q}_{\square}\}$ be a finitely overlapping cover of $100 \square$ by translates of the ellipsoid given by:
    \begin{equation*}
        \left\{x_1\gamma(\theta_{\kappa}) +x_2 \frac{\gamma'(\theta_{\kappa})}{\abs{\gamma'(\theta_{\kappa})}} +x_3 \frac{\gamma \times \gamma'(\theta_{\kappa})}{\abs{\gamma \times \gamma'(\theta_{\kappa})}} : \abs{x_1}^2 + (\abs{x_2}K^{-1})^2 + (\abs{x_3}^{1/2}K^{-1})^{2} \leq \Tilde{K}^2\right\}.
    \end{equation*}

    Now let $\{\eta_{\mathcal{Q}_{\square}}\}_{{Q_{\square} \in \mathcal{Q}_{\square}}}$ be a smooth partition of unity such that on $\square$, we have that 
    \begin{equation*}
        \sum_{{\eta}_{Q_{\square} \in \mathcal{Q}_{\square}}} \eta_{{Q}_{\square}}=1
    \end{equation*}
    with the property that each $\eta_{\mathcal{Q}_{\square}}$ satisfies the following two bounds:
    \begin{equation*}
        \norm{\eta_{Q_{\square}}} \lesssim 1, \quad \norm{\eta_{\mathcal{Q}_{\square}}}_{L^{\infty}(\R^3 \setminus T)} \lesssim R^{-1000}
    \end{equation*}
    and
    \begin{equation*}
        \abs{\eta_{\mathcal{Q}_{\square}}(x)} \lesssim \mathrm{dist}(x, \mathcal{Q}_{\square})^{-1000}, \quad \forall x \in \R^3.
    \end{equation*}
    we may also ensure that $\widehat{\eta_{Q_{\square}}}$ is supported in the following set:
    \begin{equation*}
        \left\{x_1\gamma(\theta_{\kappa}) +x_2 \frac{\gamma'(\theta_{\kappa})}{\abs{\gamma'(\theta_{\kappa})}} +x_3 \frac{\gamma \times \gamma'(\theta_{\kappa})}{\abs{\gamma \times \gamma'(\theta_{\kappa})}} : \abs{x_1} \leq \Tilde{K}, \abs{x_2} \leq \Tilde{K}K, \abs{x_3} \leq \Tilde{K}K^2 \right\}.
    \end{equation*}
    By dyadic pigeonholing we have
    \begin{equation*}
        \norm{f}_{L^{p}(Y)} \lesssim \log R\, \left\|\sum_{\square} \sum_{T \in \mathbb{W}_{\square}}  \eta_{Y_\square}f_{T}\right\|_{L^{p}(Y)} + R^{-1000} \left( \sum_{T \in \mathbb{W}} \norm{f_T}_{p}^2\right)^{1/2},
    \end{equation*}
    where for each $\square$ we have that $Y_{\square}$ is a union of over a smaller $\textit{subset}$ $Q_{\square}$ and each $Q_{\square} \subseteq Y_{\square}$ intersects a number of sets of the form $T$ (with $T \in \mathbb{W}_{\square}$ where the number of such intersections is contained in the set $ [M'(\square), 2M'(\square))$ are constant. $\square \in \mathbb{B}$ (upto a possible factor of 2).

    We dyadic pigeonhole again to yield that
    \begin{equation*}
        \left\|\sum_{\square} \sum_{T \in \mathbb{W}_{\square}} \eta_{Y_{\square}} f_{T}\right\|_{L^{p}(Y)} \lesssim (\log R)^{2} \left\|\sum_{\square \in \mathbb{B}} \sum_{T \in \mathbb{W}_{\square}} \eta_{Y_{\square}} f_{T}\right\|_{L^{p}(Y)},
    \end{equation*}
    where now $\abs{\mathbb{W}_{\square}}$ and $M'(\square)$ are constant over $\square \in \mathbb{B}$ (once again, upto a factor of 2). We do one more pigeonholing step
    \begin{equation*}
        \left\|\sum_{\square} \sum_{T \in \mathbb{W}_{\square}} \eta_{Y_{\square}} f_{T}\right\|_{L^{p}(Y)} \lesssim (\log R) \left\|\sum_{\square \in \mathbb{B}} \sum_{T \in \mathbb{W}_{\square}} \eta_{Y_{\square}} f_{T}\right\|_{L^{p}(Y')},
    \end{equation*}
    where $Y'$ is now a union over $K^2$ balls $Q \subseteq Y$ such that each ball $2Q$ intersects a number of sets of the form $Y_{\square}$, where the number of intersections is contained in the interval $[M'', 2M'')$. Let us fix such a $Q \subseteq Y'$. Now using the decoupling theorem for $C^2$ cones (discussion in section 9.3) and applying H\"older's inequality, we have
    \begin{align*}
        &\left\|\sum_{\square} \sum_{T \in \mathbb{W}_{\square}} \eta_{Y_{\square}} 
        f_{T}\right\|_{L^{p}(Q)} \\ &\lesssim_{\epsilon}(M'')^{1/2-1/p} \left(\sum_{\square \in \mathbb{B}} \left\|\sum_{T \in \mathbb{W}_{\square}}\eta_{{Y}_{\square}} f_{T}\right\|_{L^{p}(2Q)}^{p} \right)^{1/p}\\
        & + R^{-900}\left(\sum_{T \in \mathbb{W}}\norm{f_{T}}_{p}^{2}\right)^{1/2}.
    \end{align*}
    Now summing over all $Q$ gives us
    \begin{align*}
        \norm{f}_{L^{p}(Y)} 
        &\lesssim (\log R)^{O(1)}B^{100}K^{\epsilon/100}(M'')^{1/2-1/p}\\
        &\times \left(\left\|\sum_{T \in \mathbb{W}_{\square}}\eta_{{Y}_{\square}} f_{T}\right\|_{L^{p}(Y_{\square})}^{p} \right)^{1/p} +R^{-800}\left(\sum_{T \in \mathbb{W}}\norm{f_{T}}_{p}^{2}\right)^{1/2}.
    \end{align*}
    We will bound this by using the inductive assumption following the use of a Lorentz rescaling map.
    For each $\theta \in [-1,1]$, define the Lorentz rescaling map $L_{\theta}$ at $\theta$ by
    \begin{align*}
        & L_{\theta}\left( x_1\gamma(\theta_{\kappa}) +x_2 \frac{\gamma'(\theta_{\kappa})}{\abs{\gamma'(\theta_{\kappa})}} +x_3 \frac{\gamma \times \gamma'(\theta_{\kappa})}{\abs{\gamma \times \gamma'(\theta_{\kappa})}} \right) \\=  &x_1\gamma(\theta_{\kappa}) + K x_2 \frac{\gamma'(\theta_{\kappa})}{\abs{\gamma'(\theta_{\kappa})}} + K^2x_3 \frac{\gamma \times \gamma'(\theta_{\kappa})}{\abs{\gamma \times \gamma'(\theta_{\kappa})}} 
    \end{align*}
    For each $\square \in \mathbb{B}$ and $T \in \mathbb{W}_{\square}$, let $g_{T} = f_{T} \circ L$, where $L=L_{\theta_{k}(\square)}$. Then it can be shown that (as a consequence of a calculation in \cite{GGGHMW2022}) 
    \begin{equation*}
        \left\|\sum_{T \in \mathbb{W}_{\square}} f_{T}\right\|_{L^{p}(Y_{\square})} \leq K^{3/p}\left\|\sum_{T \in \mathbb{W}_{\square}}g_T\right\|_{L^{p}(L^{-1}Y_{\square})}
    \end{equation*}.
    A consequence of the inequalities $(\ref{1})$ and $(\ref{2})$ one can show that for each $T \in \mathbb{W}_{\square}$, the set $L^{-1}(T)$ is equivalent to a plank of length of $\Tilde{R}^{1+ \delta} \times \Tilde{R}^{1/2- \delta} \times \Tilde{R}^{\delta}$, where the longest direction parallel to $L^{-1}(\gamma \times \gamma')(\theta_{\tau(T)})$. This allows us to reduce to our initial setup so that we can use our inductive assumption. In \cite{GGGHMW2022} it is shown that
    \begin{align*}
        L(\tau) 
        &\subseteq \bigg\{ x_1\gamma(\theta_{\kappa}) +x_2 \frac{\gamma'(\theta_{\kappa})}{\abs{\gamma'(\theta_{\kappa})}} +x_3 \frac{\gamma \times \gamma'(\theta_{\kappa})}{\abs{\gamma \times \gamma'(\theta_{\kappa})}}:\\
        & 1 \leq \abs{x_1} \leq 2.01, \abs{x_2} \leq (1.01)\Tilde{R}^{-1/2}, \abs{x_3} \leq \Tilde{R}^{-1}\bigg\}.
    \end{align*}
    This neatly sets up the induction step. Inductively applying the theorem at scale $\Tilde{R}$ gives that 
    \begin{equation*}
        \left\|\sum_{T \in \mathbb{W}_{\square}} f_{T}\right\|_{L^{p}(Y_{\square})} \lesssim_{\epsilon, \delta} R^{\epsilon}K^{-2 \epsilon}\left(\frac{M'}{\abs{\mathbb{W_{\square}}}}\right)^{1/2-1/p}\left(\sum_{T \in \mathbb{W}_{\square}}\norm{f_T}_{p}^{2}\right)^{1/2}.
    \end{equation*}
    Thus summing over $\square \in \mathbb{B}$ we have that 
    \begin{equation*}
        \norm{f}_{L^{p}(Y)} \lesssim_{\epsilon, \delta} B^{10^{10} / \epsilon}K^{-\epsilon}\left(\frac{M' M'' }{\abs{\mathbb{W}_{\square}}}\right)^{1/2-1/p}\left(\sum_{\square \in \mathbb{B}}\left(\sum_{T \in \mathbb{T \in \mathbb{W}_{\square}}}\norm{f_{T}}_{p}^{2}\right)^{p/2}\right)^{1/p}
        \end{equation*}
        By the dyadically constant property of $\norm{f_{T}}_{p}$ (this is a consequence of the final pigeonholing step), we note that this is 
        \begin{equation*}
        \lesssim_{\epsilon, \delta} \left(\frac{M' M'' }{\abs{\mathbb{W}}}\right)^{1/2-1/p} \left(\frac{\abs{\mathbb{B}\mathbb{W_{\square}}}}{\abs{\mathbb{W}}}\right)^{1/p}\left(\sum_{T \in \mathbb{T \in \mathbb{W}_{\square}}}\norm{f_{T}}_{p}^{2}\right)^{p/2}.
        \end{equation*}
        We begin by showing that $\frac{\abs{\mathbb{B}}\abs{\mathbb{W_{\square}}}}{\abs{\mathbb{W}}} \lesssim 1$, This is because
        \begin{equation*}
            \abs{W}= \sum_{T \in \mathbb{W}}1 \geq \sum_{\square \in \mathbb{B}} \sum_{\substack{T \in \mathbb{W} \\
            \square = \square(T)}} \gtrsim \abs{\mathbb{B}}\abs{\mathbb{W}}.
        \end{equation*}
        Thus in order to complete the inductive argument we must show that $M'M'' \lesssim M$.
        Let $Q \subseteq Y'$ where $Q$ is a ball of radius $R^{1/2}$. We have, by the definition of $M$ that 
        \begin{align*}
            M &\gtrsim \sum_{\substack{T \in \mathbb{W}\\ 2T \cap Q \neq \emptyset}} \sum_{\substack{\square \in \mathbb{B}\\ \square = \square(T)}} 1\\
             &= \sum_{\square \in \mathbb{B}} \sum_{\substack{\square \in \mathbb{B}\\ \square = \square(T)\\ 2T \cap Q \neq \emptyset}} 1\\ 
             &\geq \sum_{\square \in \mathbb{B}} \sum_{\substack{\square \in \mathbb{B}\\ 2T \cap Q \neq \emptyset}} 1 .
        \end{align*}
        Using this we perform the following computation (by beginning with recalling the definition of $M'$ and $M''$),
        \begin{align*}
            M'M'' &\sim \sum_{\substack{\square \in \mathbb{B}\\ m(Y_{\square} \cap Q) >0}} M'\\
            &\leq \sum_{\substack{\square \in \mathbb{B}\\ m(Y_{\square} \cap Q) >0}} \sum_{Q_{\square} \subseteq Y_{\square}}M'\frac{m(Q_{\square} \cap 2Q)}{m(Y_{\square} \cap 2Q)}\\
             &\sim \sum_{\substack{\square \in \mathbb{B}\\ m(Y_{\square} \cap Q) >0}} \sum_{Q_{\square} \subseteq Y_{\square}} \sum_{\substack{T \in \mathbb{W}_{\square} \\ Q_{\square} \cap (1.5)T \neq \emptyset}} \frac{m(Q_{\square} \cap 2Q)}{m(Y_{\square} \cap 2Q)}\\
             &= \sum_{\substack{\square \in \mathbb{B}\\ m(Y_{\square} \cap Q) >0}} \sum_{T \in \mathbb{W}_{\square}} \sum_{\substack{Q_{\square} \subseteq Y_{\square} \\ Q_{\square} \cap (1.5)T \neq \emptyset}} \frac{m(Q_{\square} \cap 2Q)}{m(Y_{\square} \cap 2Q)}\\
             &\leq \sum_{\substack{\square \in \mathbb{B}\\ m(Y_{\square} \cap Q) >0}} \sum_{\substack{T \in \mathbb{W}_{\square}\\ 2T \cap \neq \emptyset}} \sum_{Q_{\square} \subseteq Y_{\square}} \frac{m(Q_{\square} \cap 2Q)}{m(Y_{\square} \cap 2Q)}\\
            &\lesssim \sum_{\square \in \mathbb{B}} \sum_{\substack{T \in \mathbb{W}_{\square}\\ 2T \cap Q \neq \emptyset}} 1\\
            &\lesssim M,
        \end{align*}
        where in the second to last inequality we have used the observation that if $Q_{\square} \cap 2Q \neq \emptyset$ and if $T \in \mathbb{W}_{\square}$ is such that $Q_{\square} \cap (1.5)T \neq \emptyset$, then we must have that $2T \cap Q \neq \emptyset$.

        This closes the induction and completes the proof.
    \end{proof}

\section{Appendix}

\subsection{Proof of flat decoupling}
We give a proof of Theorem \ref{thm_flat_decoupling}. First, by a simple tiling argument and the triangle and H\"older's inequalities, we may assume $K=1$.

The idea is to use interpolation. The case $p=2$ follows from Plancherel's identity, and the case $p=\infty$ follows from the triangle and H\"older's inequalities. However, there are a few technical obstacles we need to consider. First, it is not of the form of a general sublinear operator since we have extra Fourier support conditions on each $f_T$. Also, on the right hand side of the decoupling inequality we have a mixed norm. The full detail is as follows.

\begin{proof}[Rigorous proof of flat decoupling]
For each $T\in \mathbb T$, let $\psi_T$ be a Schwartz function such that $\widehat {\psi_T}$ is supported on $CT$ and equals $1$ on $T$. Define the vector-input linear operator $L$ by
$$
L((f_T)_{T\in \mathbb T})=\sum_{T\in \mathbb T}f_T* \psi_T,
$$
where $f_T$ are arbitrary functions (with no Fourier support assumption) in $L^p(\R^3)$. Then by Plancherel's identity, we have
\begin{equation}
    \norm{L((f_T)_{T\in \mathbb T})}_{L^2(\R^3)}=\norm{\norm{f_T* \psi_T}_{L^2(\R^3)}}_{\ell^2(T\in \mathbb T)}\lesssim_C \norm{\norm{f_T}_{L^2(\R^3)}}_{\ell^2(T\in \mathbb T)},
\end{equation}
using Young's inequality, since $\norm{\psi_T}_{L^1(\R^3)}\sim 1$. On the other hand, by the triangle and H\"older's inequalities, we have
\begin{align*}
    \norm{L((f_T)_{T\in \mathbb T})}_{L^\infty(\R^3)}
    &\le(\#\mathbb T)^{1/2}\norm{\norm{f_T* \psi_T}_{L^\infty(\R^3)}}_{\ell^2(T\in \mathbb T)}\\
    &\lesssim_C (\#\mathbb T)^{1/2}\norm{\norm{f_T}_{L^\infty(\R^3)}}_{\ell^2(T\in \mathbb T)}.
\end{align*}
Thus, using the interpolation theorem for mixed normed spaces (see, for instance, Theorem 2 of Section 7 of \cite{Mixed_norm_decoupling}), we have for every $2\le p\le \infty$ that
\begin{equation}
    \norm{L((f_T)_{T\in \mathbb T})}_{L^p(\R^3)}\lesssim_C (\#\mathbb T)^{\frac 1 2-\frac 1 p}\norm{\norm{f_T}_{L^p(\R^3)}}_{\ell^2(T\in \mathbb T)}.
\end{equation}
Lastly, given any $f_T$ Fourier supported on $T$, we have $f_T=f_T*\psi_T$. Thus
\begin{equation}
    \norm{\sum_{T\in \mathbb T}f_T}_{L^p(\R^3)}=\norm{\sum_{T\in \mathbb T}f_T*\psi_T}_{L^p(\R^3)}\lesssim_C (\#\mathbb T)^{\frac 1 2-\frac 1 p}\norm{\norm{f_T}_{L^p(\R^3)}}_{\ell^2(T\in \mathbb T)}.
\end{equation}
\end{proof}

\subsection{Proof of decoupling for \texorpdfstring{$C^2$}{Lg} curves}
In this section, we roughly sketch the proof of Theorem 10 using the language of \cite{Demeter_textbook}, and point out the changes we need to make when we only have $C^2$ regularity. For simplicity of notation, we take $p=6$.

{\it Remark.} Unless otherwise stated, an uncoloured reference without a hyperlink will mean a label in \cite{Demeter_textbook}.

Below and henceforth, for each Schwartz function $f:\R^2\to \C$ and each interval $I\sub \R$, we denote by $f_I$ the Fourier restriction of $f$ onto $I\times \R$, namely, $\widehat{f_I}:=\widehat {f}1_{I\times \R}$.

\begin{thm}[Bourgain-Demeter decoupling for planar $C^2$ curves]\label{thm_decoupling_C2}
   Let $\gamma:[0,1]\to \R$ be a $C^2$-curve such that $|\gamma''(t)|\sim 1$. Let $\delta\in 4^{-\N}$ and let $N_\gamma(\delta)$ be the (vertical) $\delta$-neighbourhood of the graph of $\gamma$ over $[0,1]$, namely,
   \begin{equation}
       N_\gamma(\delta):=\{(t,\gamma(t)+s):t\in [0,1],|s|\le \delta\}.
   \end{equation}   
   Let $\mathcal I_{\delta^{1/2}}$ be a tiling of $[0,1]$ into intervals of length $\delta^{1/2}$. Let $C>0$ and denote by $D_{\gamma}(\delta)$ the smallest constant such that the following inequality holds for all $f$ Fourier supported on $N_\gamma(C\delta)$:
   \begin{equation}
       \norm{f}_{L^6(\R^2)}\leq D_{\gamma}(\delta)\norm{\norm{f_I}_{L^6(\R^2)}}_{\ell^2(I\in \mathcal I_{\delta^{1/2}})}.
   \end{equation}
   Then the following holds: for every $\eps>0$ there exists some $C_{\eps,\gamma,C}$ such that
   \begin{equation}
       D_{\gamma}(\delta)\leq C_{\eps,\gamma,C} \delta^{-\eps},\quad \forall \delta\in 4^{-\N}.
   \end{equation}
\end{thm}

The following corollary will be useful when the curve is given by a parametric equation.
\begin{cor}\label{cor_decoupling_C^2_curve}
    Let $\phi:[0,1]\to \R^2$ be a $C^2$-curve with nonzero velocity and nonzero curvature. Let $\delta\in 4^{-\N}$ and let $N'_\phi(\delta)$ be the $\delta$-neighbourhood of the image of $\phi$ in $\R^2$, namely,
    \begin{equation}
        N'_\phi(\delta):=\{\phi(t)+s:t\in [0,1],|s|\le \delta\}.
    \end{equation}
    Let $\mathcal T_{\delta^{1/2}}$ be a covering of $N'_\phi(C\delta)$ by $O(1)$-overlapping rectangles $T$ of dimensions $O(\delta^{1/2})\times O(\delta)$, such that $\cup \mathcal T_{\delta^{1/2}}\sub N'_\phi(O(\delta))$. Then
    \begin{equation}
       \norm{f}_{L^6(\R^2)}\lesssim_\eps \delta^{-\eps} \norm{\norm{\mathcal R_T f}_{L^6(\R^2)}}_{\ell^2(T\in \mathcal T_{\delta^{1/2}})},\quad \forall \eps>0,
   \end{equation}
   where $\mathcal R_T f:=(\hat f 1_T)^\vee$ and the implicit constant also depends on $C$ and all implicit constants in the big $O$ notations above.
\end{cor}
\begin{proof}
    Since $|\phi'|=1$, by the implicit function theorem and a compactness argument, we can cut the image of $\phi$ into $O(1)$ many pieces, such that each piece after a suitable rotation becomes the graph of some $C^2$ function $\gamma$ with non-vanishing curvature. By the triangle and Cauchy-Schwarz inequalities, it suffices to consider each such piece. However, after rotation, we can view the tiling $\mathcal T_{\delta^{1/2}}$ as a $O(1)$ overlapping cover of $N_\gamma(\delta)$ by rectangles of dimensions $\delta^{1/2}\times \delta$ that are all contained in $N_\gamma(O(\delta))$. By a simple tiling argument, we may apply Theorem \ref{thm_decoupling_C2} to get our desired conclusion, since decoupling is invariant under rotations.
\end{proof}

Now we come to the proof of Theorem \ref{thm_decoupling_C2}. One of the most important modifications is for Proposition 10.2.

Let $0<a<1<A$ be fixed constants, and denote by $\mathcal F=\mathcal F(a,A)$ the family of all $C^2$ functions $\gamma:[0,1]\to \R$ satisfying
\begin{equation}
\begin{aligned}
    &\sup_{t\in [0,1]}\{|\gamma(t)|,|\gamma'(t)|,|\gamma''(t)|\}\le A,
    &\inf_{t\in [0,1]} |\gamma''(t)|\ge a
\end{aligned}
\end{equation}

\begin{prop}[Parabolic rescaling for $\mathcal F(a,A)$]\label{prop_10.2_Demeter}
Let $\sigma\in [\delta,1]\cap 4^{-\N}$. For each $J\in \mathcal I_{\sigma^{1/2}}$, denote by $\mathcal I_{\delta^{1/2}}(J)$ the collection of all intervals $I\in \mathcal I_{\delta^{1/2}}$ that are also contained in $J$. 

Then there exists some $\tilde \gamma\in \mathcal F(a,A)$ such that for every Schwartz function $f$ Fourier supported on $N_\gamma(\delta)$, we have
\begin{equation}
       \norm {f_J}_{L^6(\R^2)}\leq D_{\tilde\gamma}(\sigma^{-1}\delta)\norm{\norm{f_I}_{L^6(\R^2)}}_{\ell^2(I\in \mathcal I_{\delta^{1/2}}(J))}.
   \end{equation}    
\end{prop}
\begin{proof}
    Write $J=[c,c+\sigma^{1/2}]$ and let 
    \begin{equation}
        \widehat g(t,s):=\widehat f(\sigma^{1/2}t+c,\sigma s-\gamma(c)-\gamma'(c)\sigma^{1/2}t),
    \end{equation}
    which is supported on $N_{\tilde \gamma}(\sigma^{-1}\delta)$ where
    \begin{equation}
        \tilde \gamma(t):=\sigma^{-1}\left[\gamma(\sigma^{1/2}t+c)-\gamma(c)-\gamma'(c)\sigma^{1/2}t\right].
    \end{equation}
    If $\gamma\in \mathcal F(a,A)$, we can check that $\tilde \gamma\in \mathcal F(a,A)$ as well. Applying the definition of $D_{\tilde \gamma}(\sigma^{-1}\delta)$ to $g$, we have
    \begin{equation}
       \norm {g}_{L^6(\R^2)}\leq D_{\tilde\gamma}(\sigma^{-1}\delta)\norm{\norm{g_{\tilde I}}_{L^6(\R^2)}}_{\ell^2(\tilde I\in \mathcal I_{\sigma^{-1/2}\delta^{1/2}})}.
   \end{equation}   
   Reversing the change of variables and cancelling the same Jacobian factor on both sides, we obtain the desired result.
\end{proof}

The next modification will be for Proposition 10.14.

Fix $I_1=[0,1/4]$ and $I_2=[1/2,1]$. Define $BD_\gamma(\delta)$ to be the smallest constant such that the inequality
\begin{equation}
    \norm{|f_1f_2|^{1/2}}_{L^6(\R^2)}\le BD_\gamma(\delta)\prod_{i=1}^2\norm{\norm{f_I}_{L^6(\R^2)}}^{1/2}_{\ell^2(I\in \mathcal I_{\delta^{1/2}}(I_i))}
\end{equation}
holds for all $f_i$ with Fourier support on $N_{\gamma}(\delta)\cap (I_i\times \R)$, $i=1,2$ respectively. For our convenience, we also define
\begin{equation}
    D(\delta):=\sup_{\gamma\in \mathcal F(a,A)}D_\gamma(\delta),\quad BD(\delta):=\sup_{\gamma\in \mathcal F(a,A)}BD_\gamma(\delta).
\end{equation}
We also have the following bilinear analogue of Proposition \ref{prop_10.2_Demeter}.
\begin{prop}[Bilinear parabolic rescaling for $\mathcal F(a,A)$]\label{prop_10.2_Demeter_bilinear}
Let $\sigma\in [\delta,1/16]\cap 4^{-\N}$. For $i=1,2$, intervals $J_i\in \mathcal I_{\sigma^{1/2}}(I_i)$ and every Schwartz function $f$ Fourier supported on $N_\gamma(\delta)$, we have
\begin{equation}
       \norm {|f_{J_1}f_{J_2}|^{1/2}}_{L^6(\R^2)}\leq BD(\sigma^{-1}\delta)\prod_{i=1}^2\norm{\norm{f_I}_{L^6(\R^2)}}^{1/2}_{\ell^2(I\in \mathcal I_{\delta^{1/2}}(J_i))}.
   \end{equation}    
\end{prop}
\begin{proof}
    The proof is similar to that of Proposition \ref{prop_10.2_Demeter}, so we leave it as an exercise.
\end{proof}
\begin{lem}[Sub-multiplicativity of decoupling constants]\label{lem_sub_multiplicativity}
    For $\delta_1,\delta_2\in 4^{-\N}$ we have
    \begin{equation}
        D(\delta_1\delta_2)\le D(\delta_1)D(\delta_2),\quad BD(\delta_1\delta_2)\le BD(\delta_1)BD(\delta_2).
    \end{equation}
\end{lem}
\begin{proof}
    The first inequality follows from Minkowski's inequality and Proposition \ref{prop_10.2_Demeter}. The second inequality follows from H\"older's inequality, Minkowski's inequality and Proposition \ref{prop_10.2_Demeter_bilinear}.
\end{proof}
\begin{prop}[Reduction to bilinear decoupling]
    For each $\eps>0$, there exists $C_\eps$ such that for all $\delta\in 4^{-\N}$ we have
    \begin{equation}
        D(\delta)\leq C_\eps \delta^{-\eps}\max_{\sigma\in [\delta,1]\cap 4^{-\N}}BD(\sigma).
    \end{equation}
\end{prop}
\begin{proof}
    Everything in the proof of Proposition 10.14 should remain unchanged (modulo difference in notation), except that we now use Proposition \ref{prop_10.2_Demeter_bilinear} for the family $\mathcal F(a,A)$.
\end{proof}
The content of Corollary 10.18 is unchanged.
\begin{prop}[Corollary 10.18]\label{prop_bootstrap_decoupling}
    Assume $\delta=4^{-2^u}$ for some $u\in \N$. Then for each $1\le s\le u$,
    \begin{equation}
        BD(\delta)\lesssim_{s,\eps} \delta^{-\eps-2^{-s-2}}\prod_{l=1}^s D(\delta^{1-2^{-l}})^{2^{l-s-1}}.
    \end{equation}
\end{prop}
\begin{proof}[Proof of Proposition \ref{prop_bootstrap_decoupling}]
    All arguments in \cite{Demeter_textbook} leading to Corollary 10.18 are the same, except the paragraph right after (10.16). In this case, the longer side of each $T\in \mathcal F_I$ is now pointing in the direction $(\gamma'(c_I),-1)$. Since $|\gamma''|\sim 1$, when $|c_I-c_{\tilde I}|\sim 1$, the directions are also $\sim 1$ separated. This ensures the bilinear Kakeya inequality in Theorem 10.21 is applicable. We also remark here that the analogue of the rigorous proof of Theorem 11.10 also remains unchanged.
\end{proof}
Now we come to the proof of Theorem \ref{thm_decoupling_C2}.
\begin{proof}
    Iterating Proposition \ref{prop_bootstrap_decoupling} in the same way as in \cite{Demeter_textbook}, we obtain $BD(\delta)\lesssim_\eps \delta^{-\eps}$ when $\delta\in 4^{-2^\N}$. 

    We still need to prove $BD(\delta)\lesssim_\eps \delta^{-\eps}$ for a general $\delta\in 4^{-\N}$. This part was not given in \cite{Demeter_textbook}, so we give it here for completeness. The idea is a simple induction on scales argument.
    
    Let $m\ge 1$ and assume we have proved $BD(\delta)\lesssim_\eps \delta^{-\eps}$ for every $\delta=4^{-n}$ where $0\le n< 2^{m}$. We need to show $BD(\delta)\lesssim_\eps \delta^{-\eps}$ for every $\delta=4^{-n}$ where $2^{m}< n< 2^{m+1}$. Then $n-2^m\in (0,2^m)$, and so by Lemma \ref{lem_sub_multiplicativity}, we have
    \begin{equation}
        BD(\delta)=BD(4^{-2^m}4^{2^m-n})\le BD(4^{-2^m})BD(4^{2^m-n})\lesssim_\eps 4^{2^m\eps}4^{(n-2^m)\eps}=\delta^{-\eps}, 
    \end{equation}
    using both the result already established for $\delta\in 4^{-2^\N}$ and the induction hypothesis.
\end{proof}

\subsection{Proof of decoupling for \texorpdfstring{$C^2$}{Lg} cones}\label{sec_decoupling_C2_cone}
In this section, we give a proof of Theorem \ref{thm_Bourgain-Demeter_decoupling}. Although the most important idea of the proof is already given in the appendix of \cite{GGGHMW2022}, we would like to provide more details. We also refer the reader to Section 12.2 of \cite{Demeter_textbook} as well as Section 8 of \cite{BD2015}.

We first prove Part \ref{item_BD_decoupling_01}. For the covering part, let $\theta\in [0,1]$ be arbitrary, we choose $\omega\in \delta^{1/2}\Z\cap [0,1]$ that is closest to $\theta$. Then the result follows from Proposition \ref{prop_error_e_i}. The bounded overlap follows exactly from Part \ref{item_2_same_and_distinct} of Lemma \ref{lem_same_and_distinct}. (Note the planks $A_{\omega,\delta^{1/2}}$ are denoted $C_1 P_{\omega,\delta^{1/2}}$ there.)

\subsubsection{Preliminary reductions}\label{sec_reparametrization}
To prove Part \ref{item_BD_decoupling_02}, we use an iteration argument known as Pramanik-Seeger type iteration \cite{PS2007}. For simplicity of notation we only consider $p=6$. By splitting $\Gamma$ into $\lesssim |\log K|$ many pieces and rescaling each piece back to $1\le\xi_3\le 2$, we may assume $K=1$, which leads to 
a loss of $O(K^2|\log K|)$ by the triangle and Cauchy-Schwarz inequalities (so $K^{10}$ is safe). Lastly, also by a simple tiling argument and triangle and Cauchy-Schwarz inequalities, we may assume the collection $\{A_{\omega,\delta^{1/2}}\}$ is disjoint.

\subsubsection{Introducing intermediate scale}

For $\delta\in 4^{-\N}$, we denote by $D(\delta)$ the smallest constant such that for every family of Schwartz functions $\{f_\omega:\omega\in \Omega\}$ where each $f_\omega$ has Fourier support in $A_{\omega,\delta^{1/2}}$, we have
        \begin{equation}
            \left\|\sum_{\omega\in \Omega}f_\omega\right\|_{L^6(\R^3)}\leq D(\delta)\norm{\norm{f_\omega}_{L^6(\R^3)}}_{\ell^2(\omega\in \Omega)}.
        \end{equation}
Our goal is to show that $D(\delta)\lesssim_\eps \delta^{-\eps}$ for every $\eps>0$.

To this end, we define an intermediate scale 
\begin{equation}
    \lambda:=\delta^{(1-\eps)/2},
\end{equation}
and consider the lattice $\Sigma=\lambda \Z\cap [0,1]$. The choice of $\lambda$ will be clear soon. We then recall \eqref{eqn_CP_sigma_lambda} (with $K=1$):
\begin{equation}
    CP_{\sigma,\lambda}=\left\{\sum_{i=1}^3 \xi_i \mathbf e_i(\sigma):|\xi_1|\le C\lambda^2,|\xi_2|\le C\lambda,C^{-1}\le \xi_3\le C \right\}.
\end{equation}
Since $A_{\omega,\delta^{1/2}}\sub C_1P_{\omega,\lambda}$, by Lemma \ref{lem_cover_lambda}, we know that the collection $\{C_3P_{\sigma,\lambda}:\sigma\in \Sigma\}$ forms a $O(1)$-overlapping cover of $\cup_{\omega\in \Omega}A_{\omega,\delta^{1/2}}$. 

Recall the notation $\omega\prec \sigma$, which means that $\omega-\sigma\in (\lambda/2,\lambda/2]$. Define for each $\sigma\in \Sigma$ that 
\begin{equation}
    f_\sigma=\sum_{\omega\prec \sigma}f_\omega,
\end{equation}
which has Fourier support in $C_3P_{\sigma,\lambda}$ by Part \ref{item_2_cover_lambda} of Lemma \ref{lem_cover_lambda}. Also, we have 
\begin{equation}
    \sum_{\omega\in \Omega}f_\omega=\sum_{\sigma\in \Sigma}f_\sigma.
\end{equation}
Thus, we may first apply the definition of $D(\lambda^2)$ to get
\begin{equation}\label{eqn_applying_D(lambda)}
    \left\|\sum_{\omega\in \Omega}f_\omega\right\|_{L^6(\R^3)}\leq D(\lambda^2)\norm{\norm{f_\sigma}_{L^6(\R^3)}}_{\ell^2(\sigma\in \Sigma)}.
\end{equation}
It remains to estimate each $\norm{f_\sigma}_{L^6(\R^3)}$. 

\subsubsection{Approximation by cylinder}
We also encounter some technical difficulty here, namely, the curve $\mathbf e_3$ is merely $C^1$ by Lemma \ref{lem_frame_derivatives} since the torsion function $\uptau$ may not be differentiable. This will hinder applications of Theorem \ref{thm_decoupling_C2} as well as multiple Taylor expansions.

To deal with this issue, we will reparametrize $\mathbf e_3$; note that the planks remain unchanged.
Let us be more precise. Recall by Lemma \ref{lem_frame_derivatives} that $\mathbf e'_3=-\uptau e_2$. Put
\begin{equation}
    s=s(\theta)=\int_0^\theta -\frac 1 {\uptau (t)}dt,
\end{equation}
and redefine
\begin{equation}\label{eqn_reparametrization}
    \tilde {\mathbf e}_i(\theta)=\mathbf e_i(s(\theta)),\quad i=1,2,3.
\end{equation}
In this way, we obtain the new Frenet frame formula:
\begin{equation}\label{eqn_frenet_frame_new}
    \tilde {\mathbf e}'_1=-\uptau^{-1}\tilde {\mathbf e}'_2,\quad \tilde {\mathbf e}'_2=\uptau^{-1}\tilde {\mathbf e}_1-\tilde {\mathbf e}_3,\quad \tilde {\mathbf e}'_3=\tilde {\mathbf e}_2.
\end{equation}
Thus $\tilde {\mathbf e}_3$ becomes $C^2$, at the cost that $\tilde {\mathbf e}_1$ may not necessarily be $C^2$.

For simplicity, we abuse notation and still use $\mathbf e_i$ to mean $\tilde {\mathbf e}_i$.    

Fix $\sigma \in \Sigma$. For simplicity of notation we will assume without loss of generality that 
\begin{equation}\label{eqn_e_i=e_i}
    \mathbf e_1(\sigma)=(1,0,0),\quad \mathbf e_2(\sigma)=(0,1,0),\quad \mathbf e_3(\sigma)=(0,0,1).
\end{equation}
Also write
\begin{equation}
    \mathbf e_i=(e_{i1},e_{i2},e_{i3}).
\end{equation}

Let $\eps>0$ be arbitrary. Since we allow a loss of the form $\delta^{-\eps}$, we may cut $[0,1]$ into smaller intervals of length $\delta^{\eps}$. By the triangle and Cauchy-Schwarz inequalities, it suffices to assume all angles $\theta,\omega,\sigma\in [0,\delta^{\eps}]$. For the same reason, we may partition each $A_{\omega,\delta^{1/2}}$ and assume $|\xi_3-1|\le \delta^{\eps}$ on $A_{\omega,\delta^{1/2}}$.

The key is the following geometric observation. It is more complicated than the argument outlined in the proof of Theorem 11 in \cite{GGGHMW2022}, since we have chosen to perform the approximation directly, instead of first reducing to the form $\mathbf e_3(t)=(1,t,\phi(t))$ for some $\phi\in C^2$.
\begin{lem}\label{lem_PS_cylinder}
    For each $\omega\prec \sigma$, the plank $A_{\omega,\delta^{1/2}}$ is contained in the $O(\delta)$ neighbourhood of the part of the cylinder 
    \begin{align}
        R_{\sigma,\omega,\lambda}&:=\left\{
        (e_{31}(\theta),e_{32}(\theta),\xi_3):|\theta-\omega|\lesssim \delta^{1/2},|\xi_3-1|\lesssim \delta^{\eps}\right\}.
    \end{align}
\end{lem}

\begin{proof}[Proof of Lemma \ref{lem_PS_cylinder}]

Let $\xi=\sum_{i=1}^3 \xi_i \mathbf e_i(\omega)\in A_{\omega,\delta^{1/2}}$. It suffices to find $\theta,\xi'_3$ such that $|\xi-(e_{31}(\theta),e_{32}(\theta),\xi'_3)|\lesssim \lambda^2\delta^\eps$.

We Taylor expand $\mathbf e_3$ with respect to the centre $\omega$. Using \eqref{eqn_frenet_frame_new}, $|\xi_1|\lesssim \delta$ and $|\omega-\theta|\le\delta^{1/2}$, we get
\begin{align}
    &\sum_{i=1}^3 \xi_1 \mathbf e_i(\omega)-(e_{31}(\theta),e_{32}(\theta),\xi_3' )\nonumber\\
    =&\Big(\xi_2 e_{21}(\omega)+(\xi_3-1)e_{31}(\omega)+e_{21}(\theta-\omega)+O(\delta),\nonumber\\
    &\xi_2 e_{22}(\omega)+(\xi_3-1)e_{32}(\omega)+e_{22}(\omega)(\theta-\omega)+O(\delta)\nonumber\\
    &\xi_2 e_{23}(\omega)+\xi_3 e_{33}(\omega)-\xi_3'+O(\delta)\Big).\label{eqn_lemma_cylinder_proof}
\end{align}
We thus choose $\xi'_3=\xi_2 e_{23}(\omega)+\xi_3 e_{33}(\omega)$, which satisfies $|\xi'_3-1|\lesssim \delta^\eps$ since $|e_{33}(\omega)-1|=|e_{33}(\omega)-e_{33}(\sigma)|\lesssim \delta^\eps$. This makes the third coordinate of \eqref{eqn_lemma_cylinder_proof} of the order $O(\delta)$.

Next, we put
\begin{equation}\label{eqn_theta_proof_lemma_cylinder}
    \theta=\omega+\frac {\xi_2 e_{22}(\omega)-(\xi_3-1)e_{32}(\omega)}{e_{22}(\omega)},
\end{equation}
which satisfies $|\theta-\omega|\lesssim \delta^{1/2}$. Indeed, the denominator is $\sim 1$ since $e_{22}(\omega)\sim 1$. For the numerator, we have $|\xi_2|\lesssim \delta^{1/2}$ and $|\xi_3-1|\leq \delta^\eps$, and using Lemma \ref{eqn_e_i=e_i} and \eqref{eqn_frenet_frame_new}, we get
\begin{equation}
    |e_{32}(\omega)|=|e_{32}(\omega)-e_{32}(\sigma)|\lesssim |\omega-\sigma|\le \lambda,
\end{equation}
and so $|\theta-\omega|\lesssim \delta^\eps \lambda\le \delta^{1/2}$. This makes the second coordinate of \eqref{eqn_lemma_cylinder_proof} of the order $O(\delta)$.

It remains to estimate the first coordinate. Using \eqref{eqn_theta_proof_lemma_cylinder}, the first coordinate of \eqref{eqn_lemma_cylinder_proof} is within $O(\delta)$ of
\begin{equation}\label{eqn_Sep_1_10pm}
    e_{22}(\omega)^{-1}(\xi_3-1)(e_{31}(\omega)e_{22}(\omega)-e_{21}(\omega)e_{32}(\omega)).
\end{equation}
Now we recall $|\omega-\sigma|\le \lambda$, and use \eqref{eqn_frenet_frame_new}, Taylor expansion and \eqref{eqn_e_i=e_i} to get 
\begin{equation}
    e_{31}(\omega)=e_{31}(\sigma)+e_{21}(\sigma)(\omega-\sigma)+O(\lambda^2)=O(\lambda^2).
\end{equation}
Similarly, we also have $e_{32}(\omega)=O(\lambda^2)$. Thus \eqref{eqn_Sep_1_10pm} is bounded above by $O(\lambda^2)$, and so the first coordinate of \eqref{eqn_lemma_cylinder_proof} is of the order $O(\lambda^2\delta^\eps)=O(\delta)$.
\end{proof}

\subsubsection{Applying decoupling for \texorpdfstring{$C^2$}{Lg} curves}
The most important consequence of Lemma \ref{lem_PS_cylinder} is that if we project every plank $R_{\sigma,\omega,\lambda}$ onto $\mathbf e_3(\sigma)^\perp$, then we are essentially in the case of decoupling for $C^2$ curves as in Corollary \ref{cor_decoupling_C^2_curve}. This is a general philosophy of decoupling called the projection principle, namely, the decoupling constant for a family of subsets is dominated by the decoupling constant for the same family of subsets under any fixed orthogonal projection. Using Corollary \ref{cor_decoupling_C^2_curve} and noting the $O(1)$ overlap, the result is that
\begin{equation}\label{eqn_projection_decoupling}
    \left\|f_\sigma\right\|_{L^6(\R^3)}\lesssim_\eps \delta^{-\eps^2}\norm{\norm{f_\omega}_{L^6(\R^3)}}_{\ell^2(\omega\prec\sigma)}.
\end{equation}
The choice of $\delta^{-\eps^2}$ but not $\delta^{-\eps}$ is necessary for the bootstrap inequality in the next subsection to give the desired conclusion.

We give more details of the proof of \eqref{eqn_projection_decoupling} here. Fix $z\in \R$. For each $\omega\prec \sigma$, define $g_{z,\omega}:\R^2\to \C$ by 
\begin{equation}
    \widehat {g_{z,\omega}}(\xi_1,\xi_2):=\int_{\R}\widehat {f_\omega}(\xi_1,\xi_2,\xi_3)e(z \xi_3)d\xi_3.
\end{equation}
Then by Lemma \ref{lem_PS_cylinder}, $\widehat {g_{z,\omega}}$ is supported on the $O(\delta)$ neighbourhood of
\begin{equation}
    \{(e_{31}(\theta),e_{32}(\theta)):|\theta-\omega|\lesssim \delta^{1/2}\},
\end{equation}
which has $O(1)$ overlap, by the same proof of Lemma \ref{lem_same_and_distinct} as in Section \ref{sec_covering_lemmas}. The planar curve $(e_{31}(\theta),e_{32}(\theta))$ is $C^2$ and has curvature comparable to $ 1$. Thus we may apply Corollary \ref{cor_decoupling_C^2_curve} to the family $\{g_{z,\omega}:\omega\prec \sigma\}$ to obtain
\begin{equation}
    \left\|\sum_{\omega\prec \sigma}g_{z,\omega}\right\|_{L^6(\R^2)}\lesssim_\eps \delta^{-\eps^2}\norm{\norm{g_{z,\omega}}_{L^6(\R^2)}}_{\ell^2(\omega\prec\sigma)}.
\end{equation}
The desired inequality \eqref{eqn_projection_decoupling} then follows by Fubini's theorem and Minkowski's inequality.

\subsubsection{Iteration}
Combining \eqref{eqn_applying_D(lambda)} and \eqref{eqn_projection_decoupling}, we thus get
\begin{equation}    
            \left\|\sum_{\omega\in \Omega}f_\omega\right\|_{L^6(\R^3)}\leq D(\lambda^2)C_\eps\delta^{-\eps^2}\norm{\norm{f_\omega}_{L^6(\R^3)}}_{\ell^2(\omega\in \Omega)}, \quad \forall \eps>0, 
\end{equation}
which, by definition of $D(\delta)$, implies the following bootstrap inequality:
\begin{equation}
    D(\delta)\le D(\delta^{1-\eps})C_\eps\delta^{-\eps^2},\quad \forall \eps>0.
\end{equation}
It suffices to iterate this bootstrap inequality. For $\delta\ge 1/2$ we trivially have $D(\delta)\sim 1$. If $\delta<1/2$, we let $N$ be the largest integer such that $\delta^{(1-\eps)^N}< 1/2$, whence $N\lesssim \eps^{-1}\log\log \delta^{-1}$, and
\begin{equation}
    D(\delta)\le D(\delta^{(1-\eps)^{N+1}})C_\eps^{N}\delta^{-\eps^2 (1+(1-\eps)+\cdots+(1-\eps)^N)}\lesssim_\eps \delta^{-2\eps}.
\end{equation}

\bibliographystyle{alpha}
\bibliography{sources}

\newcommand{\etalchar}[1]{$^{#1}$}
\begin{thebibliography}{GGG{\etalchar{+}}22}

\bibitem[BD15]{BD2015}
J.~Bourgain and C.~Demeter.
\newblock The proof of the {$l^2$} decoupling conjecture.
\newblock {\em Ann. of Math. (2)}, 182(1):351--389, 2015.

\bibitem[BD17]{BD2017_study_guide}
J.~Bourgain and C.~Demeter.
\newblock A study guide for the {$l^2$} decoupling theorem.
\newblock {\em Chin. Ann. Math. Ser. B}, 38(1):173--200, 2017.

\bibitem[BP61]{Mixed_norm_decoupling}
A.~Benedek and R.~Panzone.
\newblock The space {$L\sp{p}$}, with mixed norm.
\newblock {\em Duke Math. J.}, 28:301--324, 1961.

\bibitem[Dem20]{Demeter_textbook}
C.~Demeter.
\newblock {\em Fourier restriction, decoupling, and applications}, volume 184 of {\em Cambridge Studies in Advanced Mathematics}.
\newblock Cambridge University Press, Cambridge, 2020.

\bibitem[Fal82]{Falconer}
K.~J. Falconer.
\newblock Hausdorff dimension and the exceptional set of projections.
\newblock {\em Mathematika}, 29(1):109–115, 1982.

\bibitem[FO14]{FasslerOrponen}
K.~F\"{a}ssler and T.~Orponen.
\newblock On restricted families of projections in $\mathbb{R}^3$.
\newblock {\em Proceedings of the London Mathematical Society}, 109(2):353--381, 2014.

\bibitem[GGG{\etalchar{+}}22]{GGGHMW2022}
S.~Gan, S.~Guo, L.~Guth, T.~L.~J. Harris, D.~Maldague, and H.~Wang.
\newblock On restricted projections to planes in $\mathbb{R}^3$.
\newblock {\em Amer. J. Math. (to appear)}, 2022.
\newblock arXiv:2207.13844.

\bibitem[GGM22]{GanGuthMaldague}
S.~Gan, L.~Guth, and D.~Maldague.
\newblock An exceptional set estimate for restricted projections to lines in $\mathbb{R}^3$, 2022.

\bibitem[GIOW18]{guth2018falconers}
L.~Guth, A.~Iosevich, Y.~Ou, and H.~Wang.
\newblock On {F}alconer's distance set problem in the plane.
\newblock 2018.

\bibitem[Har22]{Harris2022}
T.~L.~J. Harris.
\newblock Improved bounds for restricted projection families via weighted {F}ourier restriction.
\newblock {\em Anal. PDE}, 15(7):1655--1701, 2022.

\bibitem[Har23]{harris2023length}
T.~L.~J. Harris.
\newblock Length of sets under restricted families of projections onto lines, 2023.

\bibitem[JJK14]{JarvenpaaJarvenpaa}
E.~J\"{a}rvenp\"{a}\"{a}, M.~J\"{a}rvenp\"{a}\"{a}, and T.~Keleti.
\newblock Hausdorff dimension and non-degenerate families of projections.
\newblock {\em J. Geom. Anal.}, 24(4):2020--2034, 2014.

\bibitem[JJLL08]{JJLL}
E.~J\"{a}rvenp\"{a}\"{a}, M.~J\"{a}rvenp\"{a}\"{a}, F.~Ledrappier, and M.~Leikas.
\newblock One-dimensional families of projections.
\newblock {\em Nonlinearity}, 21(3):453--463, 2008.

\bibitem[Kau68]{Kaufman}
R.~Kaufman.
\newblock On hausdorff dimension of projections.
\newblock {\em Mathematika}, 15(2):153--155, 1968.

\bibitem[KM75]{KaufmanMattila}
R.~Kaufman and P.~Mattila.
\newblock Hausdorff dimension and exceptional sets of linear transformations.
\newblock {\em Annales Academiae Scientiarum Fennicae. Series A 1, Mathematica}, 1:387–392, 1975.

\bibitem[KOV22]{KOV}
A.~Käenmäki, T.~Orponen, and L.~Venieri.
\newblock A {M}arstrand-type restricted projection theorem in $\mathbb{R}^{3}$.
\newblock {\em Amer. J. Math. (to appear)}, 2022.
\newblock arXiv:2207.13844.

\bibitem[Mar54]{Marstrand}
J.~M. Marstrand.
\newblock Some fundamental geometrical properties of plane sets of fractional dimensions.
\newblock {\em Proceedings of the London Mathematical Society}, s3-4(1):257--302, 1954.

\bibitem[Mat75]{Mattila_Marstrand}
P.~Mattila.
\newblock Hausdorff dimension, orthogonal projections and intersections with planes.
\newblock {\em Ann. Acad. Sci. Fenn. Ser. AI Math}, 1(2):227--244, 1975.

\bibitem[Mat15]{Mattilabook}
P.~Mattila.
\newblock {\em Fourier analysis and {H}ausdorff dimension}, volume 150 of {\em Cambridge Studies in Advanced Mathematics}.
\newblock Cambridge University Press, Cambridge, 2015.

\bibitem[OO15]{OberlinOberlin}
D.~Oberlin and R.~Oberlin.
\newblock Application of a {F}ourier restriction theorem to certain families of projections in {$\mathbb{R}^3$}.
\newblock {\em J. Geom. Anal.}, 25(3):1476--1491, 2015.

\bibitem[OV18]{OrponenVenieri}
T.~Orponen and L.~Venieri.
\newblock {Improved Bounds for Restricted Families of Projections to Planes in $\mathbb{R}^3$}.
\newblock {\em International Mathematics Research Notices}, 2020(19):5797--5813, 08 2018.

\bibitem[PS07]{PS2007}
M.~Pramanik and A.~Seeger.
\newblock {$L^p$} regularity of averages over curves and bounds for associated maximal operators.
\newblock {\em Amer. J. Math.}, 129(1):61--103, 2007.

\bibitem[PYZ22]{PYZ}
M.~Pramanik, T.~Yang, and J.~Zahl.
\newblock A {F}urstenberg-type problem for circles, and a {K}aufman-type restricted projection theorem in $\mathbb{R}^3$.
\newblock {\em Amer. J. Math. (to appear)}, 2022.
\newblock arXiv:2207.02259.

\bibitem[RW23]{RenWang}
K.~Ren and H.~Wang.
\newblock Furstenberg sets estimate in the plane, 2023.

\bibitem[Tao19]{inductiononscales}
T.~Tao.
\newblock Abstracting induction on scales arguments.
\newblock {\em Terry Tao's Blog https://terrytao.wordpress.com/2019/06/14/abstracting-induction-on-scales-arguments/}, 2019.

\bibitem[Wol00]{Wolff2000}
T.~Wolff.
\newblock Local smoothing type estimates on {$L^p$} for large {$p$}.
\newblock {\em Geom. Funct. Anal.}, 10(5):1237--1288, 2000.

\bibitem[Yan21]{Yang_thesis}
T.~Yang.
\newblock {\em Configurations and decoupling: a few problems in Euclidean harmonic analysis}.
\newblock PhD thesis, University of British Columbia, 2021.

\end{thebibliography}

\end{document}